\documentclass[11pt]{article}
\input{macros}
\usepackage{mathrsfs}
\usepackage{todonotes}

\newcommand{\modif}[1]{#1}
\newcommand{\smodif}[1]{#1}

\begin{document}

\title{Which exceptional low-dimensional projections of a\\ Gaussian point
cloud
can be found in polynomial time?}
\author{Andrea Montanari\thanks{Department of Mathematics and Department of Statistics, 
		Stanford University} \;\; and \;\; Kangjie Zhou\thanks{Department of Statistics, 
		Stanford University}}
\date{\today}
\maketitle

\begin{abstract}
Given $d$-dimensional standard Gaussian vectors   $\xx_1,\dots,\xx_n$,
we consider the set of all empirical distributions of its $m$-dimensional 
projections, for $m$ a fixed constant. Diaconis and Freedman \cite{diaconis1984asymptotics} proved that, 
if $n/d\to \infty$, all such distributions converge to the
standard Gaussian distribution.
In contrast, we study the proportional asymptotics, whereby $n,d\to \infty$ 
with $n/d\to \alpha \in (0, \infty)$. In this case, 
the projection of the data points along a typical random subspace is again Gaussian, 
but the set $\cuF_{m,\alpha}$ of all
probability distributions that are asymptotically feasible as
$m$-dimensional projections contains non-Gaussian distributions
corresponding to exceptional subspaces.

Non-rigorous methods from statistical physics yield an indirect characterization
of $\cuF_{m,\alpha}$ in terms of a generalized Parisi formula. 
Motivated by the goal of putting this formula on a rigorous basis,
and to understand whether these projections can be found efficiently, 
we study the subset $\cuF^{\salg}_{m,\alpha}\subseteq \cuF_{m,\alpha}$ of distributions that
can be realized by a class of iterative algorithms. We prove that this set is characterized by a certain stochastic optimal control problem,
and obtain a dual characterization of this problem in terms of a 
variational principle that extends  Parisi's formula.

As a by-product, we obtain computationally achievable values for a class
of random optimization problems including `generalized spherical perceptron'
models.
\end{abstract}

\tableofcontents

\section{Introduction}
Let $(\xx_i)_{i\le n} \sim_{\iid} \normal(\bzero, \id_d)$
be independent standard Gaussian vectors and denote by
$\XX \in \R^{n \times d}$ the matrix with rows 
$\xx_i^\top$ for $i \in [n]$. We are interested in characterizing the set of low-dimensional 
empirical distributions of this ``Gaussian cloud'',
in the proportional asymptotics
whereby $n, d \to \infty$ with $n/d \to \alpha \in (0, \infty)$.
Namely, fixing $m\ge 1$, we define\footnote{Here and below $\cuP(\R^m)$
	denotes the set of probability distributions on $\R^m$, \modif{and $\stackrel{w}{\Rightarrow}$ denotes weak convergence of probability measures.}}:
\begin{equation}\label{eq:FeasibleFirst}
	\begin{split}
		\cuF_{m,\alpha}:= \Big\{P \in \cuP (\R^{m}):\, & \exists
		\WW = \WW_n (\XX,\omega), \mbox{ s.t. } \WW^\top \WW = I_m, \,\, \\
		& \frac{1}{n} \sum_{i=1}^{n} \delta_{\WW^\top \xx_i} \stackrel{w}{\Rightarrow} P\, \mbox{ in probability}
		\Big\}\, ,
	\end{split}
\end{equation}
where $\omega$ represents some additional randomness independent of $\XX$ (without loss of generality, we can take $\omega$ to be uniformly random on $[0,1]$).
In words, this is the set of probability distributions on $\R^m$ that can be approximated by the empirical distribution of the projections $\{ \WW^\top \xx_i \}_{i \le n}$ for $\WW$ any $d \times m$ orthogonal matrix. \modif{Probability measures in $\cuF_{m, \alpha}$ are referred to as $(\alpha, m)$-feasible distributions.} By general arguments (cf. \cite[Lemma E.8]{montanari2022overparametrized}),
the set $\cuF_{m,\alpha}$ is closed under weak convergence.

This feasible set $\cuF_{m,\alpha}$ was first studied in theoretical 
statistics as a null model for projection 
pursuit \cite{friedman1974projection,friedman1987exploratory}. In particular,
Diaconis and Freedman \cite{diaconis1984asymptotics} established 
that\footnote{Strictly speaking, their result 
	applies to $n,d\to\infty$, with $n/d\to\infty$ at any rate,
	but the treatment given in their paper can be adapted to yield the claimed limit here.}
\begin{align*}
	\lim_{\alpha\to\infty}
	\sup_{P\in\cuF_{m,\alpha}} d_{\sKS}\big(P,\normal(0,1)\big) = 0\, ,
\end{align*}
with $d_{\sKS}$ denoting the Kolmogorov-Smirnov distance.
Later, Bickel, Kur and Nadler \cite{bickel2018projection} first attempted to 
characterize the feasible set $\cuF_{m, \alpha}$ under the proportional 
limit. They obtained certain upper and lower bounds in terms of the second moment of 
the target distribution, as well as the Kolmogorov-Smirnov distance between the target distribution and standard Gaussian measure. 
Tighter inner and outer bounds on $\cuF_{m, \alpha}$ 
(and generalizations of this set) were 
established in the recent paper \cite{montanari2022overparametrized}, together 
with applications to  supervised learning problems. 

The structure of the feasible set $\cuF_{m,\alpha}$ is directly related
to the asymptotics of random optimization problems of the form 
\begin{align}\label{eq:GeneralOpt}
	\mbox{maximize} \,\, H_{n, d} \left( \WW \right) := \, \frac{1}{n} \sum_{i=1}^{n} h \left( \WW^\top \xx_i \right), \quad \mbox{subject to} \ \WW \in O(d, m),
\end{align}
where $h$ is bounded and continuous, and we denote by $O(d,m)$ the set of $d\times m$ orthogonal matrices.

A simple and yet not fully understood example of
the type of random optimization problems \eqref{eq:GeneralOpt} is provided by the so-called spherical
perceptron problem.
Given data $\{ (y_i,\oxx_i) \}_{i \le n} \sim_{\iid}\Unif(\{+1,-1\})\otimes \normal(\bzero,\id_d)$
and a parameter $\kappa\in\R$, we would like to
find a vector $\ww\in\R^d$, $\|\ww\|_2=1$
such that $y_i\<\ww,\oxx_i\>\ge \kappa$ for all $i\le n$.
This is known in machine learning as a linear classifier with margin
$\kappa$ for the data $\{ (y_i,\oxx_i) \}_{i \le n}$, see \cite{shalev2014understanding} for further background. 

For $\kappa\ge 0$, an explicit threshold $\alpha_*(\kappa)$ is known such that a $\kappa$-margin classifier 
$\ww$ exists with high probability  if $n/d\to\alpha<\alpha_*(\kappa)$, and 
does not exist if $n/d\to\alpha>\alpha_*(\kappa)$ \cite{gardner1988space,shcherbina2003rigorous,stojnic2013another}. However, 
such a phase transition  has not been established for
$\kappa<0$.
Franz and Parisi used non-rigorous spin glass techniques to derive a conjectured threshold in \cite{franz2016simplest,franz2019critical}. 
They further speculated that the structure of near optima of this problem is related to dense
`disordered' sphere packings in high  dimensions \cite{parisi2020theory}. 
Upper and lower bounds $\alpha_{\sUB}(\kappa)$, $\alpha_{\sLB}(\kappa)$ on the phase transition threshold 
of the spherical perceptron for $\kappa<0$, as well as efficient algorithms to find a solution
were recently studied in \cite{montanari2024tractability, el2022algorithmic}.

The negative spherical perceptron problem can be rephrased as a question about the value of an optimization problem of the form  
\eqref{eq:GeneralOpt}, with $m=1$. Defining $\xx_i=y_i\oxx_i$ and taking $h_{\kappa}(t) = \min(t - \kappa, 0)$, we are led to consider the optimization problem
\begin{align*}
	\mbox{maximize} \,\, \frac{1}{n} \sum_{i=1}^{n} h_{\kappa} \big(\<\ww, \xx_i\>\big), \quad \mbox{subject to } \ww \in \S^{d-1},
\end{align*}
where $\S^{d-1}$ denotes the $d$-dimensional unit sphere.
A $\kappa$-margin solution exists if and only if the value of this problem is zero. 

Proposition 4.1 in \cite{montanari2022overparametrized} implies that for any $h \in C_b (\R^m)$ (the set of all bounded continuous functions on $\R^m$):
\begin{equation}\label{eq:feasibility_test_fcn}
	\pliminf_{n, d \rightarrow \infty} \max_{\WW \in O(d, m)} H_{n,d} \left( \WW \right) = \, \sup_{P \in \cuF_{m, \alpha}} \left\{ \int_{\R^m} h(z) P(\d z) \right\} =:
	\VH_{m,\alpha}(h).
\end{equation}
Therefore, characterizing the feasible set $\cuF_{m,\alpha}$
would allow us to determine $\VH_{m,\alpha}(h)$, the asymptotics of the global maximum for all problems of the form 
\eqref{eq:GeneralOpt}.

Vice versa, determining $\VH_{m,\alpha}(h)$ for all
$h \in C_b (\R^m)$ provides a complete characterization of the closed convex hull of $\cuF_{m,\alpha}$, \modif{which is the closure of the set containing all convex combinations of finitely many elements in $\cuF_{m,\alpha}$}. This claim is a direct consequence of the following duality theorem, whose proof is deferred to \cref{sec:ProofDuality}.
\begin{thm}\label{thm:dual_char_prob}
	Denote by $\cuP (\R^m)$ the set of all probability distributions on $\R^m$. Assume $E \subset \cuP (\R^m)$ is convex and closed under weak limit. Then, for any $\mu \in \cuP (\R^m)$, $\mu \in E$ if and only if for any $h \in C_b (\R^m)$,
	\begin{equation*}
		\int_{\R^m} h \, \d \mu \le \sup_{\nu \in E} \left\{ \int_{\R^m} h \,
		\d \nu \right\}.
	\end{equation*}
\end{thm}
Throughout the paper, we will often 
move between $\cuF_{m,\alpha}$ and its dual $\VH_{m,\alpha}(\,\cdot\, )$, which is a functional on $C_b(\R^m)$.

The main goal of this paper is to provide a Parisi-type formula for 
a subset of  $\cuF_{m,\alpha}$ that can be realized via polynomial-time
computable projections, namely\footnote{\modif{Similar to \cite[Lemma E.8]{montanari2022overparametrized}, we know that $\cuF^{\salg}_{m,\alpha}$ is closed under weak limits.}}
\begin{equation*}
	\begin{split}
		\cuF^{\salg}_{m,\alpha} := \Big\{P \in \cuP (\R^{m}):\, & \exists
		\WW = \WW_n (\XX,\omega) \mbox{ polytime computable, s.t. } \\
		&\WW^\top \WW = I_m, \,
		\frac{1}{n} \sum_{i=1}^{n} \delta_{\WW^\top \xx_i} \stackrel{w}{\Rightarrow} P\, \mbox{ in probability}
		\Big\}\, .
	\end{split}
\end{equation*}
More explicitly, $\WW_n (\XX,\omega)$ is `polytime computable' means that there exists an algorithm, accepting $(\XX,\omega)$ (or its finite-precision approximation) as input and computing
$\WW_n (\XX,\omega)$  in time polynomial in $n, d$. In what follows, we will describe a class of efficient algorithms for computing $\WW_n (\XX,\omega)$, 
and characterize the resulting set of computationally feasible distributions, thus establishing
an inner bound on $\cuF^{\salg}_{m,\alpha}$. These
algorithms are based on the incremental approximate message passing (IAMP) algorithms 
that have been recently developed to optimize the Hamiltonians of mean-field spin glasses
\cite{montanari2021optimization,el2021optimization}.
Recent work by Huang and Sellke \cite{huang2022tight,huang2024optimization}
proves that ---in the spin glass context--- IAMP algorithms are optimal within the broader class of
Lipschitz algorithms. This provides rigorous evidence for the expectation that the class of IAMP algorithms
characterize the fundamental computational limit of the random optimization problem~\eqref{eq:GeneralOpt}, and related ones.

The main contributions of this paper are as follows.

Section \ref{sec:replica_method} provides further background, by  deriving a general prediction for $\VH_{m,\alpha}(h)$
using non-rigorous techniques from statistical physics. The resulting prediction
takes the form of a generalized Parisi formula. This conjecture is
a useful benchmark as well as a motivation for our theory.

In Section \ref{sec:main_results}, we present our main results, namely:
\begin{enumerate}
	\item We characterize a set of probability distributions that can be realized 
	via $m$-dimensional polynomial-time computable projections  
	(Theorem \ref{thm:general_amp_achievability}), which provides an inner
	bound for $\cuF^{\salg}_{m,\alpha}$.
	
	As mentioned above, our inner bound is based on computing the projection matrix $\WW = \WW_n (\XX,\omega)$ via an IAMP
	algorithm, and we denote the resulting set of distributions by $\cuF^{\sAMP}_{m,\alpha}\subseteq \cuF^{\salg}_{m,\alpha}$.
	We show that the probability distributions in  $\cuF^{\sAMP}_{m,\alpha}$ can be
	represented as the laws of a certain class of stochastic integrals.
	\item Using this stochastic integral representation,
	it is immediate to derive a lower bound on
	$$
	\VH^{\salg}_{m,\alpha}(h) := \sup_{P \in \cuF_{m,\alpha}^{\mbox{\tiny\rm alg}}} \left\{ \int_{\R^m} h (z)\, P(\d z) \right\}
	$$
	for general $h \in C_b (\R^m)$.
	We will denote this lower bound by $\VH^{\sAMP}_{m,\alpha}(h)$ (Theorem \ref{thm:ValH}).
	In particular, for any $\eps>0$, there exists an IAMP algorithm 
	returning $\widehat{\WW}_n^{\sAMP}$, such that
	$$
	H_{n,d} \left( \widehat{\WW}_n^{\sAMP} \right) \ge \, \VH_{m,\alpha}^{\sAMP}(h)-\eps
	$$
	with probability converging to $1$ as $n,d\to\infty$, with $n/d\to\alpha$.
	
	\modif{By analogy with previous results in spin glass theory,
		we expect $\VH^{\sAMP}_{m,\alpha}(h)= \VH^{\salg}_{m,\alpha}(h)$ in general. Further, recalling that
		$\VH_{m,\alpha}(h)$ denote the analogous quantity where the supremum is taken over 
		$P \in \cuF_{m,\alpha}$, we expect $\VH^{\sAMP}_{m,\alpha}(h)= \VH^{\salg}_{m,\alpha}(h)= \VH_{m,\alpha}(h)$
		for `directions' $h$ that satisfy a no overlap gap condition (see Theorem \ref{thm:tight_under_no-ogp}). } 
	\item For a fixed $h$, computing the value of $\VH_{m,\alpha}^{\sAMP}(h)$ is equivalent to solving a stochastic
	optimal control problem, as described in \cref{sec:VH_statement}. When $m=1$, we use a duality argument to derive a Parisi-type formula
	for $\VH^{\sAMP}_{1,\alpha}(h)$, which takes the form of a variational principle over a suitable function space
	(Theorem \ref{thm:two_stage_strong_duality}). 
	This variational principle
	turns out to be closely related to the conjectured formula for  $\VH_{1,\alpha}(h)$,
	which we derive using the replica method in Section \ref{sec:replica_method}.
\end{enumerate}

For the first two results, we generalize techniques developed in the context of 
spin glasses in  \cite{montanari2021optimization,el2021optimization}. However, the third point
(deriving a Parisi-type formula via duality) poses significantly new challenges. 
The approach used in previous work was to establish a Hamilton-Jacobi-Bellman (HJB) equation for the value of the stochastic
optimal control problem, and then show that the latter is equivalent to the Parisi-type PDE via Legendre-Fenchel duality.
While we follow a similar route, we need to consider a more general class of optimal control problems,
corresponding to more general initializations for the HJB equation. As a consequence, we lack
a priori convexity estimates on the solution of the HJB equation and establishing that the latter is indeed well-posed
requires a novel proof. 

In particular, we will show that for any test function $h$ that satisfies certain regularity conditions and general function order parameter $\mu$ (cf. \cref{defn:gamma_space}), the Parisi PDE associated with the replica prediction of the optimization problem \eqref{eq:GeneralOpt} admits a unique weak solution. Our assumptions are very weak and almost minimal in order for the Parisi functional to be well-defined. As a comparison, earlier work on IAMP algorithms for the negative spherical perceptron \cite{el2022algorithmic} assumes that $h$ is concave and has some specific structure, and that $\mu$ is non-negative and non-decreasing.

Proofs of our main results are presented
in Section \ref{sec:iamp} and Section \ref{sec:char_opt_ctrl}. 
Section \ref{sec:iamp} describes the IAMP algorithm and presents the proof of our general feasibility theorem, i.e.,
Theorem \ref{thm:general_amp_achievability}. Section \ref{sec:char_opt_ctrl}
presents the proof of the Parisi-type formula
(Theorem \ref{thm:two_stage_strong_duality}), \smodif{building on the technical preliminaries developed in \cref{sec:solve_Parisi_PDE,sec:verification_argument,sec:variational_problem}.}
\modif{In these subsequent sections, we construct a solution to the Parisi PDE and establish its regularity, prove a dual relationship between the value of our stochastic optimal control problem and the solution to the Parisi PDE, and analyze the corresponding Parisi variational problem.}

\subsection*{Concurrent work}
After this work was first presented (as part of the Ph.D. thesis of the second author), we became aware that 
partially overlapping results have been obtained independently by Brice Huang, Mark Sellke, and Nike Sun \cite{huang2024preparation}.
The authors also consider the Ising case and obtain hardness results for Lipschitz algorithms.

\subsection*{Notations}

We will follow the convention of using boldface letters for matrices or vectors whose
dimensions diverge as $n,d\to\infty$, and normal fonts otherwise.
We denote the standard scalar product between two vectors $\uu,\vv$ by $\<\uu,\vv\>$,
and the matrix scalar product by $\<\AA,\BB\> = \Tr(\AA^\top\BB)$. We use $\norm{\cdot}_2$ to denote the Euclidean norm of a vector. We use $\norm{\cdot}_{L^p}$ to denote the standard $L^p$ norm of a function for $p \in [1, +\infty]$.

We denote by $\S_+^m$ the convex cone of $m\times m$ positive semi-definite matrices. For $d \ge m$, we denote by $O(d, m)$ the set of all $d \times m$ orthogonal matrices. For a subset $S$ in a topological space, we denote by $\close S$ its closure.
For $p \ge 1$, we denote by $C^k (\R^p)$ the collection of all functions that have continuous $k$-th derivatives in $\R^p$. We also denote by $C_b (\R^p)$ the set of all bounded continuous functions on $\R^p$, and by $C_{c}^{\infty} (\R^p)$ the set of all infinitely differentiable functions with compact supports. We use $\cuP (\R^p)$ to denote the set of all probability measures on $\R^p$ equipped with the topology of weak convergence, unless otherwise stated. For a sequence of probability measures $(\nu_n)_{n \ge 1}$ and a probability measure $\nu$, we write $\nu_n \stackrel{w}{\Rightarrow} \nu$ if $\nu_n$ weakly converges to $\nu$. We use $\Law (X)$ to denote the law of a random variable $X$. For two random variables $X, Y$, we write $X \perp Y$ if $X$ is independent of $Y$.

For a function $h$, we denote by $\operatorname{conc} (h)$ the (upper) concave envelope of $h$. Namely, $\operatorname{conc} (h)$ is the pointwise minimum of all concave functions that dominate $h$. For $l, k \ge 1$, and a differentiable mapping $F: \R^k \to \R^l$, we denote by $J_F (x)$, or $ \partial F / \partial x \in \R^{l \times k}$ the Jacobian matrix of $F$, namely for $x \in \R^k$: $J_F(x)_{ij} = \partial F_i / \partial x_j$. We occasionally use $J_F$ as a shorthand for $J_F (x)$ whenever the variable $x$ is clear from the context. We say that a function $\psi:\R^{p}\to\R$ is pseudo-Lipschitz (of order $2$) if there exists a constant $C$ such that, 
for all $x,y\in \R^p$, 
$$
|\psi(x)-\psi(y)|\le C(1+\|x\|_2+\|y\|_2)\|x-y\|_2.
$$
Let $\{ B_t \}_{t \in [0, 1]}$ be an $m$-dimensional standard Brownian motion, and let $\{ \cF_t \}_{t \in [0, 1]}$ be its canonical filtration. For $s \le t$, we denote by $D[s, t]$ the space of all admissible controls on the interval $[s, t]$, i.e., the collection of all progressively measurable processes $\{ \Phi_r \}_{s \le r \le t}$ satisfying
\begin{equation}
	\sigma(\Phi_r) \subset \cF_r, \ \forall r \in [s, t], \ \mbox{and} \ \E \left[ \int_{s}^{t} \Phi_r \Phi_r^\top \d r \right] < \infty.
	\label{eq:Ddef}
\end{equation}

%
%
\section{Conjectures from statistical physics}\label{sec:replica_method}

This section will be devoted to the prediction of the feasible set $\cuF_{m, \alpha}$ using physicists' replica method. 
Based on the duality between $\cuF_{m, \alpha}$ and $\VH_{m,\alpha}(\cdot)$, we will state the general prediction for $\VH_{m,\alpha}(h)$ in the following conjecture, with detailed calculations deferred to Appendix~\ref{append:replica}. Recall that $\VH_{m,\alpha}(h)$ is defined in Eq.~\eqref{eq:feasibility_test_fcn}.

\begin{conj}[Replica prediction for $\VH_{m,\alpha}(h)$]\label{conj:Parisi_formula_mdim}
	\smodif{For any fixed $h \in C_b (\R^m)$, we have almost surely
		\begin{equation*}
			\lim_{n \to \infty} \max_{\WW \in O(d, m)} \frac{1}{n} \sum_{i=1}^{n} h \left( \WW^\top \xx_i \right) = \VH_{m,\alpha}(h)\,.
		\end{equation*}
		Further, $\VH_{m,\alpha}(h)$ has the following variational representation:
		\begin{equation*}
			\VH_{m,\alpha}(h) = \inf_{(\mu, M, C) \in \mathscr{U} \times \mathscr{I}_m \times \S_+^m} \mathsf{F}_m (\mu, M, C) \,.
	\end{equation*}}
	In the above display, the $m$-dimensional Parisi functional $\mathsf{F}_m: \mathscr{U} \times \mathscr{I}_m \times \S_+^m \to \R$ is defined as
	\begin{equation*}
		\mathsf{F}_m (\mu, M, C) = f_{\mu} (0, 0) + \frac{1}{2 \alpha} \int_{0}^{1} \Tr \left( M(t) \left( C + \int_{t}^{1} \mu(s) M(s) \d s \right)^{-1} \right) \d t,
	\end{equation*}
	where $f_{\mu}: [0, 1] \times \R^m \to \R$ solves the $m$-dimensional Parisi PDE:
	\begin{equation}\label{eq:Parisi_mdim}
		\begin{split}
			\partial_t f_{\mu} (t, x) + \, & \frac{1}{2} \mu (t)\<\nabla_{x} f_{\mu} (t, x), M(t) \nabla_{x} f_{\mu} (t, x)\> +\frac{1}{2} \operatorname{Tr}\left( M(t) \nabla_{x}^2 f_{\mu} (t, x) \right) = \, 0, \\
			f_{\mu} (1, x) = \, & \sup_{u \in \R^m} \left\{ h(x + u) - \frac{1}{2} \<u, C^{-1}u\> \right\}\, ,
		\end{split}	
	\end{equation}
	$\S_+^m$ denotes the set of $m \times m$ positive definite matrices, and
	\begin{align*}
		\mathscr{I}_m &:= \left\{ M: [0, 1) \to \S_+^m : \int_{0}^{1} M(t) \,\d t = I_m \right\},\\
		\mathscr{U} &:= \left\{ \mu: [0, 1) \to \R_{\ge 0}: \ \mu \ \mbox{{\rm non-decreasing}}, \ \int_{0}^{1} \mu(t) \d t < \infty \right\}.
	\end{align*}
\end{conj}

\begin{rem}[Replica prediction for $\VH_{1,\alpha}(h)$]\label{rmk:ReplicaM1}
	The formulas in Conjecture \ref{conj:Parisi_formula_mdim} can be significantly simplified for
	the case $m=1$. In this case, $M(t) = r(t)$ is a non-negative function on $[0, 1]$ satisfying $\int_{0}^{1} r(t) \d t = 1$, and $C = c$ is a positive scalar. We then have
	\begin{equation*}
		\mathsf{F}_1 (\mu, r, c) = f_{\mu} (0, 0) + \frac{1}{2 \alpha} \int_{0}^{1} \frac{r(t) \d t}{c + \int_{t}^{1} \mu(s) r(s) \d s },
	\end{equation*}
	where $f_{\mu}$ solves the PDE
	\begin{equation}\label{eq:Parisi_m1_first}
		\begin{split}
			&\partial_s f_{\mu} (s, x)+\frac{1}{2} r(s) \left( \mu (s) \partial_x f_{\mu} (s, x)^2 + \partial_x^2 f_{\mu} (s, x) \right) = \, 0, \\
			&	f_{\mu} (1, x) = \,  \sup_{u \in \R} \left\{ h( x + u) - \frac{u^2}{2c}  \right\}.
		\end{split}	
	\end{equation}
	Examples of this formula in the literature sometimes use a different parametrization
	of the time variable.
	Namely, they use the change of variable $s\mapsto t(s) = \int_{0}^{s} r(u) \d u$.
	Under this change of variable, we recast $f_{\mu} (s, x)$ as $f_{\mu} (t, x)$ and $\mu(s)$ as $\mu(t)$. The Parisi PDE then reads
	\begin{equation}
		\begin{split}
			&\partial_t f_{\mu} (t, x)+\frac{1}{2} \mu (t) \partial_x f_{\mu} (t, x)^2 +  \frac{1}{2} \partial_x^2 f_{\mu} (t, x) = \, 0, \\
			&f_{\mu} (1, x) = \,  \sup_{u \in \R} \left\{ h( x + u) - \frac{u^2}{2c}  \right\}.
		\end{split}	\label{eq:ParisiEq-m1}
	\end{equation}
	Further, the Parisi functional reduces to
	\begin{equation}
		\mathsf{F}_1 (\mu, c) =  f_{\mu} (0, 0) + \frac{1}{2 \alpha} \int_{0}^{1} \frac{\d q}{c + \int_{q}^{1} \mu(u) \d u }\, .\label{eq:ParisiF-m1}
	\end{equation}
	Notice that the simplified Parisi functional does not depend on $r$ any more. The replica prediction for $\VH_{1,\alpha}(h)$ then becomes
	\begin{equation}\label{eq:ReplicaPred_m1}
		\lim_{n \to \infty} \max_{\ww \in \S^{d-1}} \frac{1}{n} \sum_{i=1}^{n} h \left( \langle \ww, \xx_i \rangle \right) = \VH_{1,\alpha}(h) = \inf_{(\mu, c) \in \mathscr{U} \times \R_{> 0}} \mathsf{F}_1 (\mu, c).
	\end{equation}
	In the following sections, we will drop the subscript ``$1$'' and use $\mathsf{F} (\mu, c)$ instead of 
	$\mathsf{F}_1 (\mu, c)$.
\end{rem}

%
%
\section{Main results}\label{sec:main_results}

In this section, we present our main results regarding $\cuF_{m, \alpha}^{\salg}$, the set of  computationally
feasible probability distributions in $\cuF_{m, \alpha}$. 
In Section \ref{sec:IAMP}, we describe the class of two-stage AMP algorithms to be analyzed.
Section \ref{sec:FeasibleIAMP} presents the characterization of the set of
probability distributions in $\cuF_{m, \alpha}^{\salg}$ that can be
realized using our algorithm. Section \ref{sec:VH_statement} then states our main achievability result for $\VH^{\salg}_{m,\alpha}(\cdot)$.
Finally, in Section~\ref{sec:Parisi_formula} we establish the extended Parisi variational principle for
$\VH^{\salg}_{1,\alpha}(\cdot)$.
%
%
\subsection{Overview of the algorithm}
\label{sec:IAMP}
We first give a brief description of our two-stage AMP algorithm. For further
background information and discussion, we refer to Section~\ref{sec:iamp}.

Our algorithm has two stages: The first stage consists of $T_1$ identical iterations, followed by an incremental stage of
$T_2$ iterations with small weights, where $T_1$ and $T_2$ are two positive integers to be determined.

For $t = 1, \cdots, T_1$, we update $\VV^t\in\R^{n\times m}$, $\WW^t\in\R^{d\times m}$ according to
\begin{align*}
	\WW^{t + 1} & = \frac{1}{\sqrt{n}} \XX^\top F (\VV^{t}) - \WW^{t} K_{t}^\top, \\
	\VV^{t} & = \frac{1}{\sqrt{n}} \XX \WW^{t} - \frac{d}{n} F (\VV^{t-1})\, ,
\end{align*}
with $\WW^1 = \XX^\top F_0 / \sqrt{n}$ \smodif{and $F_0 \in \R^{n \times m}$ independent of $\XX$.} Here $F:\R^m \to \R^m$ is understood to be applied row-wise.
Namely, for $\VV\in\R^{n\times m}$ (with rows $\vv_i$), 
$F (\VV)\in\R^{n\times m}$ is the matrix whose $i$-th row is $F(\vv_i)$.
Further, the Onsager correction term $K_t\in\R^{m\times m}$ is defined as
\begin{equation*}
	K_{t} = \frac{1}{n} \sum_{i=1}^{n} \frac{\partial F}{\partial \vv_i^t} (\vv_i^t),
\end{equation*}
\smodif{where $\vv_i^t$ is the $i$-th row of $\VV^t$, and $\frac{\partial F}{\partial \vv_i^t} (\vv_i^t)$ denotes the Jacobian matrix of $F$ with respect to $\vv_i^t$.}
We will show (using general tools from the analysis of AMP algorithms) 
that, in the limit of large $T_1$ after $n,d\to \infty$,
this iteration converges to an approximate fixed point of some non-linear function,
which will be the starting point of the second stage of our algorithm.

In the second stage, we allow each iterate to depend on all previous ones. Denote $\WW^{\le t} = (\WW^s)_{1 \le s \le t}$ and $\VV^{\le t} = (\VV^s)_{1 \le s \le t}$. We iterate, for $T_1 + 1 \le t \le T_1+T_2$:
\begin{equation}\label{eq:amp_iter_2nd}
	\begin{split}
		\WW^{t + 1} & = \frac{1}{\sqrt{n}} \XX^\top F_t (\VV^{\le t}) - \sum_{s=1}^{t} G_s (\WW^{\le s}) K_{t, s}^\top, \\
		\VV^{t} & = \frac{1}{\sqrt{n}} \XX G_t(\WW^{\le t}) - \sum_{s=1}^{t} F_{s-1} (\VV^{\le s-1}) D_{t, s}^\top,
	\end{split}
\end{equation}
where
\begin{equation*}
	K_{t, s} = \frac{1}{n} \sum_{i=1}^{n} \frac{\partial F_t}{\partial \vv_i^s} \left( \vv_i^{1}, \cdots, \vv_i^t \right), \ D_{t, s} = \frac{1}{n} \sum_{i=1}^{d} \frac{\partial G_t}{\partial \ww_i^s} \left( \ww_i^{1}, \cdots, \ww_i^t \right), \quad t \ge s.
\end{equation*}
As before, we will overload the notations and let $F_t$ and $G_t$ operate on its argument matrices row-wise. 
We further assume that $F_t$ and $G_t$ have the following specific structure:
\begin{align*}
	F_t \left( \vv^{1}, \cdots, \vv^{t} \right) = \, & \vv^t \Phi_{t-1} \left( \vv^{1}, \cdots, \vv^{t-1} \right), \\
	G_t \left( \ww^{1}, \cdots, \ww^{t} \right) = \, & \ww^t \Psi_{t-1} \left( \ww^{1}, \cdots, \ww^{t-1} \right),
\end{align*}
where $\Phi_{t-1}$ and $\Psi_{t-1}$ 
are matrix-valued \smodif{mappings} that satisfy certain moment constraints from the state evolution of AMP. The second stage of our algorithm involves $T_2$ 
iterations with the above choices of $F_t$ and $G_t$.

Finally, the output of our two-stage AMP algorithm will be a linear combination of $\WW^{T_1}$ and the incremental AMP iterations in the second stage. To be concrete, we will show that 
\begin{equation*}
	\plim_{n,d\to\infty} \frac{1}{n} (\WW^{T_1})^{\top} \WW^{T_1} = Q,
\end{equation*}
where $Q \in \S_+^m$ is a deterministic $m\times m$ matrix satisfying $0 \preceq Q \preceq I_m$, which will be characterized in \cref{sec:fixed_pt_amp}. For any such $Q$, let $Q_{1}, \cdots, Q_{T_2}$ be $T_2$ deterministic $m\times m$ matrices such that 
$$
\sum_{t=1}^{T_2} Q_t^\top Q_t = I_m - Q,
$$
we then compute
\begin{equation*}
	\WW_Q = \frac{1}{\sqrt{n}}  \WW^{T_1} + \frac{1}{\sqrt{n}} \sum_{t=1}^{T_2} G_{T_1 + t+1} \left( \WW^{\le T_1 + t + 1} \right) Q_t\, .
\end{equation*}
We set the final output of our algorithm to be $\widehat{\WW}_n^{\sAMP} = \WW_Q (\WW_Q^\top \WW_Q)^{-1/2}$, which is guaranteed to be a $d \times m$ orthogonal matrix. The set of $(\alpha, m)$-feasible distributions achievable by our two-stage AMP algorithm will be studied in the next section.
%
%
\subsection{A set of computationally feasible distributions}
\label{sec:FeasibleIAMP}

We begin with stating our AMP achievability result for general $m$, and then simplify our formulas to the special case $m=1$.
\begin{defn}\label{defn:Q_contraction}
	Let $Q \in \S_+^m$ be such that $0 \preceq Q \preceq I_m$, and $V \sim \normal(0, Q)$. We say that a differentiable function $F: \R^{m} \to \R^m$ is a $Q$-contraction, if
	\begin{equation*}
		\E \left[ F \left( V \right) F \left( V \right)^\top \right] = \, \alpha Q,
	\end{equation*}
	and there exists some $S \in \S_+^m$, $S \succ 0$, such that
	\begin{equation}\label{eq:Q_contraction}
		\E \left[ \JF \left( V \right)^\top S\, \JF \left( V \right) \right] \preceq \, \alpha S,
	\end{equation}
	where $\JF$ denotes the Jacobian matrix of $F$.
\end{defn}
We are now in position to state our main achievability results for AMP algorithms.
\begin{thm}[Inner bound for $\cuF_{m, \alpha}^{\salg}$]\label{thm:general_amp_achievability}
	Let $Q \in \S_+^m$ be such that $0 \preceq Q \preceq I_m$, and $V \sim \normal(0, Q)$.
	Assume that $F$ is a $Q$-contraction as per \cref{defn:Q_contraction}. Let $(B_t)_{0 \le t \le 1}$ be an $m$-dimensional standard Brownian motion independent of $V$. Define the filtration $\{ \cF_t \}$ by
	\begin{equation*}
		\cF_t = \sigma \left( V, (B_s)_{0 \le s \le t} \right), \ 0 \le t \le 1.
	\end{equation*}
	Assume $Q(t) \in L^2 ([0, 1] \to \R^{m \times m} )$ satisfies
	\begin{equation*}
		\int_{0}^{1} Q(t) Q(t)^\top \d t = I_m - Q,
	\end{equation*}
	and $\{ \Phi_t \}_{0 \le t \le 1}$ is an $m \times m$ matrix-valued progressively measurable stochastic process with respect to the filtration $\{ \cF_t \}_{0 \le t \le 1}$, satisfying
	\begin{equation*}
		\E \left[ \Phi_t \Phi_t^\top \right] \preceq \frac{I_m}{\alpha}, \, \forall t \in [0, 1].
	\end{equation*}
	Then, we have that $\operatorname{Law} ( U) \in \cuF_{m, \alpha}^{\salg}$, where
	\begin{equation}\label{eq:stochastic_integral_rep}
		U = V + \frac{1}{\alpha} F(V) + \int_{0}^{1} Q(t) \left( I_m + \Phi_t \right) \d B_t.
	\end{equation}
\end{thm}

The proof of Theorem~\ref{thm:general_amp_achievability} is presented in Section~\ref{sec:iamp}, where we also describe in greater details the two-stage AMP algorithm utilized to prove it. For the case $m=1$, we obtain the following feasibility result as a corollary of \cref{thm:general_amp_achievability}.
\begin{cor}[Inner bound for $\cuF_{1, \alpha}^{\salg}$]\label{cor:unsupervised_amp_achievability}
	Let $q\in[0,1]$, $v \sim \normal(0, q)$ and assume that $F:\R\to\R$ is a $q$-contraction as per \cref{defn:Q_contraction}, i.e.,
	\begin{align*}
		\E \left[ F(v)^2 \right] = \alpha q, \quad \E \left[ F'(v)^2 \right] \le \alpha \, .
	\end{align*}
	Define the filtration $\cF_t = \sigma \left( v, (B_s)_{0 \le s \le t} \right)$ for $t \in [0, 1]$,
	with $(B_t)_{t \in [0, 1]}$ a standard Brownian motion independent of $v$. Assume $q(t) \in L^2 [0, 1]$ satisfies
	\begin{equation*}
		\int_{0}^1 q(t)^2 \d t = 1 - q,
	\end{equation*}
	and $(\phi_t)_{0 \le t \le 1}$ is a real-valued progressively measurable process with respect to $\{\cF_t\}_{0 \le t \le 1}$, satisfying
	\begin{equation*}
		\E \left[ \phi_t^2 \right] \le \, \frac{1}{\alpha}, \ \forall t \in [0, 1].
	\end{equation*}
	Then, we have that 
	$\operatorname{Law} (U) \in \cuF_{1, \alpha}^{\salg}$, where
	\begin{equation*}
		U = v + \frac{1}{\alpha} F(v) + \int_{0}^{1} q(t) \left( 1 + \phi_t \right) \d B_t\, .
	\end{equation*}
\end{cor}

%
%
\subsection{Dual value $\VH_{m,\alpha}^{\salg}(h)$ and stochastic optimal control}
\label{sec:VH_statement}

Fix a function $h: \R^m \to \R$ and recall from \eqref{eq:GeneralOpt} the problem of maximizing the Hamiltonian
\begin{equation}\label{eq:GeneralOpt_recall}
	H_{n, d} (\WW) = \, \frac{1}{n} \sum_{i=1}^{n} h \left( \WW^\top \xx_i \right), \quad \WW \in O(d, m) .
\end{equation}
We will characterize the optimal value achieved by AMP algorithms
of the type described in \cref{sec:IAMP} (and detailed in \cref{sec:iamp}). \modif{For this purpose, we let $\cuF_{m, \alpha}^{\sAMP}$ denote the set of all $(\alpha, m)$-feasible distributions of the form \eqref{eq:stochastic_integral_rep}, and define 
	\begin{equation*}
		\VH_{m,\alpha}^{\sAMP} (h) = \, \sup_{P \in \cuF_{m, \alpha}^{\sAMP}} \int_{\R^m} h(z) P (\d z).
	\end{equation*}
}
Equivalently,
\begin{equation}\label{eq:SOC_First}
	\begin{split}
		& \VH_{m,\alpha}^{\sAMP} (h) := \, \sup \, \E \left[ h \left( V + \frac{1}{\alpha} F(V) + \int_{0}^{1} Q(t) \left( I_m + \Phi_t \right) \d B_t \right) \right], \\
		& \mbox{subject to} \,\,\, 0 \preceq Q \preceq I_m, \, \int_{0}^{1} Q(t) Q(t)^\top \d t = I_m - Q, \,\, \text{$F$ is a $Q$-contraction,} \\
		& \quad\quad\quad\quad \,\,\,  \Phi \in D [0, 1] \,\, \text{and} \,\, \E \left[ \Phi_t \Phi_t^\top \right] \preceq \frac{I_m}{\alpha}, \, \forall t \in [0, 1].
	\end{split}
\end{equation}
Note that for any fixed choice of $Q$, $\{ Q(t) \}_{t \in [0, 1]}$ and $F$, \eqref{eq:SOC_First} is a stochastic optimal control problem for the control process $\{ \Phi_t \}_{t \in [0, 1]}$ on the time interval $[0,1]$, \modif{which will be analyzed in greater details in \cref{sec:lagrange_dual}.}

As a direct consequence of \cref{thm:general_amp_achievability}, the following theorem establishes that the asymptotic maximum of $H_{n, d} (\WW)$ achieved by our two-stage AMP algorithm equals $\VH_{m,\alpha}^{\sAMP} (h)$.
\begin{thm}[Optimal value of $H_{n, d} (\WW)$ achieved by AMP algorithms] \label{thm:ValH}
	Assume $h: \R^m \to \R$ is continuous, bounded from above, and of at most linear growth at infinity. Then, the following holds.
	\begin{itemize}
		\item[$(a)$] \emph{Upper bound.}	 Let $\widehat{\WW}^{\sAMP}_n$ be the output of any AMP algorithm as described in 
		\cref{sec:IAMP} (and further elaborated in \cref{sec:iamp}). Then, we have almost surely,
		\begin{equation*}
			\lim_{n \to \infty} H_{n,d}\left(\widehat{\WW}^{\sAMP}_{n}  \right) \le \, \VH_{m,\alpha}^{\sAMP} (h).
		\end{equation*}
		\item[$(b)$] \emph{Achievability.}
		For any $\veps > 0$, there exists a two-stage AMP algorithm that outputs $\widehat{\WW}^{\sAMP}_{n}$, such that almost surely,
		\begin{equation*}
			\lim_{n \to \infty} H_{n,d}\left( \widehat{\WW}^{\sAMP}_{n}  \right) \ge  \VH_{m,\alpha}^{\sAMP} (h) - \veps \, .
		\end{equation*}
	\end{itemize}
\end{thm}
Of course, since AMP is a polynomial-time algorithm, \cref{thm:ValH} implies that
\begin{equation*}
	\VH_{m,\alpha}^{\salg} (h) \ge \VH_{m,\alpha}^{\sAMP} (h)\, .
\end{equation*}

\modif{From now on, we focus on the case $m = 1$, for which the function parameter $q(t)$ in \cref{cor:unsupervised_amp_achievability} and \eqref{eq:SOC_First} can be eliminated via a reparametrization of the time variable. See below for a precise statement and \cref{sec:proof_simple_1dim} for its proof.
	\begin{prop}\label{prop:feasible_simple_1dim}
		When $m = 1$, we have
		\begin{equation}\label{eq:SOC_First_1dim}
			\begin{split}
				& \VH_{1,\alpha}^{\sAMP} (h) := \, \sup \, \E \left[ h \left( v + \frac{1}{\alpha} F(v) + \int_{q}^{1} \left( 1 + \phi_t \right) \d B_t \right) \right], \\
				& \mbox{subject to} \,\, q \in [0, 1], \, \text{$F$ is a $q$-contraction}, \phi \in D [q, 1] \,\, \text{and} \sup_{t \in [q, 1]} \E \left[ \phi_t^2 \right] \le \frac{1}{\alpha}.
			\end{split}
		\end{equation}
	\end{prop}
	For $m = 1$, we also define the maximum of \eqref{eq:GeneralOpt_recall} achievable by our two-stage AMP algorithm with a fixed choice of $q$:
	\begin{equation*}
		\begin{split}
			& \VH_{1,\alpha}^{\sAMP} (q, h) := \, \sup \, \E \left[ h \left( v + \frac{1}{\alpha} F(v) + \int_{q}^{1} \left( 1 + \phi_t \right) \d B_t \right) \right], \\
			& \mbox{subject to $F$ being a $q$-contraction, } \phi \in D[q, 1] \,\, \text{and} \sup_{t \in [q, 1]} \E \left[ \phi_t^2 \right] \le \frac{1}{\alpha}. 
		\end{split} 
	\end{equation*}
	It is immediate to see that
	\begin{equation*}
		\VH_{1,\alpha}^{\sAMP} (h) = \, \sup_{q \in [0, 1]} \VH_{1,\alpha}^{\sAMP} (q, h).
	\end{equation*}
	In what follows, we present an extended variational principle for $\VH_{1,\alpha}^{\sAMP} (q, h)$, and establish the optimality of our two-stage AMP algorithm under a `no overlap gap' condition.
}
%

%

\subsection{Extended Parisi variational principle}
\label{sec:Parisi_formula}

We now develop a variational principle that is dual to 
the stochastic optimal control problem~\eqref{eq:SOC_First_1dim}.
The resulting formula is closely related to the Parisi  
variational principle, that we derived heuristically in Section 
\ref{sec:replica_method} (note that the Parisi formula for $m=1$ is in Remark \ref{rmk:ReplicaM1}).

We begin with defining two relevant function spaces.

\begin{defn}[Space of functional order parameters]\label{defn:gamma_space}
	Define
	\begin{equation*}
		\FS:= \left\{ (\mu, c) \in L^1 [0, 1] \times \R_{> 0}:\, \mu \vert_{[0, t]} \in L^{\infty} [0, t] \ \text{and} \ c + \int_{t}^{1} \mu(s) \d s > 0 \ \text{for all} \ t \in [0, 1) \right\},
	\end{equation*}
	and let 
	\begin{equation*}
		\FSG:= \Big\{\gamma: [0, 1] \to \R_{> 0} \mbox{ {\rm absolutely continuous}}:\, 
		(\mu, c) \in \FS \ \mbox{{\rm for }} \mu = \gamma' / \gamma^2, \, c = 1 / \gamma(1)
		\Big\}\, .
	\end{equation*}
	Further, for any $q \in [0, 1)$, define
	\begin{equation*}
		\FS (q) = \left\{ (\mu, c) \in \FS:\, \mu (t) = 0, \ \forall t \in [0, q] \right\}
	\end{equation*}
	and
	\begin{equation*}
		\FSG (q) =  \left\{ \gamma \in \FSG: \, \gamma \vert_{[0, q]} \ \mbox{is constant} \right\}.
	\end{equation*}
	Then, we see that $\gamma \in \FSG (q)$ if and only if $(\mu, c) \in\FS (q)$ with $(\mu, c) = (\gamma' / \gamma^2, 1 / \gamma(1))$.
	\modif{In the rest of this paper, this correspondence (and, equivalently, $1/\gamma(t)=c+\int_t^1\mu(s)\,\de s$) will be implicitly assumed.}
\end{defn}
We now extend the definition of the Parisi functional from $\mathscr{U} \times \R_{>0}$ to the larger space $\FS$. For $(\mu, c) \in \FS$, define
\begin{equation}\label{eq:Parisi_Func_extended}
	\mathsf{F} (\mu, c) = f_{\mu} (0, 0) + \frac{1}{2 \alpha} \int_{0}^{1} \frac{\d t}{c + \int_{t}^{1} \mu(s) \d s},
\end{equation}
where $f_{\mu}$ is the solution to the Parisi PDE~\eqref{eq:ParisiEq-m1}, which we copy here for the reader's convenience:
\begin{equation}\label{eq:parisi_mu_gen}
	\begin{split}
		&\partial_t f_{\mu} (t, x)+\frac{1}{2} \mu (t) \partial_x f_{\mu} (t, x)^2 +  \frac{1}{2} \partial_x^2 f_{\mu} (t, x) = \, 0, \\
		&f_{\mu} (1, x) = \,  \sup_{u \in \R} \left\{ h( x + u) - \frac{u^2}{2c}  \right\}.
	\end{split}
\end{equation}

\modif{
	To ensure that the Parisi functional~\eqref{eq:Parisi_Func_extended} is well-defined on the extended function space $\FS$, we establish the following theorem on the existence, uniqueness, and regularity of solutions to the Parisi PDE~\eqref{eq:parisi_mu_gen}, whose proof is presented in \cref{sec:solve_Parisi_PDE}.
	\begin{thm}\label{thm:solve_Parisi_PDE}
		Assume $(\mu, c) \in \FS$, and $h \in C^2 (\R)$ is Lipschitz continuous and bounded from above. Then, the Parisi PDE~\eqref{eq:parisi_mu_gen} has a unique weak solution $f_{\mu}$ such that $f_{\mu} (t, \cdot) \in C^2 (\R)$ for all $t \in [0, 1]$, and
		(for $1/\gamma(t)=c+\int_t^1\mu(s)\,\de s$)
		\begin{equation}\label{eq:parisi_regularity}
			\begin{split}
				\norm{\partial_x f_{\mu} (t, \cdot)}_{L^{\infty} (\R)} \le \, & \norm{h'}_{L^{\infty} (\R)}, \\
				\partial_x^2 f_{\mu} (t,  x) > \, & - \gamma(t), \, \forall (t, x) \in [0, 1] \times \R.
			\end{split}
		\end{equation}
		Further, if $\sup_{z \in \R} h''(z) < 1/c$, then we have additionally
		\begin{equation}\label{eq:parisi_additional_regularity}
			\partial_x^2 f_{\mu} (t,  x) \le \, C, \, \forall (t, x) \in [0, 1] \times \R,
		\end{equation}
		where the constant $C$ only depends on $(\mu, c)$.
	\end{thm}
}
We are now ready to state our main result establishing a Parisi-type variational principle for the stochastic optimal control problem~\eqref{eq:SOC_First_1dim}, \modif{with its proof deferred to \cref{sec:proof_general_q}.}
\begin{thm}\label{thm:two_stage_strong_duality}
	Assume $h \in C^2 (\R)$ is Lipschitz continuous and bounded from above.
	Then, the following holds.
	\begin{itemize} 
		\item[$(a)$] \emph{Variational formula.}
		Fix $q \in [0, 1)$ and $v\sim\normal(0,q)$. For any $\gamma \in \FSG (q)$ and $(\mu, c) \in \FS (q)$ satisfying $\mu = \gamma' / \gamma^2$ and $c = 1 / \gamma(1)$, we have:
		\begin{align*}
			\mathsf{F} (\mu, c) = \, \sup_{\substack{F: \R \to \R \\ \phi \in D[q, 1]}} \E \Bigg[ \, & h \left( v + \frac{1}{\alpha} F(v) + \int_{q}^{1} \left( 1 + \phi_t \right) \d B_t \right) - \frac{1}{2} \int_{q}^{1} \gamma(t) \left( \phi_t^2 - \frac{1}{\alpha} \right) \d t \\
			& - \frac{\gamma(q)}{2 \alpha} \left( \frac{1}{\alpha} F(v)^2 - q \right) \Bigg]\, .
		\end{align*}
		\item [$(b)$] \emph{Weak duality.} For any $q \in [0, 1)$, we have
		\begin{equation}\label{eq:weak_dual}
			\VH_{1,\alpha}^{\sAMP} (q,h) \le \, \inf_{(\mu, c) \in \mathscr{L} (q)} \mathsf{F} (\mu, c).
		\end{equation}
		\item [$(c)$] \emph{Strong duality.} Fix $q \in [0, 1)$, and assume there exists $(\mu_*, c_*) \in \FS(q)$ such that
		\begin{equation*}
			\mathsf{F} (\mu_*, c_*) = \, \inf_{(\mu, c) \in \FS (q)} \mathsf{F} (\mu, c).
		\end{equation*}
		Then there exists a feasible pair $(F^*, \phi^*)$ that satisfies
		\begin{equation}\label{eq:feasibility_F_phi}
			\E [F^*(v)^2] = \alpha q, \ \E [(F^*)'(v)^2] \le \alpha, \ \E [(\phi_t^*)^2] \modif{=} \frac{1}{\alpha}, \ \forall t \in [q, 1],
		\end{equation}
		such that
		\begin{equation}\label{eq:strong_dual}
			\mathsf{F} (\mu_*, c_*) = \, \E \left[ h \left( v + \frac{1}{\alpha} F^*(v) + \int_{q}^{1} \left( 1 + \phi_t^* \right) \d B_t \right) \right] = \VH_{1,\alpha}^{\sAMP} (q,h) .
		\end{equation}
		Further, $(F^*, \phi^*)$ is efficiently computable given access to $(\mu_*, c_*)$. 
	\end{itemize}
\end{thm}

\modif{
	\cref{thm:two_stage_strong_duality} (c) establishes that, if the infimum of $\mathsf{F} (\mu, c)$ over $\FS(q)$ is achieved, we have the dual characterization 
	\begin{align*}
		\VH_{1,\alpha}^{\sAMP} (q,h) = \, \inf_{(\mu, c) \in \FS (q)} \mathsf{F} (\mu, c)\, .
	\end{align*}
	Namely, there exists a two-stage AMP algorithm achieving the value $\inf_{(\mu, c) \in \FS (q)} \mathsf{F} (\mu, c)$, provided that this infimum is attained at some $(\mu_*, c_*) \in \FS (q)$. Assuming such a minimizer exists is  a common practice in the literature on optimization of spin glasses \cite{el2021optimization,sellke2024optimizing}. However, rigorously justifying this assumption is still an open problem, the main challenge being to prove sequential compactness of sublevel sets of $\mathsf{F} (\mu, c)$, on a space where the Parisi PDE is well-defined. 
	
	Given the optimal value achieved by our algorithm, it is natural to ask when it equals $\inf_{(\mu, c) \in \mathscr{U} \times \R_{> 0}} \mathsf{F} (\mu, c)$, the replica prediction for the global maximum of $H_{n, d}$. In the next theorem, we establish that these two values are equal under a certain ``no overlap gap'' assumption, which is analogous to the one considered in spin glass theory (see, e.g., \cite[Assumption 2]{el2021optimization}).}

\modif{\begin{thm}[No overlap gap]\label{thm:tight_under_no-ogp}
		Assume that there exists $(\mu_*, c_*) \in \mathscr{U} \times \R_{> 0}$, such that $\mu_* = 0$ on $[0, q]$ and $\mu_*$ is strictly increasing on $[q, 1)$ for some $q \in [0, 1)$, and moreover
		\begin{equation*}
			\mathsf{F} (\mu_*, c_*) = \inf_{(\mu_*, c_*) \in \mathscr{U} \times \R_{> 0}} \mathsf{F} (\mu, c).
		\end{equation*}
		Then 
		the infimum over $(\mu,c)\in \FS(q)$ is also achieved at $(\mu_*,c_*)$ and
		the same conclusion as that of \cref{thm:two_stage_strong_duality} (c) holds for $(\mu_*, c_*)$, i.e., there exists a feasible pair $(F^*, \phi^*)$ such that the corresponding two-stage AMP algorithm achieves $\mathsf{F} (\mu_*, c_*)$.
	\end{thm}

	Based on our analysis of the Parisi variational problem for general $(\mu, c) \in \FS$ in \cref{sec:variational_problem}, the proof of \cref{thm:tight_under_no-ogp} closely follows that of \cite[Corollary 2.2]{el2021optimization}. We therefore omit it for simplicity.
}

\subsection{Proof of \cref{prop:feasible_simple_1dim}}\label{sec:proof_simple_1dim}

We assume without loss of generality that $q = 0$, as this proof can be easily adapted to the case of general $q \in [0, 1]$. Under this assumption, it suffices to show that $\forall q(t) \in L^2 [0, 1]$ with $\norm{q(t)}_{L^2 [0, 1]} = 1$:
\begin{equation}\label{eq:reparametrize_equivalence}
	\begin{split}
		& \left\{ \Law \left( \int_{0}^{1} q(t) \left( 1 + \phi_t \right) \d B_t \right): \phi \in D [0, 1], \ \sup_{t \in [0, 1]} \E \left[ \phi_t^2 \right] \le \frac{1}{\alpha} \right\} \\
		= \, & \left\{ \Law \left( \int_{0}^{1} \left( 1 + \phi_t \right) \d B_t \right): \phi \in D [0, 1], \ \sup_{t \in [0, 1]} \E \left[ \phi_t^2 \right] \le \frac{1}{\alpha} \right\}.
	\end{split}
\end{equation}
With It\^{o}'s isometry, one may further assume that $q(t) \neq 0$ a.e., so that the mapping
\begin{equation*}
	t \mapsto s(t) := \int_{0}^{t} q(u)^2 \d u
\end{equation*}
is strictly increasing and satisfies $s(0) = 0$, $s(1) = 1$. Therefore, $s(t)$ admits a unique inverse $s^{-1} (t)$ with $s^{-1} (0) = 0$, $s^{-1} (1) = 1$. Now, let us define a new Gaussian process
\begin{equation*}
	W_t := \int_{0}^{t} q(u) \d B_u \iff W_0 = 0, \ \d W_t = q(t) \d B_t,
\end{equation*}
then \modif{a simple application of Dambis-Dubins-Schwarz theorem 
	(cf. \cite[Theorem 5.13]{le2016brownian}) implies
	that $\{ W_{s^{-1} (t)} \}_{0 \le t \le 1}$ is a standard Brownian motion.}

Now, according to the time-change formula for stochastic integrals (cf. Proposition 3.4.8 in \cite{karatzas2012brownian}), we obtain that
\begin{equation*}
	\int_{0}^{1} q(t) \left( 1 + \phi_t \right) \d B_t = \int_{0}^{1} \left( 1 + \phi_t \right) \d W_t = \int_{0}^{1} \left( 1 + \phi_{s^{-1} (t)} \right) \d W_{s^{-1} (t)},
\end{equation*}
where $\phi_{s^{-1} (t)} \in \cF_{s^{-1} (t)}^{B} = \cF_{s^{-1} (t)}^{W}$ is progressively measurable and satisfies $\E [\phi_{s^{-1} (t)}^2] \le 1/\alpha$. This establishes \cref{eq:reparametrize_equivalence}, completing the proof of \cref{prop:feasible_simple_1dim}.

%
%
\section{Two-stage AMP algorithm: Proof of Theorem \ref{thm:general_amp_achievability}}\label{sec:iamp}

This section is devoted to the construction and analysis of our general two-stage AMP algorithm, culminating in the proof of \cref{thm:general_amp_achievability}. We remark that in \cref{defn:Q_contraction} and \cref{thm:general_amp_achievability}, all vectors and matrices introduced in the current section have been transposed to align with standard notational conventions.

\subsection{Approximate message passing}
Following \cite{bayati2011dynamics, javanmard2013state, celentano2020estimation}, we define the general AMP algorithm as an iterative procedure, which generates two sequences of matrices $\{ \WW^t \}_{t \ge 1} \subset \R^{d \times m}$ and $\{ \VV^t \}_{t \ge 1} \subset \R^{n \times m}$ according to:
\begin{equation}\label{eq:amp_iter}
	\begin{split}
		\WW^{t + 1} & = \frac{1}{\sqrt{n}} \XX^\top F_t (\VV^{\le t}) - \sum_{s=1}^{t} G_s (\WW^{\le s}) K_{t, s}^\top, \\
		\VV^{t} & = \frac{1}{\sqrt{n}} \XX G_t(\WW^{\le t}) - \sum_{s=1}^{t} F_{s-1} (\VV^{\le s-1}) D_{t, s}^\top,
	\end{split}
\end{equation}
where $\WW^{1} = \XX^\top F_0 / \sqrt{n}$, \smodif{$F_0 \in \R^{n \times m}$ is independent of $\XX$,} and $\{ F_t: \R^{mt} \to \R^m \}_{t \ge 1}$ and $\{ G_t: \R^{mt} \to \R^m \}_{t \ge 1}$ are two sequences of Lipschitz functions. Moreover, we let $\WW^{\le t} = (\WW^{s})_{1 \le s \le t}$, $\VV^{\le t} = (\VV^{s})_{1 \le s \le t}$, and adopt the convention that the Lipschitz functions $G_t$ and $F_t$ apply row-wise to their arguments. The Onsager correction terms are defined as
\begin{equation}\label{eq:onsager_terms}
	D_{t, s} = \frac{1}{n} \sum_{i=1}^{d} \frac{\partial G_t}{\partial \ww_i^s} (\ww_i^1, \cdots, \ww_i^t), \ \,\, K_{t, s} = \frac{1}{n} \sum_{i=1}^{n} \frac{\partial F_t}{\partial \vv_i^s} (\vv_i^1, \cdots, \vv_i^t),
\end{equation}
where $\ww_i^s$ is the $i$-th row of $\WW^s$ and $\vv_i^s$ is the $i$-th row of $\VV^s$, respectively.
\begin{rem}
	We will refer to the matrices in Eq.~\eqref{eq:onsager_terms} as ``Onsager coefficients".
	The population version of these coefficients are used in some of the earlier
	literature, whereby the empirical average over $i$ is replaced by the expectation over the
	asymptotic distributions of the $\ww^t_i$'s and $\vv^t_i$'s. 
	By an induction argument in \cite{ javanmard2013state}, the high-dimensional asymptotics of
	these two versions of general AMP algorithms are the same. 
\end{rem}

As $n, d \to \infty$ and $n/d \to \alpha$, for any fixed $t \in \mathbb{N}$, the limiting joint distribution of the first $t$ AMP iterates is exactly characterized by the following proposition.
\begin{prop}[State evolution of AMP]\label{prop:state_evolution}
	Denote $\overline{Z}_{\le t} := (\overline{Z}_1, \cdots, \overline{Z}_t) \in \R^{mt}$ and $Z_{\le t} := (Z_1, \cdots, Z_t) \in \R^{mt}$, where each $Z_t, \overline{Z}_t \in \R^m$ \modif{is a row vector}. Then, the distributions of the random row vectors $\overline{Z}_{\le t} \in \R^{mt}$ and $Z_{\le t} \in \R^{mt}$ are defined as follows. Both $\overline{Z}_{\le t}$ and $Z_{\leq t}$ are multivariate normal with zero mean, and their covariance structures are specified via the following recursion:
	\begin{equation}\label{eq:covariance}
		\begin{split}
			& \E \left[ \overline{Z}_i^\top \overline{Z}_j \right] = \frac{1}{\alpha} \mathbb{E}\left[G_i \left( Z_{\leq i} \right)^\top G_j \left( Z_{\leq j} \right)\right], \quad i, j \geq 1, \\
			& \E \left[ Z_i^\top Z_j \right] =\mathbb{E}\left[F_{i-1}\left(\overline{Z}_{\leq i-1} \right)^\top F_{j-1}\left(\overline{Z}_{\leq j-1} \right)\right], \quad i, j \geq 1.
		\end{split}
	\end{equation}
	Under this specification, we have
	\begin{equation*}
		\plim_{n \to \infty} D_{t, s} = \frac{1}{\alpha} \E \left[ \frac{\partial G_t}{\partial w^s} \left( Z_{\le t} \right) \right], \ \plim_{n \to \infty} K_{t, s} = \E \left[ \frac{\partial F_t}{\partial v^s} \left( \overline{Z}_{\le t}\right) \right].
	\end{equation*}
	Furthermore, for any pseudo-Lipschitz functions $\psi_1, \psi_2$, almost surely
	\begin{equation}\label{eq:state_evolution}
		\begin{split}
			\lim_{n \to \infty} \frac{1}{d} \sum_{i=1}^{d} \psi_1 \left( \ww_i^1, \cdots, \ww_i^t \right) & = \E \left[ \psi_1 \left( Z_{\le t} \right) \right], \\
			\lim_{n \to \infty} \frac{1}{n} \sum_{i=1}^{n} \psi_2 \left( \vv_i^1, \cdots, \vv_i^t \right) & = \E \left[ \psi_2 \left( \overline{Z}_{\le t} \right) \right],
		\end{split}
	\end{equation}
	as $n, d \to \infty$ and $n/d \to \alpha$. 
\end{prop}
\begin{proof}
	This can be deduced from the results in \cite{javanmard2013state, celentano2020estimation, montanari2022statistically}.
\end{proof}

\begin{rem}\label{rmk:Randomness}
	As mentioned in the introduction, we allow our two-stage AMP algorithm to be randomized. Within the AMP framework, randomization can be implemented by letting 
	the functions $F_t$ (or $G_t$) depend on some additional randomness. For example, one can replace
	$F_t(\vv_i^1,\dots,\vv^t_i)$ by 
	$F_t(\vv_i^1,\dots,\vv^t_i,\omega_i)$ with $(\omega_i)_{i\ge 1}\sim_{\iid} \Unif [0,1]$, the uniform distribution on $[0, 1]$. For simplicity of notation, we will leave this dependence implicit.
	Expectations in the state evolution equations are understood to be taken with respect to
	these random variables as well.
\end{rem}
\begin{cor}\label{cor:se_weak_conv}
	As $n / d \to \alpha$, the empirical distribution of $(\vv_i^1, \cdots, \vv_i^t)_{1 \le i \le n}$ almost surely weakly converges to the law of $\overline{Z}_{\le t}$. Similarly, the empirical distribution of $(\ww_i^1, \cdots, \ww_i^t)_{1 \le i \le d}$ almost surely weakly converges to the law of $Z_{\le t}$.
\end{cor}
\begin{proof}
	We only prove the first part as the second part is exactly identical. Following the notation of \cite{berti2006almost}, we denote by $\nu_n$ the empirical distribution of $(\vv_i^1, \cdots, \vv_i^t)_{1 \le i \le n}$ and by $\nu$ the law of $\overline{Z}_{\le t}$. Since for each $u$, the function $f_u (x) = \exp(i \langle u, x \rangle)$ is bounded Lipschitz, we know from Eq.~\eqref{eq:state_evolution} that
	\begin{equation*}
		\P \Big( \lim_{n \to \infty} \nu_n (f_u) = \nu (f_u) \Big) = 1, \,\, \forall u\in \R^{mt}.
	\end{equation*}
	Applying Theorem 2.6 in \cite{berti2006almost} implies that $\nu_n \stackrel{w}{\Rightarrow} \nu$ almost surely.
\end{proof}

\subsection{First stage: Fixed-point AMP}\label{sec:fixed_pt_amp}
In this section, we present the first stage of our general two-stage AMP algorithm, which consists of several identical AMP iterations that finally converges to a fixed point of the state evolution equations~\eqref{eq:covariance}. This fixed point will be the starting point of the pure incremental part in the second stage.

We now specify our choices of $\{ F_t \}_{t \ge 0}$ and $\{ G_t \}_{t \ge 1}$ in this stage. We set all $F_t$ with $t \ge 1$ to be $F$, a $Q$-contraction that only depends on $\VV^t$, and all $G_t$ to be the identity mapping on $\WW^t$. Consequently, $D_{t} = (d/n) I_m$ and  the AMP iterations reduce to
\begin{equation}\label{eq:amp_iter_stage_1}
	\begin{split}
		\WW^{t + 1} & = \frac{1}{\sqrt{n}} \XX^\top F (\VV^{t}) - \WW^{t} K_{t}^\top, \\
		\VV^{t} & = \frac{1}{\sqrt{n}} \XX \WW^{t} - \frac{d}{n}  F (\VV^{t-1}) ,
	\end{split}
\end{equation}
where we still have $\WW^{1} = \XX^\top F_0 / \sqrt{n}$, and
\begin{equation}\label{eq:onsager_terms_stage_1}
	K_{t} = \frac{1}{n} \sum_{i=1}^{n} \frac{\partial F}{\partial \vv_i^t} (\vv_i^t)
\end{equation}
are the Onsager terms. Finally, for $\omega \sim \Unif [0, 1]$ independent of everything else, we let $F_0 = F_0 (\omega)$ be such that
\begin{equation*}
	\E \left[ F_0 (\omega)^\top F_0 (\omega) \right] = \alpha Q.
\end{equation*}

\begin{prop}\label{lem:cov_mu_Q}
	Let $F$ and $F_0$ be as specified above. Then, for all $t \ge 1$, $\E [ \overline{Z}_t^\top \overline{Z}_t ] = Q$. Further, we have
	\begin{equation}\label{eq:lim_cov_Q}
		\lim_{t \to \infty} \E \left[ \overline{Z}_{t+1}^\top \overline{Z}_t \right] = Q.
	\end{equation}
\end{prop}
\begin{proof}
	We first prove $\E [ \overline{Z}_t^\top \overline{Z}_t ] = Q$ by induction. For $t = 1$, our choice of $F_0$ implies
	\begin{equation*}
		\E \left[ \overline{Z}_1^\top \overline{Z}_1 \right] = \, \frac{1}{\alpha} \E \left[ Z_1^\top Z_1 \right] = \frac{1}{\alpha} \E \left[ F_0^\top F_0 \right] = Q.
	\end{equation*}
	Now assume the conclusion holds for $t$. For $t+1$, we have
	\begin{align*}
		\E \left[ \overline{Z}_{t+1}^\top \overline{Z}_{t+1} \right] = \, \frac{1}{\alpha} \E \left[ Z_{t+1}^\top Z_{t+1} \right] = \frac{1}{\alpha} \mathbb{E}\left[F \left(\overline{Z}_{t} \right)^\top F \left(\overline{Z}_{t} \right)\right] = Q,
	\end{align*}
	since $F$ is a $Q$-contraction. This completes the induction. 
	
	To prove \cref{eq:lim_cov_Q}, we define $C_t = \E [\overline{Z}_{t+1}^\top \overline{Z}_t]$, then we have
	\begin{equation*}
		C_1 = \frac{1}{\alpha} \E \left[ Z_2^\top Z_1 \right] = \frac{1}{\alpha} \E \left[ F \left( \overline{Z}_1 \right)^\top F_0 \right] = 0,
	\end{equation*}
	and the recurrence relation
	\begin{equation*}
		C_{t+1} = \frac{1}{\alpha} \E \left[ F \left( \overline{Z}_{t+1} \right)^\top F \left( \overline{Z}_{t} \right) \right].
	\end{equation*}
	Define for any $C \in \S_+^m$ satisfying $0 \preceq C \preceq Q$:
	\begin{equation*}
		\psi (C) = \frac{1}{\alpha} \E_{(C, Q)} \left[ F \left( \overline{Z} \right)^\top F \left( \overline{Z}' \right) \right],
	\end{equation*}
	where $\E_{(C, Q)}$ denotes the expectation taken under $(\overline{Z}, \overline{Z}')^\top
	\sim \normal \left( 0, \begin{bmatrix}
		Q & C \\
		C & Q 
	\end{bmatrix} \right)$. We then know that $\psi (Q) = Q$, and $C_{t+1} = \psi (C_t)$. Next, we will show that
	\begin{itemize}
		\item [$(a)$] $\psi$ is increasing with respect to the Loewner order, i.e., for $0 \preceq A \preceq B \preceq Q$, we have $\psi (A) \preceq \psi (B)$.
		\item [$(b)$] For any $0 \preceq C \preceq Q$, $\psi (C) = C$ if and only if $C = Q$, namely $Q$ is the only fixed point of $\psi$ \modif{in the region $\{ C \in \S_+^m : 0 \preceq C \preceq Q \}$}.
	\end{itemize}
	\noindent \textbf{Proof of $(a)$.} Denote $H = B - A \succeq 0$, and define for $t \in [0, 1]$:
	\begin{equation*}
		\psi_{H} (t) = \psi(A + tH).
	\end{equation*}
	Then, it suffices to show that $\psi_{H} (1) \succeq \psi_{H} (0)$. Using Gaussian interpolation (see, e.g., \cite[Lemma 1.3.1]{talagrand2010mean}), we deduce that
	\begin{align*}
		\psi_{H}' (t) = \, \frac{1}{\alpha} \E_{(A+tH, Q)} \left[ J_F \left( \overline{Z} \right)^\top H J_F \left( \overline{Z}' \right) \right] \succeq 0,
	\end{align*}
	which implies $\psi_{H} (1) \succeq \psi_{H} (0)$ and completes the proof of part $(a)$.
	
	\vspace{0.5em} \noindent \textbf{Proof of $(b)$.} First note that, since $\cuF_{m, \alpha}^{\salg}$ is closed under weak limits, we can assume without loss of generality that \eqref{eq:Q_contraction} in \cref{defn:Q_contraction} holds with strict inequality. Now, assume by contradiction that $\psi (C) = C$ for some $0 \preceq C \preceq Q$, $C \neq Q$. Let $H = Q - C$, and 
	define for $\beta \in \R^m \backslash \{0 \}$ and $t \in [0, 1]$:
	\begin{equation*}
		\psi_{\beta, H} (t) = \beta^\top \psi(C + tH) \beta.
	\end{equation*}
	Similar to the proof of $(a)$, we can show that $\psi_{\beta, H}' (t) \ge 0$ and $\psi_{\beta, H}'' (t) \ge 0$. 
	Hence,
	\begin{align*}
		\beta^\top \left( Q - C \right) \beta = \, & \beta^\top \left( \psi(Q) - \psi(C) \right) \beta = \psi_{\beta, H} (1) - \psi_{\beta, H} (0) = \int_{0}^{1} \psi_{\beta, H}' (t) \d t \\
		\le \, & \psi_{\beta, H}' (1) = \frac{1}{\alpha} \beta^\top \E_{\overline{Z} \sim \normal (0, Q)} \left[ J_F \left( \overline{Z} \right)^\top H \JF \left( \overline{Z} \right) \right] \beta,
	\end{align*}
	which implies that for all $\beta \in \R^m \backslash \{ 0 \}$,
	\begin{equation*}
		\left\langle H, \beta \beta^\top \right\rangle \le \, \left\langle H, \frac{1}{\alpha} \E_{\overline{Z} \sim \normal (0, Q)} \left[ J_F \left( \overline{Z} \right) \beta \beta^\top \JF \left( \overline{Z} \right)^\top \right] \right\rangle,
	\end{equation*}
	thus leading to $\forall S \in \S_+^m \backslash \{ 0 \}$,
	\begin{align*}
		\left\langle H, S \right\rangle \le \, \left\langle H, \frac{1}{\alpha} \E_{\overline{Z} \sim \normal (0, Q)} \left[ J_F \left( \overline{Z} \right) S \JF \left( \overline{Z} \right)^\top \right] \right\rangle,
	\end{align*}
	contradicting our assumption that $F$ is a $Q$-contraction. This proves part $(b)$.
	
	Finally, we show that $C_t \to Q$ as $t \to \infty$. Since $\psi$ is increasing and $C_1 = 0 \preceq C_2$, we know that the sequence $\{ C_t \}$ is increasing. Further since $C_t \preceq Q$ is bounded, we know that $C_t \to C$ for some $0 \preceq C \preceq Q$, and $C$ is a fixed point of $\psi$. By part $(b)$, we must have $C = Q$. This completes the proof.
\end{proof}
As discussed in \cref{sec:IAMP}, this fixed-point AMP stage will stop after $T_1$ steps, where $T_1 \in \mathbb{N}_+$ is to be determined later. 

%
%

\subsection{Second stage: Incremental AMP}

In this section, we describe the second stage of our algorithm, which is an incremental AMP (IAMP) 
procedure introduced in \cite{montanari2021optimization}.
We will see that the asymptotics of this incremental stage admit a
stochastic integral representation, under a suitable scaling limit. 

For this IAMP stage, the non-linear functions $\{ F_t \}_{t \ge T_1}$ and $\{ G_t \}_{t \ge T_1 + 1}$ are chosen to satisfy the following assumption:
\begin{ass}\label{ass:F_t_and_G_t}
	Consider the random variables $(V^t)_{t \ge 1}$ and $(W^t)_{t \ge 1}$ defined as follows:
	\begin{equation*}
		\begin{split}
			& (V^t)_{1 \le t \le T_1} = (\overline{Z}_t)_{1 \le t \le T_1}, \ (V^t)_{t \ge T_1 + 1} \sim_{\iid} \normal(0, I_m), \ (V^t)_{t \ge T_1 + 1} \perp (V^t)_{1 \le t \le T_1} , \\
			& (W^t)_{1 \le t \le T_1} = (Z_t)_{1 \le t \le T_1}, \ (W^t)_{t \ge T_1 + 1} \sim_{\iid} \normal(0, I_m), \ (W^t)_{t \ge T_1 + 1} \perp (W^t)_{1 \le t \le T_1}.
		\end{split}
	\end{equation*}
	We impose the following second moment constraints on $\{ F_t \}_{t \ge T_1}$ and $\{ G_t \}_{t \ge T_1 + 1}$:
	\begin{enumerate}
		\item $F_{T_1}$ is independent of everything else, with $\E [F_{T_1}^\top F_{T_1} ] = I_m$. $G_{T_1 + 1}$ is only a function of $W^{T_1 + 1}$, satisfying
		\begin{equation*}
			\E \left[ G_{T_1 + 1} \left( W^{T_1 + 1} \right)^\top G_{T_1 + 1} \left( W^{T_1 + 1} \right) \right] = \alpha I_m.
		\end{equation*} 
		\item The functions $\{ F_t \}_{t \ge T_1 + 1}$ and $\{ G_t \}_{t \ge T_1 + 2}$ have the form:
		\begin{equation*}
			F_t (V^{\le t}) = V^t \Phi_{t-1} (V^{\le t-1}), \ G_t(W^{\le t}) = W^t \Psi_{t-1} (W^{\le t-1}),
		\end{equation*}
		where \modif{$\Phi_{t-1}: \R^{m(t-1)} \to \R^{m \times m}$ and $\Psi_{t-1}: \R^{m(t-1)} \to \R^{m \times m}$} are $m \times m$ matrix-valued functions satisfying
		\begin{equation*}
			\begin{split}
				& \E \left[ \Phi_{t-1} \left( V^{\le t-1} \right)^\top \Phi_{t-1} \left( V^{\le t-1} \right) \right] = I_m, \\
				& \E \left[ \Psi_{t-1} \left( W^{\le t-1} \right)^\top \Psi_{t-1} \left( W^{\le t-1} \right) \right] = \alpha I_m.
			\end{split}
		\end{equation*}
	\end{enumerate}
\end{ass}

\modif{It is straightforward to show that there exist functions $\{ F_t \}_{t \ge T_1}$ and $\{ G_t \}_{t \ge T_1 + 1}$ satisfying the above assumption. With such choices, the state evolution of the IAMP stage is characterized by the following:}

\begin{prop}\label{prop:state_evolution_iamp}
	Under Assumption~\ref{ass:F_t_and_G_t}, we have $(Z_t)_{t \ge T_1 + 1} \sim_{\iid} \normal (0, I_m)$ is independent of $(Z_t)_{1 \le t \le T_1}$, and $(\overline{Z}_t)_{t \ge T_1 + 1} \sim_{\iid} \normal (0, I_m)$ is independent of $(\overline{Z}_t)_{1 \le t \le T_1}$.
\end{prop}

\begin{proof}
	According to Proposition~\ref{prop:state_evolution}, we already know that $(Z_t)_{t \ge T_1 + 1}$ and $(\overline{Z}_t)_{t \ge T_1 + 1}$ are centered multivariate Gaussians, hence it suffices to show that for any $k \ge 1$,
	\begin{equation}\label{eq:induction_claim_1}
		\begin{split}
			& (Z_t)_{T_1 + 1 \le t \le T_1 + k} \sim_{\iid} \normal (0, I_m), \quad (Z_t)_{T_1 + 1 \le t \le T_1 + k} \perp (Z_t)_{1 \le t \le T_1}, \\ 
			& (\overline{Z}_t)_{T_1 + 1 \le t \le T_1 + k} \sim_{\iid} \normal (0, I_m), \quad (\overline{Z}_t)_{T_1 + 1 \le t \le T_1 + k} \perp (\overline{Z}_t)_{1 \le t \le T_1}.
		\end{split}
	\end{equation}
	We prove the above claim via induction on $k$. For $k=1$, using Eq.~\eqref{eq:covariance} and Assumption~\ref{ass:F_t_and_G_t}, we obtain that
	\begin{equation*}
		\E \left[ Z_{T_1 + 1}^\top Z_{T_1 + 1} \right] = \E \left[ F_{T_1}^\top F_{T_1} \right] = I_m.    
	\end{equation*}
	Further, for any $1 \le t \le T_1$, we have
	\begin{align*}
		\E \left[ Z_{T_1 + 1}^\top Z_{t} \right] = \, \E \left[ F_{T_1}^\top F \left( \overline{Z}_{t-1} \right) \right] = 0.
	\end{align*}
	Therefore, $Z_{T_1 + 1} \sim \normal (0, I_m)$, and is independent of $(Z_t)_{1 \le t \le T_1}$. Similarly, by Eq.~\eqref{eq:covariance} and Assumption~\ref{ass:F_t_and_G_t}, we know that for all $1 \le t \le T_1$,
	\begin{align*}
		\E \left[ \overline{Z}_{T_1 + 1}^\top \overline{Z}_{T_1 + 1} \right] = \, & \frac{1}{\alpha} \E \left[ G_{T_1 + 1} \left( W^{T_1 + 1} \right)^\top G_{T_1 + 1} \left( W^{T_1 + 1} \right) \right] = I_m, \\
		\E \left[ \overline{Z}_{T_1 + 1}^\top \overline{Z}_t \right] = \, & \frac{1}{\alpha} \E \left[ G_{T_1 + 1} \left( W^{T_1 + 1} \right)^\top W^t \right] = 0,
	\end{align*}
	where the last line follows from the fact that $W^t \perp W^{T_1 + 1}$.
	As a consequence, we deduce that $\overline{Z}_{T_1 + 1} \sim \normal (0, I_m)$ and is independent of $(\overline{Z}_t)_{1 \le t \le T_1}$.
	This completes the base case of our induction.
	
	Now assume that our claim~\eqref{eq:induction_claim_1} holds for $k \in \mathbb{N}$. For $k+1$, we have
	\begin{align*}
		& \E \left[ Z_{T_1 + k+1}^\top Z_{T_1 + k+1} \right] = \E \left[ F_{T_1 + k} (\overline{Z}_{\le T_1 + k})^\top F_{T_1+k} (\overline{Z}_{\le T_1+k}) \right] \\
		=\, & \E \left[ \Phi_{T_1+k-1} (\overline{Z}_{\le T_1+k-1})^\top \overline{Z}_{T_1+k}^\top \overline{Z}_{T_1+k} \Phi_{T_1+k-1} (\overline{Z}_{\le T_1+k-1}) \right] \\
		=\, & \E \left[ \Phi_{T_1 + k-1} (\overline{Z}_{\le T_1+k-1})^\top \E \left[ \overline{Z}_{T_1+k}^\top \overline{Z}_{T_1+k} \right] \Phi_{T_1+k-1} (\overline{Z}_{\le T_1+k-1}) \right] \\
		=\, & \E \left[ \Phi_{T_1+k-1} (\overline{Z}_{\le T_1+k-1})^\top \Phi_{T_1+k-1} (\overline{Z}_{\le T_1+k-1}) \right] = I_m,
	\end{align*}
	and for all $t \le T_1 + k$,
	\begin{align*}
		& \E \left[ Z_{T_1+k+1}^\top Z_{t} \right] = \E \left[ F_{T_1+k} (\overline{Z}_{\le T_1+k})^\top F_{t-1} (\overline{Z}_{\le t-1}) \right] \\
		=\, & \E \left[ \Phi_{T_1+k-1} (\overline{Z}_{\le T_1+k-1})^\top \overline{Z}_{T_1+k}^\top F_{t-1} (\overline{Z}_{\le t-1}) \right] \\
		=\, & \E \left[ \Phi_{T_1+k-1} (\overline{Z}_{\le T_1+k-1})^\top \E \left[ \overline{Z}_{T_1+k} \right]^\top F_{t-1} (\overline{Z}_{\le t-1}) \right] = 0,
	\end{align*}
	where the last equality follows from our induction hypothesis.
	This proves that 
	$$
	(Z_t)_{T_1 + 1 \le t \le T_1 + k + 1} \sim_{\iid} \normal (0, I_m),
	$$
	and are independent of $(Z_t)_{1 \le t \le T_1}$. With a similar calculation, we can show that 
	$$
	(\overline{Z}_t)_{T_1 + 1 \le t \le T_1 + k + 1} \sim_{\iid} \normal (0, I_m),
	$$
	and are independent of $(\overline{Z}_t)_{1 \le t \le T_1}$.
	This completes the induction step and the proof of the proposition.
\end{proof}
The IAMP stage is run for $T_2$ steps after the fixed-point AMP stage, for some $T_2 \in \mathbb{N}_+$ to be determined. In what follows, we combine the two stages of our AMP algorithm to complete the proof of \cref{thm:general_amp_achievability}.

\subsection{Combining the two stages}
Let $\WW_F = \WW^{T_1} / \sqrt{n}$ be the output of the fixed-point AMP stage. By state evolution,
\begin{align*}
	\WW_F^\top \WW_F = \, & \frac{1}{n} \left( \WW^{T_1} \right)^\top \WW^{T_1} = \frac{1}{n} \sum_{i=1}^{d} \left( \ww_i^{T_1} \right)^\top \ww_i^{T_1} \\
	\to \, & \frac{1}{\alpha} \E \left[ Z_{T_1}^\top Z_{T_1} \right] = \E \left[ \overline{Z}_{T_1}^\top \overline{Z}_{T_1} \right] = Q
\end{align*}
almost surely as $n \to \infty$. Let $Q_1, \cdots, Q_{T_2}$ be $T_2$ non-random $m \times m$ matrices such that
\begin{equation*}
	\sum_{t=1}^{T_2} Q_t^\top Q_t = I_m - Q,
\end{equation*}
we define
\begin{align*}
	\WW_I &= \frac{1}{\sqrt{n}} \sum_{t=1}^{T_2} G_{T_1 + t+1} \left( \WW^{\le T_1 + t + 1} \right) Q_t\,
\end{align*}
as the output of the IAMP stage. Following \cref{sec:IAMP}, we construct the final output of our two-stage algorithm by letting $\WW_Q = \WW_F + \WW_I$, and setting
\begin{equation*}
	\widehat{\WW}_n^{\sAMP} = \WW_Q (\WW_Q^\top \WW_Q)^{-1/2} \in O(d, m).
\end{equation*}
The theorem below characterizes the set of $(\alpha, m)$-feasible distributions achievable by our AMP algorithm:
\begin{thm}\label{thm:iamp_feasibility}
	Let Assumption~\ref{ass:F_t_and_G_t} hold, and further assume that for all $t \ge 0$, $F_t$ is continuous, and for all $t \ge 1$, $G_t$ is Lipschitz continuous. Let the weight matrix $\widehat{\WW}_n^{\sAMP}$ be constructed as above, then we have almost surely,
	\begin{equation}\label{eq:AMP_convergence_final_disc}
		\frac{1}{n} \sum_{i=1}^{n} \delta_{ \xx_i^\top \widehat{\WW}_n^{\sAMP}} \stackrel{w}{\Rightarrow} \operatorname{Law} \left( \overline{Z}_{T_1} + \frac{1}{\alpha} F \left( \overline{Z}_{T_1-1} \right) + \sum_{t=1}^{T_2} \left( \overline{Z}_{T_1 + t+1} + F_{T_1+t} \left( \overline{Z}_{\le T_1+t} \right) A_t \right) Q_t \right)
	\end{equation}
	as $n \to \infty$, where
	\begin{equation*}
		A_t = \frac{1}{\alpha} \E \left[ \Psi_{T_1+t} \left( Z_{\le T_1+t} \right) \right], \ 1 \le t \le T_2.
	\end{equation*}
	Moreover, for any sequence of $m \times m$ matrices $(A_t)_{1 \le t \le T_2}$ satisfying $A_{t}^\top A_{t} \preceq I_m / \alpha$, and functions $(F_{T_1 + t})_{1 \le t \le T_2}$ satisfying point 2 of Assumption~\ref{ass:F_t_and_G_t} (not necessarily continuous),
	\begin{equation}\label{eq:AMP_feasible_disc}
		\operatorname{Law} \left( \overline{Z}_{T_1} + \frac{1}{\alpha} F \left( \overline{Z}_{T_1-1} \right) + \sum_{t=1}^{T_2} \left( \overline{Z}_{T_1 + t+1} + F_{T_1+t} \left( \overline{Z}_{\le T_1+t} \right) A_t \right) Q_t \right) \in \cuF_{m, \alpha}^{\salg}.
	\end{equation}
\end{thm}
\begin{proof}
	By our assumption, we know that for all $1 \le s, t \le T$, $G_t G_s$ is pseudo-Lipschitz of order $2$. Hence, we have almost surely,
	\begin{align*}
		\WW_I^\top \WW_I =\, & \frac{1}{n} \sum_{t=1}^{T_2} \sum_{s=1}^{T_2} Q_t^\top G_{T_1 + t + 1} \left( \WW^{\le T_1 + t + 1} \right)^\top G_{T_1+s+1} \left( \WW^{\le T_1+s+1} \right) Q_s \\
		=\, & \sum_{t=1}^{T_2} \sum_{s=1}^{T_2} Q_t^\top \left( \frac{1}{n} \sum_{i=1}^{d} G_{T_1+t+1} (\ww_i^{\le T_1+t+1})^\top G_{T_1+s+1} (\ww_i^{\le T_1+s+1}) \right) Q_s \\
		\to\, & \frac{1}{\alpha} \sum_{t=1}^{T_2} \sum_{s=1}^{T_2} Q_t^\top \E \left[ G_{T_1+t+1} \left( Z_{\leq T_1+t+1} \right)^\top G_{T_1+s+1} \left( Z_{\leq T_1+s+1} \right) \right] Q_s \\
		=\, & \sum_{t=1}^{T_2} \sum_{s=1}^{T_2} Q_t^\top \E \left[ \overline{Z}_{T_1+t+1}^\top \overline{Z}_{T_1+s+1} \right] Q_s = \sum_{t=1}^{T_2} Q_t^\top Q_t = I_m - Q.
	\end{align*}
	Similarly, we can show that $\WW_F^\top \WW_I \to 0$ almost surely. As we already know $\WW_F^\top \WW_F \to Q$, this implies that $\WW_Q^\top \WW_Q \to I_m$ almost surely as $n \to \infty$. Using Slutsky's theorem, it now suffices to consider the limiting empirical distribution of the rows of $ \XX \WW_Q$. By direct calculation, we obtain that
	\begin{align*}
		\XX \WW_{Q} = \, & \XX \WW_{F} + \XX \WW_{I} = \VV^{T_1} + \frac{d}{n} F(\VV^{T_1-1}) + \frac{1}{\sqrt{n}} \sum_{t=1}^{T_2} \XX G_{T_1 + t+1} \left( \WW^{\le T_1 + t + 1} \right) Q_t \\
		= \, & \VV^{T_1} + \frac{d}{n} F(\VV^{T_1-1}) + \sum_{t=1}^{T_2} \left( \VV^{T_1 + t+1} + \sum_{s=1}^{T_1+t+1} F_{s-1} (\VV^{\le s-1}) D_{T_1+t+1, s}^\top \right) Q_t,
	\end{align*}
	where by state evolution,
	\begin{align*}
		D_{T_1+t+1, s} = \, & \frac{1}{n} \sum_{i=1}^{d} \frac{\partial G_{T_1+t+1}}{\partial \ww_i^s} ( \ww_i^{\le T_1+t+1}) \stackrel{a.s.}{\to} \, \frac{1}{\alpha} \E \left[ \frac{\partial G_{T_1+t+1}}{\partial w^s} \left( Z_{\leq T_1+t+1} \right) \right] \\
		\stackrel{(i)}{=} \, & \frac{\bone_{s = T_1+t+1}}{\alpha} \E \left[ \Psi_{T_1+t} \left( Z_{\le T_1+t} \right) \right]^\top = \bone_{s = T_1+t+1} A_t^\top.
	\end{align*}
	\modif{In the above display, $(i)$ is due to the observation that if $s < T_1 + t + 1$, then
		\begin{align*}
			& \E \left[ \frac{\partial G_{T_1+t+1}}{\partial w^s} \left( Z_{\leq T_1+t+1} \right) \right] = \, \E \left[ Z_{T_1 + t + 1} \frac{\partial \Psi_{T_1+t}}{\partial w^s} \left( Z_{\leq T_1+t} \right)  \right] \\
			= \, & \E \left[ Z_{T_1 + t + 1} \right] \E \left[ \frac{\partial \Psi_{T_1+t}}{\partial w^s} \left( Z_{\leq T_1+t} \right)  \right] = 0.
		\end{align*}
	}
	Therefore, it suffices to consider the limiting empirical distribution of
	\begin{equation}\label{eq:simplified_AMP_final}
		\VV^{T_1} + \frac{d}{n} F(\VV^{T_1-1}) + \sum_{t=1}^{T_2} \left( \VV^{T_1 + t+1} + F_{T_1+t} \left( \VV^{\le T_1+t} \right) A_t \right) Q_t.
	\end{equation}
	Now, since $F_{T_1+t}$ and $F$ are continuous, \modif{applying \cref{cor:se_weak_conv}} and continuous mapping theorem implies that the limiting empirical distribution of the rows of \eqref{eq:simplified_AMP_final} almost surely weakly converges to
	\begin{equation*}
		\Law \left( \overline{Z}_{T_1} + \frac{1}{\alpha} F \left( \overline{Z}_{T_1-1} \right) + \sum_{t=1}^{T_2} \left( \overline{Z}_{T_1 + t+1} + F_{T_1+t} \left( \overline{Z}_{\le T_1+t} \right) A_t \right) Q_t \right),
	\end{equation*}
	which proves \cref{eq:AMP_convergence_final_disc}. 
	
	To show the second part of the theorem, note that in the IAMP stage, the only requirement for the function $\Psi_{T_1+t}$ is
	\begin{equation*}
		\E \left[ \Psi_{T_1+t} \left( Z_{\le T_1+t} \right)^\top \Psi_{T_1+t} \left( Z_{\le T_1+t} \right) \right] = \alpha I_m.
	\end{equation*}
	Hence, for any $t \ge 1$ and $A_{t} \in \R^{m \times m}$ such that $A_{t}^\top A_{t} \preceq I_m / \alpha$, there exists a function $\Psi_{T_1+t}$ that satisfies the condition of \cref{thm:iamp_feasibility}, and that
	\begin{equation*}
		\E \left[ \Psi_{T_1+t} \left( Z_{\le T_1+t} \right) \right] = \alpha A_t.
	\end{equation*}
	This proves that if $F$ and $F_{T_1 + t}$ are continuous, then
	\begin{equation*}
		\operatorname{Law} \left( \overline{Z}_{T_1} + \frac{1}{\alpha} F \left( \overline{Z}_{T_1-1} \right) + \sum_{t=1}^{T_2} \left( \overline{Z}_{T_1 + t+1} + F_{T_1+t} \left( \overline{Z}_{\le T_1+t} \right) A_t \right) Q_t \right) \in \cuF_{m, \alpha}^{\salg},
	\end{equation*}
	and can be achieved by our two-stage AMP algorithm, where the only constraint on $\{ A_t \}_{1 \le t \le T_2}$ is $A_t^\top A_t \preceq I_m / \alpha$. The second part of Theorem~\ref{thm:iamp_feasibility} then follows immediately by combining this result and the fact that $\cuF_{m, \alpha}^{\salg}$ is closed under weak limits, since we can approximate general $L^2$-integrable functions by continuous functions to arbitrary accuracy.
\end{proof}

We next take the limit $T_1 \to \infty$. Recall that \cref{lem:cov_mu_Q} implies
\begin{equation*}
	\lim_{T_1 \to \infty} \E \left[ \left\| \overline{Z}_{T_1} - \overline{Z}_{T_1-1} \right\|_2^2 \right]  = 0.
\end{equation*}
Leveraging our choice of $F_{T_1 + t}$ for the IAMP stage, the $(\alpha, m)$-feasible distribution given by \cref{eq:AMP_feasible_disc} can be further simplified. Note that in the following theorem, we recast $T_2$ as $T$, the $\overline{Z}_t$'s as $V$ and $V^t$'s, \smodif{and absorb the $A_t$'s into the $\Phi_{t-1}$'s.}

\begin{thm}\label{thm:FeasibleDiscrete}
	For any $T \in \mathbb{N}_+$, let $(V^t)_{1 \le t \le T} \sim_{\iid} \sN (0, I_m)$ be independent of $V \sim \normal (0, Q)$. Define
	\begin{equation}\label{eq:discrete_riemann_sum}
		U = V + \frac{1}{\alpha} F \left( V \right) + \sum_{t=1}^{T} \left( V^{t+1} + V^{t} \Phi_{t-1} \left( V^{\le t-1} \right) \right) Q_t,
	\end{equation}
	where $\sum_{t=1}^{T} Q_t^\top Q_t = I_m - Q$, and
	\begin{equation*}
		\E \left[ \Phi_{t-1} \left( V^{\le t-1} \right)^\top \Phi_{t-1} \left( V^{\le t-1} \right) \right] \preceq \frac{I_m}{\alpha}, \ \forall t \ge 1.
	\end{equation*}
	Then, we have $\Law(U) \in \cuF_{m, \alpha}^{\salg}$.
\end{thm}

Finally, we take the scaling limit $T \to \infty$ for the feasible distributions described in 
Theorem \ref{thm:FeasibleDiscrete}, yielding a stochastic integral representation for  probability measures achieved by our two-stage AMP algorithm. As a consequence, we prove \cref{thm:general_amp_achievability}. For ease of exposition, we only consider the case $m=1$, as the proof for general $m\ge 2$ follows similarly.

For $m=1$, Eq.~\eqref{eq:discrete_riemann_sum} reads
\begin{align*}
	& U = v + \frac{1}{\alpha} F (v) + \sum_{t=1}^{T} q_t ( v^{t+1} + \phi_{t-1} (v^{\le t-1}) v^{t} ), \\
	& \sum_{t=1}^{T} q_t^2 = 1 - q, \ \E \left[ F (v)^2 \right] = \alpha q, \ \E \left[ F'(v)^2 \right] \le \alpha, \\
	& \E \left[ \phi_{t-1} (v^{\le t-1})^2 \right] \le \frac{1}{\alpha}, \ \forall t \ge 1,
\end{align*}
where $(v^t)_{1 \le t \le T} \sim_{\iid} \sN (0, 1)$ is independent of $v \sim \normal(0, q)$. We first show that the It\^{o} integral with respect to a family of simple adapted processes can be approximated by the $U$ defined above in $L^2$-metric to arbitrary accuracy.
\begin{lem}\label{lem:simple_process}
	Let $(B_t)_{0 \le t \le 1}$ be a standard Brownian motion independent of $v \sim \normal(0, q)$, and $F$ be a $Q$-contraction. Define the filtration $\{ \cF_t \}_{0 \le t \le 1}$ by
	\begin{equation*}
		\cF_t = \sigma \left( v, (B_s)_{0 \le s \le t} \right), \ 0 \le t \le 1.
	\end{equation*}
	Assume $0 = t_0 < t_1 < \cdots < t_T = 1$ is an arbitrary discretization of $[0, 1]$, $\{ q(t_j) \}_{0 \le j \le T-1}$ is a sequence of scalars satisfying
	\begin{equation*}
		\sum_{j=0}^{T-1} q(t_j)^2 (t_{j+1} - t_j) = 1 - q,
	\end{equation*}
	$\{ \phi_{t_j} \}_{0 \le j \le T-1}$ is a sequence of random variables adapted to $\{ \cF_{t_j} \}_{0 \le j \le T-1}$, and that
	\begin{equation*}
		\E \left[ \phi_{t_j}^2 \right] \le \frac{1}{\alpha}, \ \forall j = 0, \cdots, T-1.
	\end{equation*}
	Then, $\operatorname{Law} (U) \in \cuF_{1, \alpha}^{\salg}$, where
	\begin{equation*}
		U = v + \frac{1}{\alpha} F(v) + \sum_{j=0}^{T-1} q(t_j) (1 + \phi_{t_j}) (B_{t_{j+1}} - B_{t_j}).
	\end{equation*}
\end{lem}

\begin{proof}
	Without loss of generality, we may assume that each $t_j = j/T$. Otherwise, one can always reparametrize the time argument. Under this simplification, we have $\forall k \in \mathbb{N}$:
	\begin{align*}
		U = \, & v + \frac{1}{\alpha} F(v) + \sum_{j=0}^{T-1} q(j/T) (1 + \phi_{j/T}) (B_{(j+1)/T} - B_{j/T}) \\
		= \, & v + \frac{1}{\alpha} F(v) + \sum_{j=0}^{T-1} \sum_{i=1}^{2^k} q(j/T) (1 + \phi_{j/T}) (B_{(j 2^k + i)/2^k T} - B_{(j 2^k + i - 1)/2^k T}) \\
		= \, & v + \frac{1}{\alpha} F(v) + \sum_{l=0}^{2^k T-1} q(l/2^k T, k) (1 + \phi_{l/2^k T, k}) (B_{(l+1)/2^k T} - B_{l/2^k T}),
	\end{align*}
	where $q(l/2^k T, k) = q(j/T)$, $\phi_{l/2^k T, k} = \phi_{j/T}$ if $l = j 2^k + i - 1$ for some $1 \le i \le 2^k$. We further notice that the sequences $\{ q(l/2^k T, k) \}_{0 \le l \le 2^k T - 1}$ and $\{ \phi_{l/2^k T, k} \}_{0 \le l \le 2^k T - 1}$ satisfy the conditions in the statement of Lemma~\ref{lem:simple_process}, since $\{ l / 2^k T \}_{0 \le l \le 2^k T}$ is just a more refined discretization of $[0, 1]$. Now for each $1 \le l \le 2^k T$, define the $\sigma$-algebra 
	\begin{equation*}
		\cF_{l/2^k T, k} = \sigma \left( v, (B_{r/2^k T})_{1 \le r \le l} \right) = \sigma \left( v, (B_{r/2^k T} - B_{(r-1)/2^k T})_{1 \le r \le l} \right),
	\end{equation*}
	it then follows that
	\begin{equation*}
		\E \left[ \phi_{l/2^k T, k} \vert \cF_{l/2^k T, k} \right] = \E \left[ \phi_{j/T} \vert \cF_{l/2^k T, k} \right] = \E \left[ \phi_{j/T} \vert \cF_{j/T, k} \right].
	\end{equation*}
	According to Paul L\'{e}vy's construction of Brownian motion (cf. Chapter 1 of \cite{morters2010brownian}), we know that $\cF_{j/T, k} \uparrow \cF_{j/T}$ as $k \to \infty$. Since $\E [\phi_{j/T}^2] \le 1/\alpha$, we know that $\phi_{j/T}$ is integrable. Applying L\'{e}vy's upwards theorem yields
	\begin{equation*}
		\lim_{k \to \infty} \E \left[ \phi_{j/T} \vert \cF_{j/T, k} \right] = \E \left[ \phi_{j/T} \vert \cF_{j/T} \right] = \phi_{j/T}, \ \mbox{almost surely and in $L^2$}.
	\end{equation*}
	As a consequence, we deduce that
	\begin{equation*}
		U_{k} = v + \frac{1}{\alpha} F(v) + \sum_{l=0}^{2^k T-1} q(l/2^k T, k) \left( 1 + \E \left[ \phi_{l/2^k T, k} \vert \cF_{l/ 2^k T, k} \right] \right) (B_{(l+1)/2^k T} - B_{l/2^k T})
	\end{equation*}
	converges to $U$, almost surely and in $L^2$ as $k \to \infty$.
	
	We next approximate $U_{k}$ by a sequence of random variables whose distributions are in $\cuF_{1, \alpha}^{\salg}$. For future convenience, we denote $T_k = 2^k T$, and define for $1 \le t \le T_k$:
	\begin{align*}
		& v^t = \sqrt{T_k} \left( B_{t/T_k} - B_{(t-1)/T_k} \right), \ q_t = \frac{q((t-1)/T_k, k)}{\sqrt{T_k}}, \\ & \phi_{t-1} \left( v^{\le t-1}\right) = \E \left[ \phi_{(t-1)/T_k, k} \vert \cF_{(t-1)/ T_k, k} \right].
	\end{align*}
	Note that the above quantities also depend on $k$, but we supress this dependence to avoid heavy notation. We can then write
	\begin{align*}
		U_{k} & = v + \frac{1}{\alpha} F(v) + \sum_{t=1}^{T_k} \sqrt{T_k} q_t \left( 1 + \phi_{t-1} \left( v^{\le t-1} \right) \right) \frac{v^t}{\sqrt{T_k}} \\
		& = v + \frac{1}{\alpha} F(v) + \sum_{t=1}^{T_k} q_t \left( 1 + \phi_{t-1} \left( v^{\le t-1} \right) \right) v^t,
	\end{align*}
	where by Jensen's inequality,
	\begin{equation*}
		\E \left[ \phi_{t-1} \left( v^{\le t-1} \right)^2 \right] \le \E \left[ \phi_{(t-1)/T_k, k}^2 \right] \le \frac{1}{\alpha}.
	\end{equation*}
	
	We are now in position to complete the proof. Define
	\begin{align*}
		U_{k}^{(1)} = \, v + \frac{1}{\alpha} F(v) + \sum_{t=1}^{T_k} q_t \left( v^{t+1} + \phi_{t-1} \left( v^{\le t-1} \right) v^{t} \right),
	\end{align*}
	then we have the following estimate:
	\begin{align*}
		& \E \left[ \left( U_{k} - U_{k}^{(1)} \right)^2 \right] \\
		= \, & \E \left[ \left( \sum_{t=1}^{T_k - 1} (q_t - q_{t+1}) v^{t+1} + q_{T_k} v^{T_k+1} - q_1 v^1 \right)^2 \right] \\
		= \, & \sum_{t=1}^{T_k - 1} (q_t - q_{t+1})^2 + q_1^2 + q_{T_k}^2 \\
		= \, & \frac{1}{T_k} \left( \sum_{t=1}^{T_k - 1} (q((t-1)/T_k, k) - q(t/T_k, k))^2 + q(0, k)^2 + q((T_k-1)/T_k, k)^2 \right) \\
		\le \, & \frac{1}{T_k} \left( \sum_{j=0}^{T-1} \left( q(j/T) - q((j+1)/T) \right)^2 + q(0)^2 + q(1)^2 \right) \to 0 \ \text{as} \ k \to \infty,
	\end{align*}
	leading to $U_{k}^{(1)} - U_{k} \stackrel{p}{\to} 0$. As we've shown $U_{k} \stackrel{a.s.}{\to} U$, it follows that $U_{k}^{(1)} \stackrel{p}{\to} U$, which implies $\Law (U_{k}^{(1)}) \stackrel{w}{\Rightarrow} \Law (U)$ as $k \to \infty$. Now since $\E[ \phi_{t-1} ( v^{\le t-1} )^2 ] \le 1/\alpha$, and
	\begin{equation*}
		\sum_{t=1}^{T_k} q_t^2 = \frac{1}{T_k} \sum_{l=0}^{T_k - 1} q (l/T_k, k)^2 = \frac{1}{T} \sum_{j=0}^{T - 1} q (j/T)^2 = 1 - q
	\end{equation*}
	by our assumption, \cref{thm:FeasibleDiscrete} implies that $\operatorname{Law} (U_{k}^{(1)}) \in \cuF_{1, \alpha}^{\salg}$. Using again the fact that $\cuF_{1, \alpha}^{\salg}$ is closed under weak limits, we know that $\operatorname{Law} (U) \in \cuF_{1, \alpha}^{\salg}$.
\end{proof}

Next we move from simple adapted stochastic processes to general progressively measurable stochastic processes, thus completing the proof of Theorem~\ref{thm:general_amp_achievability} for the case $m=1$. 
\begin{proof}[Proof of Theorem \ref{thm:general_amp_achievability} for $m=1$]
	We prove this theorem via standard approximation arguments in stochastic analysis. 
	We will use several times the fact that $\cuF^{\salg}_{1,\alpha}$ is closed under weak convergence. First, note that if we define
	\begin{equation*}
		\widetilde{U} = v + \frac{1}{\alpha} F(v) + \int_{0}^{1} \widetilde{q}(t) \left( 1 + \phi_t \right) \d B_t
	\end{equation*}
	for another $\widetilde{q} \in L^2[0, 1]$ satisfying $\norm{\widetilde{q}}_{L^2}^2 = \norm{q}_{L^2}^2 = 1 - q$, then It\^{o}'s isometry implies
	\begin{align*}
		& \E \left[ \left( U - \widetilde{U} \right)^2 \right] = \int_{0}^{1} \E \left[ (q(t) - \widetilde{q}(t))^2 (1 + \phi_t)^2 \right] \d t \\
		=\, & \int_{0}^{1} (q(t) - \widetilde{q}(t))^2 \E \left[ (1 + \phi_t)^2 \right] \d t \le \left( 1 + \frac{1}{\sqrt{\alpha}} \right)^2 \norm{q - \widetilde{q}}_{L^2}^2,
	\end{align*}
	where the last inequality follows from the fact $\E [\phi_t^2] \le 1 / \alpha$.
	Since $C[0, 1]$ is dense in $L^2 [0, 1]$, we know that $\{ \widetilde{U}: \widetilde{q} \in C[0, 1] \}$ is a dense subset of $\{ U: q \in L^2 [0, 1] \}$ in the space of $L^2$-integrable random variables, which further implies that $\{ \operatorname{Law} (\widetilde{U}): \widetilde{q} \in C[0, 1] \}$ is dense in $\{ \operatorname{Law} (U): q \in L^2[0, 1] \}$ under weak limit. Since $\cuF^{\salg}_{1,\alpha}$ is closed, we can assume without loss of generality that \smodif{$q \in C[0, 1]$}. Now for any $M > 0$, we define the truncated process
	\begin{equation*}
		\phi_t^M = \phi_t \bone_{\vert \phi_t \vert \le M}, \ \mbox{and} \ U^M = v + \frac{1}{\alpha} F(v) + \int_{0}^{1} q(t) \left( 1 + \phi_t^M \right) \d B_t,
	\end{equation*}
	then we obtain that
	\begin{equation*}
		\E \left[ \left( U - U^M \right)^2 \right] = \E \left[ \left( \int_{0}^{1} q(t) \phi_t \bone_{\vert \phi_t \vert > M} \d B_t \right)^2 \right] = \int_{0}^{1} q(t)^2 \E \left[ \phi_t^2 \bone_{\vert \phi_t \vert > M} \right] \d t \to 0
	\end{equation*}
	as $M \to \infty$ by bounded convergence theorem.
	Therefore, $\operatorname{Law} (U^M) \stackrel{w}{\Rightarrow} \operatorname{Law} (U)$ as $M \to \infty$.
	
	Now it suffices to consider $U^M$, namely assuming $q(t)$ is continuous and $\phi_t$ is bounded by $M$ without loss of generality. Note that by our assumption, the stochastic process $X_t^M = q(t) (1 + \phi_t^M)$ is bounded and progressively measurable. According to Lemma 2.4 and the discussion of Problem 2.5 in \cite{karatzas2012brownian}, $X_t^M$ can be arbitrarily approximated in $L^2([0, 1] \times \Omega)$ by a sequence of simple adapted processes as described in the statement of Lemma~\ref{lem:simple_process}. As a consequence, there exists a sequence $\{U_k^M\}_{k \in \mathbb{N}}$ such that for all $k \in \mathbb{N}$, $\operatorname{Law} ( U_k^M) \in \cuF_{1, \alpha}^{\salg}$, and that $U_k^M \stackrel{L^2}{\to} U^M$. Since $\cuF^{\salg}_{1,\alpha}$ is closed, we know that $\operatorname{Law} (U^M)$ is $(\alpha, 1)$-feasible, $\forall M > 0$. Letting $M \to \infty$ and using again the fact that $\cuF^{\salg}_{1,\alpha}$ is closed, it finally follows that $\operatorname{Law} (U) \in \cuF^{\salg}_{1,\alpha}$, completing the proof of \cref{thm:general_amp_achievability}.
\end{proof}

%
%
\section{Dual characterization and Parisi formula: Proof of Theorem \ref{thm:two_stage_strong_duality}}\label{sec:char_opt_ctrl}

This section will be devoted to the proof
of Theorem \ref{thm:two_stage_strong_duality}, thus yielding a dual characterization for 
$\VH_{1,\alpha}^{\sAMP}(h)$. While our final result is limited to the case $m=1$,
we will prove several intermediate results for general $m \ge 1$.

\modif{Recall that $\cuF^{\sAMP}_{m,\alpha}$ denotes the set of distributions achievable by our two-stage AMP algorithm, as described in \cref{thm:general_amp_achievability}.} We begin by defining a subset of $\cuF^{\sAMP}_{m,\alpha}$
corresponding to a fixed pair of matrix-valued parameters $(Q, \{ Q (t) \}_{0 \le t \le 1})$, satisfying 
\begin{equation*}
	\int_{0}^{1} Q(t) Q(t)^\top \d t = I_m - Q\, ,
\end{equation*}
\modif{and a fixed $Q$-contraction $F$.}
With an abuse of notation, we denote such a \modif{collection of parameters} by \modif{$(Q, F)$}, and define
the following set of probability distributions on $\R^m$:
\begin{equation}\label{eq:mean_0_integral}
	\begin{split}
		\cuF^{\sAMP}_{m, \alpha}(Q, \modif{F}) :=\, \bigg\{ & \operatorname{Law} \left( V + \frac{1}{\alpha} F(V) + \int_{0}^{1} Q(t) \left( I_m + \Phi_t \right) \d B_t \right): \\
		& \Phi_t \in D [0, 1], \ \mbox{and} \ \E \left[ \Phi_t \Phi_t^\top \right] \preceq \frac{I_m}{\alpha}, \ \forall t \in [0, 1] \bigg\}.
	\end{split}
\end{equation}

\modif{In what follows, we formally establish the dual relationship between $\cuF^{\sAMP}_{m,\alpha}$ and $\VH_{m, \alpha}^{\sAMP} (\cdot)$ for any fixed $(Q, F)$, and introduce the associated value function $V_{\gamma}$ for $m=1$, which will play a crucial role in the proof of \cref{thm:two_stage_strong_duality}.}

\subsection{A duality principle}\label{sec:lagrange_dual}

The proposition below establishes a dual characterization for $\close \cuF_{m, \alpha}^{\sAMP} (Q, F)$, the closure of $ \cuF_{m, \alpha}^{\sAMP} (Q, F)$ with respect to the weak topology.
\begin{prop}\label{thm:convexity_fix_Q}
	$\close \cuF^{\sAMP}_{m, \alpha} (Q, F)$ is closed under weak limit and convex. As a consequence, for any $\mu \in \mathscr{P} (\R^m)$, $\mu \in \close \cuF^{\sAMP}_{m, \alpha}(Q, F)$ if and only if $\forall h \in C_b (\R^m)$,
	\begin{equation*}
		\int_{\R^m} h \d \mu \le \sup_{\nu \in \close \cuF^{\sAMP}_{m, \alpha}(Q, F)} \left\{ \int_{\R^m} h \d \nu \right\} = \sup_{\nu \in \cuF^{\sAMP}_{m, \alpha}(Q, F)} \left\{ \int_{\R^m} h \d \nu \right\}.
	\end{equation*}
\end{prop}
\begin{proof}
	By definition, $\close \cuF^{\sAMP}_{m,\alpha}(Q, F)$ is automatically closed, next we show its convexity. Assume $P_1, P_2 \in \close \cuF^{\sAMP}_{m,\alpha}(Q, F)$, then there exist two sequences of probability measures $\{ P_{1, k} \}$ and $\{  P_{2, k} \}$ \modif{in $\cuF^{\sAMP}_{m,\alpha}(Q, F)$}, such that $P_{1, k} \stackrel{w}{\Rightarrow} P_1$, $\ P_{2, k} \stackrel{w}{\Rightarrow} P_2$ as $k \to \infty$. For each $\alpha \in [0, 1]$, we aim to prove
	\begin{equation*}
		\alpha P_1 + (1 - \alpha) P_2 \in \close \cuF^{\sAMP}_{m,\alpha}(Q, F).
	\end{equation*}
	Since $\alpha P_{1, k} + (1 - \alpha) P_{2, k} \stackrel{w}{\Rightarrow} \alpha P_1 + (1 - \alpha) P_2$, it suffices to show
	\begin{equation}\label{eq:convex_prop}
		\alpha P_{1, k} + (1 - \alpha) P_{2, k} \in \close \cuF^{\sAMP}_{m,\alpha}(Q, F), \ \forall k.
	\end{equation}
	Fix $k$, by definition of $\cuF^{\sAMP}_{m,\alpha}(Q, F)$, we can write
	\begin{equation*}
		P_{1, k} = \operatorname{Law} (U_1), \ P_{2, k} = \operatorname{Law} (U_2),
	\end{equation*}
	where
	\begin{align*}
		U_i =\, V + \frac{1}{\alpha} F(V) + \int_{0}^{1} Q(t) (I_m + \Phi_t^{(i)}) \d B_t, \, i = 1, 2.
	\end{align*}
	Let $0 < \veps < 1$, define a new standard Brownian motion $(W_t)_{0 \le t \le 1}$ by requiring that $(W_t)_{0 \le t \le \veps}$ is independent of $\cF_1 = \sigma(V, (B_t)_{0 \le t \le 1})$, and for $\veps \le t \le 1$,
	\begin{equation*}
		W_t = W_{\veps} + \sqrt{1 - \veps} B_{\frac{t-\veps}{1-\veps}}.
	\end{equation*}
	We further define 	$\widetilde{Q}(t) =0$ for $0\le t\le \eps$ and
	\begin{equation*}
		\widetilde{Q}(t) = \frac{1}{\sqrt{1 - \veps}} Q \left( \frac{t - \veps}{1 - \veps} \right), \ \widetilde{\Phi}_t^{(i)} = \Phi_{\frac{t - \veps}{1 - \veps}}^{(i)}, \ i = 1, 2.
	\end{equation*}
	Then, it follows that for $i = 1, 2$,
	\begin{align*}
		U_i =\, & V + \frac{1}{\alpha} F(V) + \int_{0}^{1} Q(t) (I_m + \Phi_t^{(i)}) \d B_t \\
		=\, & V + \frac{1}{\alpha} F(V) + \int_{\veps}^{1} Q \left( \frac{t - \veps}{1 - \veps} \right) \left( I_m + \Phi_{\frac{t - \veps}{1 - \veps}}^{(i)} \right) \d B_{\frac{t-\veps}{1-\veps}} \\
		=\, & V + \frac{1}{\alpha} F(V) + \int_{\veps}^{1} \frac{1}{\sqrt{1 - \veps}} Q \left( \frac{t - \veps}{1 - \veps} \right) \left( I_m + \Phi_{\frac{t - \veps}{1 - \veps}}^{(i)}\right) \sqrt{1 - \veps} \d B_{\frac{t-\veps}{1-\veps}} \\
		=\, & V + \frac{1}{\alpha} F(V) + \int_{\veps}^{1} \widetilde{Q} (t) (I_m + \widetilde{\Phi}_t^{(i)}) \d W_t.
	\end{align*}
	Note that there exists a Bernoulli random variable $T$ such that $T \in \cF_{\veps}^{W}$, and that
	\begin{equation*}
		\P (T = 1) = \alpha = 1 - \P(T = 0).
	\end{equation*}
	For example, we can take $T = \bone_{\norm{W_{\veps}}_2 > C_{\alpha}}$ where $\P(\norm{W_{\veps}}_2 > C_{\alpha}) = \alpha$. Therefore, $T$ is independent of $(U_1, U_2)$. Set $U = T U_1 + (1 - T) U_2$, then we have
	\begin{equation*}
		\operatorname{Law} (U) = \alpha \operatorname{Law} (U_1) + (1 - \alpha) \operatorname{Law} (U_2) = \alpha P_{1, k} + (1 - \alpha) P_{2, k},
	\end{equation*}
	and
	\begin{align*}
		U =\, & V + \frac{1}{\alpha} F(V) + \int_{\veps}^{1} \widetilde{Q} (t) \left( I_m + T \widetilde{\Phi}_t^{(1)} + (1 - T) \widetilde{\Phi}_t^{(2)} \right) \d W_t \\
		=\, & V + \frac{1}{\alpha} F(V) + \int_{0}^{1} \widetilde{Q} (t) \left( I_m + T \widetilde{\Phi}_t^{(1)} + (1 - T) \widetilde{\Phi}_t^{(2)} \right) \d W_t,
	\end{align*}
	where the last line is due to the fact that $\widetilde{Q}(t) = 0$ when $0 \le t \le \veps$. By definition of $\widetilde{Q}$, $\widetilde{\Phi}$ and $T$, we know that $\operatorname{Law} (U) \in \cuF^{\sAMP}_{m,\alpha}(\widetilde{Q}, F)$. Moreover, if we denote
	\begin{equation*}
		U' = V + \frac{1}{\alpha} F(V) + \int_{0}^{1} Q (t) \left( I_m + T \widetilde{\Phi}_t^{(1)} + (1 - T) \widetilde{\Phi}_t^{(2)} \right) \d W_t,
	\end{equation*}
	then $\Law(U') \in \cuF^{\sAMP}_{m,\alpha}(Q, F)$. \modif{We next show that $\E [\norm{U' - U}_2^2] \to 0$ as $\veps \to 0$. Using It\^{o}'s isometry, and the fact that $\widetilde{\Phi}_t^{(1)}$ and $\widetilde{\Phi}_t^{(2)}$ have uniformly bounded second moments, it suffices to prove
		\begin{align*}
			\lim_{\veps \to 0} \, \norm{\widetilde{Q} - Q}_{L^2 [0, 1]}^2  = \, 0.
		\end{align*}
		Since $\widetilde{Q}(t) = Q ( (t - \veps) / (1 - \veps) ) / \sqrt{1 - \veps}$, the above equation holds for any continuous $Q$. We can then use continuous matrix-valued functions to approximate general $Q \in L^2 ([0, 1] \to \R^{m \times m})$ to conclude the proof for $Q$.
	}
	This finally implies that $\operatorname{Law} (U) \in \close \cuF^{\sAMP}_{m,\alpha}(Q, F)$, thus proving Eq.~\eqref{eq:convex_prop}, and the desired result follows immediately.
\end{proof}

\cref{thm:convexity_fix_Q} shows that,
characterizing $\close \cuF^{\sAMP}_{m, \alpha}(Q, F)$ is equivalent to computing the following quantity for each $h \in C_b(\R^m)$:
\begin{equation}\label{eq:dual_value_fixed_Q_F}
	\begin{split}
		\VH_{m,\alpha}^{\sAMP}(Q, F, h) := \, & \sup_{\Phi \in D[0, 1]} \E \left[ h \left( V + \frac{1}{\alpha} F(V) + \int_{0}^{1} Q(t) (I_m + \Phi_t) \d B_t \right) \right], \\
		\, & \text{subject to} \,\, \E \left[ \Phi_t \Phi_t^\top \right] \preceq \frac{I_m}{\alpha}, \ \forall t \in [0, 1].
	\end{split}
\end{equation}
The above constrained optimization problem can be transformed  into an unconstrained one using the method of Lagrange multipliers, based on the following theorem:
\begin{thm}[Thm. 2.9.2 in \cite{zalinescu2002convex}]\label{thm:strong_duality}
	Let $\cuX$ and $\cuY$ be two topological vector spaces, where $\cuY$ is ordered by a closed convex cone $\cuC \subset \cuY$ (namely $y_1\ge_{\cuC} y_2$ if and only if
	$y_1-y_2\in \cuC$). Assume that $f$ is a proper convex function on $\cuX$, and $H: \cuX \rightarrow (\modif{\cuY \cup \{ \infty \}}, \cuC)$ is a $\cuC$-convex map, i.e., $H((1-\lambda)x_1+\lambda x_2) \le_\cuC (1-\lambda)H(x_1)+\lambda H(x_2)$. 
	Define the following primal optimization problem, whose value we denote by $v(P_0)$:
	\begin{equation}\label{eq:primal_opt}
		\begin{split}
			\mbox{\rm minimize} \ f(x), \quad \mbox{\rm subject to} \ H(x) \le_{\cuC} 0\, ,
		\end{split}
		\tag{$P_0$}
	\end{equation}
	and the Lagrange function
	\begin{equation}\label{eq:lagrange_fct}
		L: \cuX \times \cuC^{+} \rightarrow \overline{\mathbb{R}}, \quad L\left(x, y^{*}\right):= \begin{cases}f(x)+\left\langle H(x), y^{*}\right\rangle & \text { if } x \in \operatorname{dom} H, \\ \infty & \text { if } x \notin \operatorname{dom} H, \end{cases}
	\end{equation}
	where \smodif{$\operatorname{dom} H = \{ x \in \cuX: H(x) < \infty \}$}, and $\cuC^{+}$ is the dual cone of $\cuC$. Moreover, we define the dual problem of $(P_0)$ (whose value we denote by $v(D_0)$) as
	\begin{equation}\label{eq:dual_opt}
		\begin{split}
			\mbox{\rm maximize} \ \inf_{x \in \cuX} L (x, y^*), \quad \mbox{\rm subject to} \ y^* \in \cuC^{+}. 
		\end{split}\tag{$D_0$}
	\end{equation}
	Suppose that the following Slater's condition holds:
	\begin{equation}\label{eq:slater}
		\exists x_0 \in \operatorname{dom} f: \ - H(x_0) \in \operatorname{int} \cuC. 
	\end{equation}
	Then the problem $(D_0)$ has optimal solutions and $v(P_0) = v(D_0)$, i.e., there exists $\overline{y}^* \in \cuC^+$ such that
	\begin{equation}\label{eq:optimality}
		\inf \left\{f(x) \mid H(x) \leq_{\cuC} 0\right\}=\inf \left\{L\left(x, \overline{y}^{*}\right) \mid x \in \cuX \right\}.
	\end{equation}
	Furthermore, the following statements are equivalent for any $\overline{x} \in \operatorname{dom} f$:
	\begin{enumerate}
		\item[(i)] $\overline{x}$ is a solution of $( P_{0} )$,
		\item[(ii)] $H(\overline{x}) \le_{\cuC} 0$ and there exists $\overline{y}^{*} \in \cuC^{+}$ such that
		\begin{equation*}
			0 \in \partial\left(f+\overline{y}^{*} \circ H\right)(\overline{x}) \quad \text { and } \quad\left\langle H(\overline{x}), \overline{y}^{*}\right\rangle=0,
		\end{equation*}
		\smodif{where for two functions $g$ and $h$, $\partial g$ denotes the subdifferential of $g$ and $g \circ h$ represents the function composition of $g$ and $h$.}
		\item[(iii)] There exists $\overline{y}^{*} \in \cuC^{+}$ such that $(\overline{x}, \overline{y}^{*} )$ is a saddle point for $L$, i.e.,
		\begin{equation*}
			\forall x \in \cuX, \ \forall y^{*} \in \cuC^{+}: L\left(\overline{x}, y^{*}\right) \leq L\left(\overline{x}, \overline{y}^{*}\right) \leq L\left(x, \overline{y}^{*}\right).
		\end{equation*}
	\end{enumerate}
\end{thm}

Let $\cuX$ be $\{ \operatorname{Law} (V, \{ \Phi_t \}_{0 \le t \le 1}): \{ \Phi_t \} \in D[0, 1] \}$ equipped with the weak topology, $\cuY = L^1 ([0, 1], \R^{m \times m})$ and
$\cuC = \{ \Gamma \in \cuY: \Gamma(t) \in \S_+^m, \mbox{ a.e. } t \in [0, 1] \}$. We further define for $x \in \cuX$:
\begin{align*}
	g(x) =\, & \E \left[ h \left( V + \frac{1}{\alpha} F(V) + \int_{0}^{1} Q(t) (I_m + \Phi_t) \d B_t \right) \right], \\
	H(x) =\, & \E \left[ \Phi_t \Phi_t^\top \right] - \frac{I_m}{\alpha} \in \cuY.
\end{align*}
Then, the stochastic optimal control problem~\eqref{eq:dual_value_fixed_Q_F} can be re-written as \eqref{eq:primal_opt} with $f = - g$. Since $f$ is a linear functional, $H$ is a $\cuC$-convex operator, and Slater's condition~\eqref{eq:slater} holds by definition, we can apply \cref{thm:strong_duality} to deduce that
\begin{equation}\label{eq:strong_duality}
	\begin{split}
		\VH_{m,\alpha}^{\sAMP}(Q, F, h) = \, \inf_{\Gamma \in \cuC^{+}} \sup_{\Phi \in D[0, 1]} \E \bigg[ \, & h \left(V + \frac{1}{\alpha} F(V) + \int_{0}^{1} Q(t) (I_m + \Phi_t) \d B_t \right) \\
		& - \frac{1}{2} \int_{0}^{1} \left\langle \Gamma(t), \Phi_t \Phi_t^\top - \frac{I_m}{\alpha} \right\rangle \d t \bigg],
	\end{split}
\end{equation}
where
\begin{equation*}
	\cuC^{+} = \left\{ \Gamma \in L^{\infty} ([0, 1], \R^{m \times m}): \,\,\Gamma(t) \in \S_+^m, 
	\mbox{ a.e. } t \in [0, 1] \right\}.
\end{equation*}
Further, the infimum in \cref{eq:strong_duality} is achieved at some $\overline{\Gamma} \in \cuC^{+}$.

\paragraph*{Reduction for the case $m = 1$}
\label{sec:one_dim_reduction}
Similar to \cref{prop:feasible_simple_1dim}, the function parameter $\{ q(t) \}_{t \in [0, 1]}$ in \cref{eq:dual_value_fixed_Q_F} can be eliminated when $m=1$, leading to the simplified definition:
\begin{equation}\label{eq:dual_value_fixed_Q_F_1dim}
	\begin{split}
		\VH_{1,\alpha}^{\sAMP}(q, F, h) := \, & \sup_{\phi \in D[q, 1]} \E \left[ h \left( v + \frac{1}{\alpha} F(v) + \int_{q}^{1} (1 + \phi_t) \d B_t \right) \right], \\
		\, & \text{subject to} \ \sup_{t \in [q, 1]} \E \left[ \phi_t^2 \right] \le \frac{1}{\alpha}.
	\end{split}
\end{equation}
Under this simplification, the dual characterization~\eqref{eq:strong_duality} reduces to
\begin{equation}\label{eq:strong_duality_m=1}
	\begin{split}
		\VH_{1,\alpha}^{\sAMP}(q, F, h) 
		= \, \inf_{\gamma \in L_+^{\infty} [0, 1] } \sup_{\phi \in D[q, 1]} \E \bigg[ \, & h \left( v + \frac{1}{\alpha} F(v) + \int_{q}^{1} (1 + \phi_t) \d B_t \right) \\
		& - \frac{1}{2} \int_{q}^{1} \gamma(t) \left( \phi_t^2 - \frac{1}{\alpha} \right) \d t \bigg],
	\end{split}
\end{equation}
where
\begin{equation*}
	L_+^{\infty} [0, 1] := \left\{ \gamma \in L^{\infty} [0, 1]: \gamma(t) \ge 0, \text{ a.e. } t \in [0, 1] \right\}.
\end{equation*}
For any $\gamma \in L_+^{\infty} [0, 1]$, define
\begin{equation}\label{eq:def_VAL_gamma}
	V_\gamma(q) := \sup_{\phi \in D[q, 1]} \E \left[ h \left( v + \frac{1}{\alpha} F(v) + \int_{q}^{1} (1 + \phi_t) \d B_t \right) - \frac{1}{2} \int_{q}^{1} \gamma(t) \left( \phi_t^2 - \frac{1}{\alpha} \right) \d t \right].
\end{equation}
The following lemma shows that, in order to compute $V_{\gamma} (q)$, it suffices to consider the same quantity with $v + F(v) / \alpha$ replaced by a deterministic value.
\begin{lem}\label{lem:reduce_value_function}
	Define for any $(t, z) \in [0, 1] \times \R$, the following value function:
	\begin{equation}\label{eq:redefine_value_func}
		V_{\gamma} (t,z) = \sup_{\phi \in D [t, 1]} \E \left[ h \left( z+ \int_{t}^{1} \left( 1 + \phi_s \right) \d B_s \right) - \frac{1}{2} \int_{t}^{1} \gamma(s) \left( \phi_s^2 - \frac{1}{\alpha} \right) \d s \right],
	\end{equation}
	then $V_{\gamma}(q) = \E_{v \sim \normal (0, q)} [V_\gamma(q,v+F(v) / \alpha)]$.
\end{lem}
\modif{The proof of \cref{lem:reduce_value_function} uses standard arguments in probability theory, so we omit it for brevity. In \cref{prop:two_stage_veri_arg} below, we employ tools from stochastic optimal control---specifically the Hamilton-Jacobi-Bellman (HJB) equation and verification argument---to compute $V_{\gamma} (q, z)$ for $\gamma \in \FSG$.}

%
%

\subsection{Proof of Theorem~\ref{thm:two_stage_strong_duality}}\label{sec:proof_general_q}
In this section we present the proof of Theorem~\ref{thm:two_stage_strong_duality}. This proof is based on some key propositions, whose proofs are deferred to subsequent sections.

For the proof of parts $(a)$ and $(b)$, we need the following proposition regarding the dual relationship between $V_{\gamma}$ and $f_{\mu}$, the solution to the Parisi PDE. The proof of \cref{prop:two_stage_veri_arg} is presented in \cref{sec:verification_argument}.
\begin{prop}\label{prop:two_stage_veri_arg}
	Recall the value function $V_{\gamma}$ defined in \cref{eq:redefine_value_func} and that $f_{\mu}$ denotes the unique solution of \cref{eq:parisi_mu_gen}, as established in Theorem \ref{thm:solve_Parisi_PDE}. Under the conditions of \cref{thm:two_stage_strong_duality}, we have for all $t \in [0, 1]$ and $x, z \in \R$:
	\begin{equation}\label{eq:two_stage_V_and_f}
		\begin{split}
			V_{\gamma} (t, z) =\, & \inf_{x \in \R} \left\{ f_{\mu} (t, x) + \frac{\gamma(t)}{2} (x - z)^2 \right\} + \frac{1}{2 \alpha} \int_{t}^{1} \gamma(s) \d s, \\
			f_{\mu} (t, x) =\, & \sup_{z \in \R} \left\{ V_{\gamma} (t, z) - \frac{\gamma(t)}{2} (z - x)^2 \right\} - \frac{1}{2 \alpha} \int_{t}^{1} \gamma(s) \d s.
		\end{split}
	\end{equation}
\end{prop}

\vspace{0.5em}

\noindent \textbf{Proof of $(a)$: Variational formula.}  
For any fixed $F: \R \to \R$ and $v \in \R$, by definition of $V_{\gamma}$ in Eq.~\eqref{eq:redefine_value_func}, we have
\begin{equation}\label{eq:verif_rep_q}
	\begin{split}
		& V_{\gamma} \left( q, v + \frac{1}{\alpha} F(v) \right) \\
		= \, & \sup_{\phi \in D[q, 1]} \E \left[ h \left( v + \frac{1}{\alpha} F(v) + \int_{q}^{1} \left( 1 + \phi_t \right) \d B_t \right) - \frac{1}{2} \int_{q}^{1} \gamma(t) \left( \phi_t^2 - \frac{1}{\alpha} \right) \d t \right].
	\end{split}
\end{equation}
According to \cref{prop:two_stage_veri_arg}, we know that
\begin{align*}
	V_{\gamma} \left( q, v + u \right) = \, & \inf_{x \in \R} \left\{ f_{\mu} (q, x) + \frac{\gamma(q)}{2} \left( x - v - u \right)^2 \right\} + \frac{1}{2 \alpha} \int_{q}^{1} \gamma(s) \d s, \\
	f_{\mu} (q, v) =\, & \sup_{z \in \R} \left\{ V_{\gamma} (q, z) - \frac{\gamma(q)}{2} (z - v)^2 \right\} - \frac{1}{2 \alpha} \int_{q}^{1} \gamma(s) \d s.
\end{align*}
Now since $\mu = 0$ on $[0, q]$, the Parisi PDE degenerates to a standard heat equation:
\begin{equation*}
	\partial_t f_\mu (t, x) + \frac{1}{2} \partial_x^2 f_\mu (t, x) = \, 0, \ (t, x) \in [0, q] \times \R.
\end{equation*}
As a consequence, we deduce that
\begin{align*}
	& f_{\mu} (0, 0) = \E_{v \sim \normal(0, q)} \left[ f_{\mu} (q, v) \right] \\
	= \, & \E_{v \sim \normal(0, q)} \left[ \sup_{z \in \R} \left\{ V_{\gamma} (q, z) - \frac{\gamma(q)}{2} (z - v)^2 \right\} \right] - \frac{1}{2 \alpha} \int_{q}^{1} \gamma(s) \d s,
\end{align*}
which further implies that
\begin{align*}
	\mathsf{F} (\mu, c) = \, & f_{\mu} (0, 0) + \frac{1}{2 \alpha} \left( q \gamma(q) + \int_{q}^{1} \gamma(s) \d s \right) \\
	= \, & \E_{v \sim \normal(0, q)} \left[ \sup_{z \in \R} \left\{ V_{\gamma} (q, z) - \frac{\gamma(q)}{2} \left( (z - v)^2 - \frac{q}{\alpha} \right) \right\} \right] \\
	= \, & \E_{v \sim \normal(0, q)} \left[ \sup_{u \in \R} \left\{ V_{\gamma} (q, v + u) - \frac{\gamma(q)}{2} \left( u^2 - \frac{q}{\alpha} \right) \right\} \right] \\
	= \, & \sup_{F: \R \to \R} \E_{v \sim \normal(0, q)} \left[ V_{\gamma} \left( q, v + \frac{1}{\alpha} F(v) \right) - \frac{\gamma(q)}{2 \alpha} \left( \frac{F(v)^2}{\alpha} - q \right) \right].
\end{align*}
By \cref{lem:reduce_value_function}, we know that for any fixed $F$:
\begin{align*}
	& \E_{v \sim \normal(0, q)} \left[ V_{\gamma} \left( q, v + \frac{1}{\alpha} F(v) \right) \right] \\
	= \, & \sup_{\phi \in D[q, 1]} \E \left[ h \left( v + \frac{1}{\alpha} F(v) + \int_{q}^{1} (1 + \phi_t) \d B_t \right) - \frac{1}{2} \int_{q}^{1} \gamma(t) \left( \phi_t^2 - \frac{1}{\alpha} \right) \d t \right],
\end{align*}
which concludes the proof of part $(a)$.

\vspace{0.5em}

\noindent \textbf{Proof of $(b)$: Weak duality.} By definition of $\VH_{1,\alpha}^{\sAMP} (q,h)$, we deduce that
\begin{align*}
	\VH_{1,\alpha}^{\sAMP} (q,h) \le \, & \sup_{\begin{subarray}{c}
			F: \R \to \R, \, \E [F(v)^2] \le \, \alpha q \\
			\phi \in D[q, 1], \, \sup_{t \in [q, 1]} \E [\phi_t^2] \le \, 1/\alpha
	\end{subarray}} \E \left[ h \left( v + \frac{1}{\alpha} F(v) + \int_{q}^{1} \left( 1 + \phi_t \right) \d B_t \right) \right] \\
	\le \, & \sup_{\substack{F: \R \to \R \\ \phi \in D[q,1]}} \inf_{\gamma \in \FSG (q)} \E \bigg[ \, h \left( v + \frac{1}{\alpha} F(v) + \int_{q}^{1} \left( 1 + \phi_t \right) \d B_t \right) - \frac{1}{2} \int_{q}^{1} \gamma(t) \left( \phi_t^2 - \frac{1}{\alpha} \right) \d t \\
	& \quad\quad\quad\quad\quad\quad - \frac{\gamma(q)}{2 \alpha} \left( \frac{1}{\alpha} F(v)^2 - q \right) \bigg] \\
	\stackrel{(i)}{\le} \, & \inf_{\gamma \in \FSG (q)} \sup_{\substack{F: \R \to \R \\ \phi \in D[q,1]}} \E \bigg[ \, h \left( v + \frac{1}{\alpha} F(v) + \int_{q}^{1} \left( 1 + \phi_t \right) \d B_t \right) - \frac{1}{2} \int_{q}^{1} \gamma(t) \left( \phi_t^2 - \frac{1}{\alpha} \right) \d t \\
	& \quad\quad\quad\quad\quad\quad - \frac{\gamma(q)}{2 \alpha} \left( \frac{1}{\alpha} F(v)^2 - q \right) \bigg] \\
	\stackrel{(ii)}{=} \, & \inf_{(\mu, c) \in \FS (q)} \mathsf{F} (\mu, c),
\end{align*}
where $(i)$ follows from minimax inequality, $(ii)$ follows from the variational formula of part $(a)$. This proves part $(b)$.

\vspace{0.5em}
\noindent \textbf{Proof of $(c)$: Strong duality.} For the sake of simplicity, in this section we only consider the case $1 / c_* > \sup_{z \in \R} h''(z)$. The proof for $1 / c_* \le \sup_{z \in \R} h''(z)$ is technically more complicated but not substantially different, and we defer it to \cref{sec:less_than_case}. We will need the following proposition regarding the first-order variation of the Parisi functional $\mathsf{F} (\mu, c)$, whose proof is presented in \cref{sec:greater_than_case}.
\begin{prop}\label{prop:first_order_var_Parisi}
	Under the settings of \cref{thm:two_stage_strong_duality}, for any $(\mu, c) \in \FS (q)$ such that $1/c > \sup_{z \in \R} h''(z)$, let $(X_t)_{t \in [0, 1]}$ solve the SDE (existence and uniqueness of solution will be proved in \cref{sec:verification_argument}):
	\begin{equation}\label{eq:two_stage_SDE}
		X_0 = 0, \quad \d X_t = \mu(t) \partial_x f_{\mu} (t, X_t) \d t + \d B_t, \ t \in [0, 1].
	\end{equation}
	Let $\gamma \in \FSG (q)$ be such that $\gamma' / \gamma^2 = \mu$, $\gamma(1) = 1/c$, and define
	\begin{equation}\label{eq:two_stage_opt_control}
		F(x) = \frac{\alpha}{\gamma(q)} \partial_x f_{\mu} (q, x), \ \phi_t = \frac{1}{\gamma(t)} \partial_x^2 f_{\mu} (t, X_t), \ \forall t \in [0, 1].
	\end{equation}
	Then, we have
	\begin{itemize}
		\item [(i)] $X_q = B_q \sim \normal(0, q)$, and
		\begin{equation}\label{eq:var_rep_F_opt}
			\begin{split}
				\mathsf{F} (\mu, c) = \, \E \bigg[ \, & h \left( X_q + \frac{1}{\alpha} F(X_q) + \int_{q}^{1} \left( 1 + \phi_t \right) \d B_t \right) - \frac{1}{2} \int_{q}^{1} \gamma(t) \left( \phi_t^2 - \frac{1}{\alpha} \right) \d t \\
				& - \frac{\gamma(q)}{2 \alpha} \left( \frac{1}{\alpha} F(X_q)^2 - q \right) \bigg].
			\end{split}
		\end{equation}
		\item [(ii)] $\forall 0 \le s < t \le 1$, $\E [(\partial_x f_{\mu} (t, X_t))^2] - \E [(\partial_x f_{\mu} (s, X_s))^2] = \int_{s}^{t} \gamma(u)^2 \E [\phi_u^2] \d u$. \modif{Further, the mapping $u \mapsto \E [\phi_u^2]$ is continuous in $u$.} 
		
		\item [(iii)] Assume that $\delta: [0, 1] \to \R$ is in $L^1 [0, 1]$, \smodif{$\delta \vert_{[0, t]} \in L^{\infty} [0, t]$ for any $t \in [0, 1)$}, \modif{and furthermore,} $\delta \equiv 0$ on $[0, q]$. Then, $(\mu + s \delta, c) \in \mathscr{L} (q)$ for sufficiently small $s \in \R$, and
		\begin{equation}\label{eq:first_variation_statement}
			\frac{\d}{\d s} \mathsf{F} \left( \mu + s \delta, c \right) \bigg\vert_{s=0} = \,  \frac{1}{2} \int_{q}^{1} \delta(t) \left( \E \left[ \left( \partial_x f_\mu (t, X_t) \right)^2 \right] - \frac{1}{\alpha} \int_{0}^{t} \gamma(s)^2 \d s \right) \d t.
		\end{equation}
	\end{itemize}
\end{prop}
We are now in position to complete the proof of part $(c)$. Since the infimum of $\sF$ is achieved at $(\mu_*, c_*)$, the first-order variation of $\sF$ at $\mu_*$ must be equal to $0$ for any $\delta \in L^1 [0, 1]$ such that  $\delta \vert_{[0, t]} \in L^{\infty} [0, t]$ for any $t \in [0, 1)$, and $\delta = 0$ on $[0, q]$. According to \cref{prop:first_order_var_Parisi} $(iii)$, we must have
\begin{equation*}
	\frac{1}{2} \int_{q}^{1} \delta(t) \left( \E \left[ \left( \partial_x f_{\mu_*} (t, X_t) \right)^2 \right] - \frac{1}{\alpha} \int_{0}^{t} \gamma_* (s)^2 \d s \right) \d t = 0
\end{equation*}
for all such $\delta$. Note that \cref{prop:first_order_var_Parisi} $(ii)$ implies that $\E [ ( \partial_x f_{\mu_*} (t, X_t) )^2 ]$ is continuous in $t$, we therefore deduce that
\begin{equation}\label{eq:two_stage_FOC}
	\E \left[ \left( \partial_x f_{\mu_*} (t, X_t) \right)^2 \right] = \frac{1}{\alpha} \int_{0}^{t} \gamma_* (s)^2 \d s, \ \forall t \in [q, 1].
\end{equation}
Now we define $(\phi_t^*)_{t \in [q, 1]}$ and $F^*$ according to \cref{eq:two_stage_opt_control}. Then from \cref{prop:first_order_var_Parisi} $(ii)$, we immediately know that
\begin{equation*}
	\E \left[ \left( \phi_t^* \right)^2 \right] = \frac{1}{\alpha}, \ \forall t \in [q, 1],
\end{equation*}
namely, $(\phi_t^*)_{t \in [q, 1]}$ is feasible. 
It suffices to show that $F^*$ is feasible, and
\begin{equation}\label{eq:two_stage_achievability}
	\mathsf{F} (\mu_*, c_*) = \, \E \left[ h \left( X_q + \frac{1}{\alpha} F^*(X_q) + \int_{q}^{1} \left( 1 + \phi_t^* \right) \d B_t \right) \right],
\end{equation}
since $X_q = B_q \sim \normal (0, q)$.

We first establish the feasibility of $F^*$, i.e.,
\begin{equation*}
	\E [F^*(v)^2] = \alpha q, \ \E [(F^*)'(v)^2] \le \alpha, \ v \sim \normal(0, q).
\end{equation*}
Note that since $X_q \sim \normal (0, q)$, we have
\begin{equation*}
	\E [F^*(v)^2] = \frac{\alpha^2}{\gamma_* (q)^2} \E \left[ \left( \partial_x f_{\mu_*} (q, X_q) \right)^2 \right] = \frac{\alpha}{\gamma_* (q)^2} \int_{0}^{q} \gamma_* (s)^2 \d s = \alpha q,
\end{equation*}
which follows from \cref{eq:two_stage_FOC} and the fact that $\gamma_*$ is constant on $[0, q]$. Further,
\begin{equation*}
	\E [(F^*)'(v)^2] = \frac{\alpha^2}{\gamma_* (q)^2} \E \left[ \left( \partial_x^2 f_{\mu_*} (q, X_q) \right)^2 \right] = \alpha^2 \E \left[ \left( \phi_q^* \right)^2 \right] = \alpha.
\end{equation*}
This proves that $F^*$ is feasible. \cref{eq:two_stage_achievability} then automatically follows 
from \cref{prop:first_order_var_Parisi} $(i)$. \smodif{Since $(F^*, \phi^*)$ is feasible, we know that
	\begin{equation}
		\VH_{1,\alpha}^{\sAMP} (q,h) \ge \, \E \left[ h \left( X_q + \frac{1}{\alpha} F^*(X_q) + \int_{q}^{1} \left( 1 + \phi_t^* \right) \d B_t \right) \right] = \mathsf{F} (\mu_*, c_*).
	\end{equation}
	Combining the above inequality with our conclusion from part $(b)$ yields \cref{eq:strong_dual}.
} This completes the proof of part $(c)$.

\section{Solving the Parisi PDE}\label{sec:solve_Parisi_PDE}

This section is devoted to the proof of \cref{thm:solve_Parisi_PDE}. We begin by dealing with the case $\sup_{z \in \R} h''(z) < 1/c$, in which we construct a solution to the Parisi PDE, establish its uniqueness and regularity, and prove \cref{eq:parisi_regularity,eq:parisi_additional_regularity}. We then consider the case $\sup_{z \in \R} h''(z) \ge 1/c$, and extend our results from the previous setting (except for \cref{eq:parisi_additional_regularity}) to this setting.

Recall the function spaces $\FS$ and $\FSG$ in Definition~\ref{defn:gamma_space}. In what follows, we define the convergence of sequences in these two spaces.
\begin{defn}[Convergence in $\FS$ and $\FSG$]\label{def:conv_in_L_gen}
	Let $\{(\mu_n, c_n)\}_{n=1}^{\infty}$ be a sequence in $\FS$, or equivalently, the corresponding $\{ \gamma_n \}_{n=1}^{\infty} \subset \FSG$. For any $(\mu, c) \in \FS$, we say that $(\mu_n, c_n) \stackrel{\FS}{\longrightarrow} (\mu, c)$ if $c_n \to c$, $\mu_n \to \mu$ in $L^1 [0, 1]$, and 
	\begin{equation*}
		\norm{\mu_n \big\vert_{[0, t]}}_{L^{\infty} [0, t]} \longrightarrow \norm{\mu \big\vert_{[0, t]}}_{L^{\infty} [0, t]} \ \text{for all} \ t \in [0, 1).
	\end{equation*}
	For $\gamma \in \FSG$ associated with $(\mu, c)$, we say $\gamma_n \stackrel{\FSG}{\longrightarrow} \gamma$ if $(\mu_n, c_n) \stackrel{\FS}{\longrightarrow} (\mu, c)$.
\end{defn}

\modif{
	\begin{rem}
		The convergence of sequence defined above can be metrized by the following metric on $\FS$:
		\begin{equation}\label{eq:metric_FS}
			d_{\FS} ((\mu, c), (\mu', c')) = \norm{\mu - \mu'}_{L^1 [0, 1]} + \vert c - c' \vert + \sum_{k \in \mathbb{N}} 2^{- k} \frac{| s_k(\mu)-s_k(\mu') | }{1+| s_k(\mu)-s_k(\mu') |},
		\end{equation}
		where $\{ t_k \}_{k \in \mathbb{N}}$ is a dense subset of $[0, 1]$, and $s_k (\mu) = \, \| \mu \vert_{[0, t_k]}\|_{L^{\infty} [0, t_k]}$ for $(\mu, c) \in \FS$ and $k \in \mathbb{N}$. Namely, $(\mu_n, c_n) \stackrel{\FS}{\longrightarrow} (\mu, c)$ if and only if $d_{\FS} ((\mu_n, c_n), (\mu, c)) \to 0$. 
		We note that $\FS$ is not complete in this metric. For instance, taking $\mu_n = \bone_{[0, 1/n]}$ and any $c > 0$, then $(\mu_n, c)$ is a Cauchy sequence without a limit in $\FS$.
	\end{rem}
}

\subsection{The case $\sup_{z \in \R} h''(z) < 1/c$}

We now prove \cref{thm:solve_Parisi_PDE} in the case $\sup_{z \in \R} h''(z) < 1/c = \gamma(1)$. We will first work under the stronger assumption that $h \in C^4 (\R)$ (instead of $h \in C^2 (\R)$ assumed by \cref{thm:solve_Parisi_PDE}), and then remove this assumption via an approximation argument. 
\begin{ass}\label{ass:h_regularity_gen}
	The test function $h: \R \to \R$ satisfies:
	\begin{itemize}
		\item [$(a)$] $h$ is bounded from above, i.e., $\sup_{x \in \R} h(x) < + \infty$.
		\item [$(b)$] $h \in C^{4} (\R)$. Further, for $1 \le k \le 4$,
		\begin{equation*}
			\norm{h^{(k)}}_{L^{\infty} (\R)} = \sup_{x \in \R} \left\vert h^{(k)} (x) \right\vert < \infty.
		\end{equation*}
	\end{itemize}
\end{ass}

To begin with, we establish that the terminal value $f_{\mu} (1, \cdot)$ has sufficient regularity under \cref{ass:h_regularity_gen}.

\begin{lem}\label{lem:terminal_regularity_gen}
	Assume $\sup_{z \in \R} h''(z) < \gamma(1)$, and $h$ satisfies \cref{ass:h_regularity_gen}. Recall that
	\begin{equation*}
		f_\mu (1, x) = \sup_{u \in \R} \left\{ h(x+u) - \frac{u^2}{2 c} \right\}.
	\end{equation*}
	Then, we have $\norm{\partial_x f_{\mu} (1, \cdot)}_{L^{\infty}(\R)} \le \norm{h'}_{L^{\infty} (\R)}$ and
	\begin{equation}\label{eq:curvature_bd_f_mu_1}
		-\gamma(1) < \partial_x^2 f_{\mu} (1, x) \le \frac{\gamma(1) \cdot \sup_{z \in \R} h''(z)}{\gamma(1) - \sup_{z \in \R} h''(z)}, \ \forall x \in \R.
	\end{equation}
	Further, for $k = 3, 4$, there exists constants $C = C(c, k) > 0$ such that
	\begin{equation*}
		\norm{\partial_x^k f_{\mu} (1, \cdot)}_{L^{\infty}(\R)} \le C(c, k).
	\end{equation*}
\end{lem}
\begin{proof}
	The claim on $\partial_x f_{\mu} (1, x)$ follows from the simple observation
	\begin{align*}
		\left\vert f_\mu (1, x) - f_\mu (1, y) \right\vert \le \sup_{u \in \R} \left\vert h(x+u) - h(y+u) \right\vert \le \norm{h'}_{L^{\infty} (\R)} \vert x - y \vert.
	\end{align*}
	To prove the estimates for higher-order derivatives, note that $f_{\mu} (1, \cdot)$ can be rewritten as
	\begin{align*}
		f_\mu (1, x) = \, & \sup_{u \in \R} \left\{ h(x+u) - \frac{u^2}{2 c} \right\} = \sup_{z \in \R} \left\{ h(z) - \frac{(z - x)^2}{2 c} \right\} \\
		= \, & - \frac{x^2}{2 c} + \sup_{z \in \R} \left\{ xz - \frac{c z^2}{2} + h(c z) \right\} \modif{= - \frac{x^2}{2 c} + g^* (x)},
	\end{align*}
	where we define $g(z) = cz^2/2 - h(cz)$, \modif{and $g^*$ denotes the Legendre-Fenchel transform (convex conjugate) of $g$.} It then follows that $g$ is $c_h$-strongly convex, where
	\begin{equation*}
		c_h = c - c^2 \sup_{z \in \R} h''(z) > 0.
	\end{equation*}
	Since $f_{\mu} (1, x) = - x^2 / 2c + g^* (x)$, the bounds on $\partial_x^2 f_{\mu} (1, x)$ follows immediately. Further, we have
	\begin{equation*}
		\max \left\{ \norm{g^{(3)}}_{L^{\infty} (\R)}, \norm{g^{(4)}}_{L^{\infty} (\R)} \right\} < \infty.
	\end{equation*}
	It then suffices to show that
	\begin{equation*}
		\max \left\{ \norm{(g^*)^{(3)}}_{L^{\infty} (\R)}, \norm{(g^*)^{(4)}}_{L^{\infty} (\R)} \right\} < \infty.
	\end{equation*}
	By Legendre-Fenchel duality, we have
	\begin{equation*}
		(g*)'' (x) = \frac{1}{g''(u (x))}, \ u = (g')^{-1} \implies u'(x) = \frac{1}{g''(u(x))}.
	\end{equation*}
	Since $g$ is $c_h$-strongly convex, we always have $g'' \ge c_h$. Therefore,
	\begin{align*}
		(g*)^{(3)} (x) = \, & - \frac{g^{(3)} (u(x)) u'(x)}{g''(u (x))^2} = - \frac{g^{(3)} (u(x))}{g''(u (x))^3} \in L^{\infty} (\R), \\
		(g*)^{(4)} (x) = \, & \frac{3 g^{(3)} (u(x))^2 g''(u(x))^2 u'(x) - g^{(4)} (u(x)) g''(u(x))^3 u'(x)}{g''(u (x))^6} \\
		= \, & \frac{3 g^{(3)} (u(x))^2}{g''(u (x))^5} - \frac{g^{(4)} (u(x))}{g''(u (x))^4} \in L^{\infty} (\R),
	\end{align*}
	which completes the proof.
\end{proof}

We next establish the regularity of $f_{\mu}$ on $[0, 1] \times \R$ for $\mu \in \mathsf{SF} [0, 1]$, the space of all simple functions on $[0, 1]$:
\begin{equation*}
	\mathsf{SF} [0, 1] = \left\{\mu(t)=\sum_{i=1}^m \mu_i \mathbf{1}_{\left[t_{i-1}, t_i\right)}(t): 0=t_0<t_1<\cdots<t_m=1 \right\}.
\end{equation*}

\begin{prop}\label{prop:curv_bd_gen}
	Let $h$ satisfy \cref{ass:h_regularity_gen}, and $\mu \in \mathsf{SF} [0, 1]$ be such that $(\mu, c) \in \FS$, i.e., $c + \int_{t}^{1} \mu(s) \d s > 0$ for all $t \in [0, 1]$. Assume that $\gamma(1) = 1/c > \sup_{z \in \R} h''(z)$. 
	Then,
	\begin{align}\label{eq:curv_bd_gen_1}
		\partial_x^2 f_{\mu} (t, x) > \, & - \gamma(t), \ \forall (t, x) \in [0, 1] \times \R\, .
	\end{align}
	Further, for any $\theta \in [0, 1)$ such that $\inf_{t \in [\theta, 1]} \gamma(t) > \sup_{z \in \R} h''(z)$, we have
	\begin{align}\label{eq:curv_bd_gen_2}
		\partial_x^2 f_{\mu} (t, x) \le \, & \frac{\gamma(t) \cdot \sup_{z \in \R} h''(z)}{\gamma(t) - \sup_{z \in \R} h''(z)}, \ \forall (t, x) \in [\theta, 1] \times \R.
	\end{align}
\end{prop}

\begin{proof}
	For future convenience we denote $h_2 = \sup_{z \in \R} h''(z)$. In what follows we will consider the following non-linear parabolic equation instead of the Parisi PDE:
	\begin{equation}\label{eq:Parisi_Phi_gen}
		\begin{split}
			&\partial_t \Phi (t, x) + \frac{\gamma'(t)}{2} \left( \partial_x \Phi(t, x) \right)^2 + \frac{\gamma(t)^2}{2} \partial_x^2 \Phi (t, x) = 0, \\
			&\Phi (1, x) = \left( \frac{\gamma(1)}{2} z^2 - h(z) \right)^*.
		\end{split}
	\end{equation}
	Note that the relation between $\Phi$ and $f_{\mu}$ is given by the following transform, which can be verified by direct calculation:
	\begin{equation}\label{eq:f_Phi_gen}
		\begin{split}
			f_{\mu} (t, x) = \, & - \frac{\gamma(t)}{2} x^2 + \Phi \left( t, \gamma(t) x \right) - \frac{1}{2} \int_{t}^{1} \gamma(s) \d s, \\
			\Phi(t, x) = \, & f_{\mu} \left( t, \frac{x}{\gamma(t)} \right) + \frac{x^2}{2 \gamma(t)} + \frac{1}{2} \int_{t}^{1} \gamma(s) \d s.
		\end{split}
	\end{equation}
	Further, Eqs.~\eqref{eq:curv_bd_gen_1}, \eqref{eq:curv_bd_gen_2} are equivalent to
	\begin{equation}\label{eq:curv_bd_Phi_gen}
		\begin{split}
			\partial_x^2 \Phi (t, x) > \, & 0, \ \forall (t, x) \in [0, 1] \times \R, \\
			\partial_x^2 \Phi (t, x) \le \, & \frac{1}{\gamma(t) - h_2}, \ \forall (t, x) \in [\theta, 1] \times \R.
		\end{split}
	\end{equation}
	It remains to prove the curvature bound on $\Phi (t, x)$, i.e., Eq.~\eqref{eq:curv_bd_Phi_gen}. Since $\gamma'(t) / \gamma(t)^2$ is piecewise constant, we may assume that
	\begin{equation*}
		- \left( \frac{1}{\gamma(t)} \right)' = \frac{\gamma'(t)}{\gamma(t)^2} = c_i \ \text{for} \ t \in [t_{i - 1}, t_i), \ i = 1, \cdots, m,
	\end{equation*}
	where $0 = t_0 < t_1 < \cdots < t_m = 1$ is a discretization of $[0, 1]$. \modif{We also note that \cref{eq:curvature_bd_f_mu_1} of Lemma~\ref{lem:terminal_regularity_gen} and \cref{eq:f_Phi_gen} together imply}
	\begin{equation*}
		0 < \partial_x^2 \Phi (1, x) \le \frac{1}{\gamma(1) - h_2}.
	\end{equation*}
	For $t \in [t_{m - 1}, 1)$, the Parisi PDE~\eqref{eq:Parisi_Phi_gen} reads
	\begin{equation*}
		\partial_t \Phi(t, x) + \frac{\gamma(t)^2}{2} \left( \partial_x^2 \Phi(t, x) + c_m \left( \partial_x \Phi(t, x) \right)^2 \right) = 0,
	\end{equation*}
	whose solution can be explicitly expressed using Cole-Hopf transform:
	\begin{equation*}
		\Phi (t, x) = \frac{1}{c_m} \log \E \left[ \exp \left( c_m \Phi \left( 1, x + \sqrt{\int_{t}^{1} \gamma(s)^2 \d s}  \cdot G \right) \right) \right], \ G \sim \sN (0, 1).
	\end{equation*}
	If $c_m = 0$, it is understood that
	\begin{equation*}
		\Phi (t, x) = \E \left[ \Phi \left( 1, x + \sqrt{\int_{t}^{1} \gamma(s)^2 \d s}  \cdot G \right) \right].
	\end{equation*}
	For simplicity, we denote $g (x) = \Phi(1, x)$, 
	$\kappa = c_m$ and $C(t, \gamma) = \sqrt{\int_{t}^{1} \gamma(s)^2 \d s}$, then exploiting the above expression yields that
	\begin{align*}
		\partial_x^2 \Phi (t, x)
		= \, & \frac{\kappa \E \left[ \exp \left( \kappa g \left( x + C(t, \gamma) G \right) \right) g' \left( x + C(t, \gamma) G \right)^2 \right]}{\E \left[ \exp \left( \kappa g \left( x + C(t, \gamma) G \right) \right) \right]} \\
		& + \frac{\E \left[ \exp \left( \kappa g \left( x + C(t, \gamma) G \right) \right) g'' \left( x + C(t, \gamma) G \right) \right]}{\E \left[ \exp \left( \kappa g \left( x + C(t, \gamma) G \right) \right) \right]} \\
		& - \frac{\kappa \E \left[ \exp \left( \kappa g \left( x + C(t, \gamma) G \right) \right) g' \left( x + C(t, \gamma) G \right) \right]^2}{\E \left[ \exp \left( \kappa g \left( x + C(t, \gamma) G \right) \right) \right]^2}.
	\end{align*}
	Note that applying Cauchy-Schwarz inequality gives
	\begin{align*}
		& \E \left[ \exp \left( \kappa g \left( x + C(t, \gamma) G \right) \right) g' \left( x + C(t, \gamma) G \right)^2 \right] \cdot \E \left[ \exp \left( \kappa g \left( x + C(t, \gamma) G \right) \right) \right] \\
		\ge \, & \E \left[ \exp \left( \kappa g \left( x + C(t, \gamma) G \right) \right) g' \left( x + C(t, \gamma) G \right) \right]^2.
	\end{align*}
	Hence, if $\kappa \ge 0$, we obtain that
	\begin{equation*}
		\partial_x^2 \Phi(t, x) \ge \frac{\E \left[ \exp \left( \kappa g \left( x + C(t, \gamma) G \right) \right) g'' \left( x + C(t, \gamma) G \right) \right]}{\E \left[ \exp \left( \kappa g \left( x + C(t, \gamma) G \right) \right) \right]} > 0,
	\end{equation*}
	where the last inequality follows from the fact that $g$ is strictly convex. 
	
	Next we assume that $\kappa < 0$, for any integrable test function $\phi$, we have
	\begin{align*}
		& \frac{\E \left[ \exp \left( \kappa g \left( x + C(t, \gamma) G \right) \right) \phi \left( x + C(t, \gamma) G \right) \right]}{\E \left[ \exp \left( \kappa g \left( x + C(t, \gamma) G \right) \right) \right]} \\
		= \, & \frac{1}{\E \left[ \exp \left( \kappa g \left( x + C(t, \gamma) G \right) \right) \right]} \int_{\R} \frac{1}{\sqrt{2 \pi} C(t, \gamma)} \exp \left( \kappa g(z) - \frac{(z - x)^2}{2 C(t, \gamma)^2} \right) \phi(z) \d z \\
		= \, & \int_{\R} p_{t, x} (z) \phi(z) \d z := \E_{t, x} \left[ \phi(Z) \right], 
	\end{align*}
	where we denote
	\begin{equation*}
		p_{t, x} (z) \propto \exp \left( \kappa g(z) - \frac{(z - x)^2}{2 C(t, \gamma)^2} \right) = \exp \left( - \frac{(z - x)^2}{2 C(t, \gamma)^2} - \vert \kappa \vert g(z) \right).
	\end{equation*}
	We further denote $\E_{t, x}$ and $\mathsf{Var}_{t, x}$ as the expectation and variance operator with respect to the density $p_{t, x}$, then it follows that
	\begin{align*}
		\partial_x^2 \Phi(t, x) = \, & \E_{t, x} \left[ g''(Z) \right] + \kappa \cdot \left( \E_{t, x} \left[ g'(Z)^2 \right] - \E_{t, x} \left[ g'(Z) \right]^2 \right) \\
		= \, & \E_{t, x} \left[ g''(Z) \right] - \vert \kappa \vert \cdot \mathsf{Var}_{t, x} \left( g'(Z) \right).
	\end{align*}
	Before proceeding with the proof, we recall a celebrated inequality by Brascamp and Lieb \cite{brascamp2002extensions}.
	\begin{thm}[Theorem 4.1 in \cite{brascamp2002extensions}: Poincar\'{e} inequality for log-concave densities]\label{thm:poincare_ineq}
		Let $p (z) = \exp (- u(z))$, where $u$ is twice continuously differentiable and strictly convex. Assume that $u$ has a minimum, so that $p$ decreases exponentially \modif{as $\vert x \vert \to +\infty$}, then $ \int_{\R} p (z) \d z < \infty $. \modif{Let $Z$ be a random variable with density function $p(z) / \int_{\R} p (z) \d z$.} Assume $\phi \in C^1 (\R)$ is such that $\mathsf{Var} (\phi(Z)) < \infty$. Then
		\begin{equation*}
			\mathsf{Var} (\phi(Z)) \le \E \left[ \frac{\phi'(Z)^2}{u''(Z)} \right].
		\end{equation*}
	\end{thm}
	Now we turn to the proof of strict convexity of $\Phi(t, x)$ when $\kappa < 0$. Note that  since $u(z) = \vert \kappa \vert g(z) + (z - x)^2 / 2 C(t, \gamma)^2$ is strongly convex, $p_{t, x} (z) \propto \exp (- u(z))$ is a log-concave density. According to Theorem~\ref{thm:poincare_ineq}, we have
	\begin{equation*}
		\mathsf{Var}_{t, x} \left( g'(Z) \right) \le \E_{t, x} \left[ \frac{g''(Z)^2}{u''(Z)} \right] = \E_{t, x} \left[ \frac{g''(Z)^2}{\vert \kappa \vert g''(Z) + 1 / C(t, \gamma)^2} \right] < \frac{1}{\vert \kappa \vert} \E_{t, x} \left[ g''(Z) \right],
	\end{equation*}
	which leads to
	\begin{equation*}
		\partial_x^2 \Phi(t, x) = \E_{t, x} \left[ g''(Z) \right] - \vert \kappa \vert \cdot \mathsf{Var}_{t, x} \left( g'(Z) \right) > 0.
	\end{equation*}
	\modif{In particular, we have proved that $\partial_x^2 \Phi (t_{m-1}, x) > 0$. Now, we can apply the same argument to $t \in [t_{m-2}, t_{m-1})$, viewing $\Phi(t_{m-1}, \cdot)$ as the new terminal condition, and conclude that $\partial_x^2 \Phi(t, x) > 0$ for all $(t, x) \in [t_{m-2}, t_{m-1}) \times \R$. Repeating this procedure (or using induction), it finally follows that $\partial_x^2 \Phi(t, x) > 0$ for all $(t, x) \in [0, 1] \times \R$.} 
	This proves the first part of Eq.~\eqref{eq:curv_bd_Phi_gen}.
	
	To prove the second part of Eq.~\eqref{eq:curv_bd_Phi_gen} \smodif{for $t \in [t_{m-1}, 1)$}, it suffices to show that
	\begin{equation*}
		\gamma(t) > h_2 \implies \partial_x^2 \Phi(t, x) \le \frac{1}{\gamma(t) - h_2}, \ \forall x \in \R,\, t\in [t_{m-1},1)\, .
	\end{equation*}
	Consider the situation $\kappa \ge 0$. In this case we define $p_{t, x} (z) \propto \exp (- u(z))$, where $u(z) = (z - x)^2 / 2 C(t, \gamma)^2 - \kappa g(z)$. Note that for all $z \in \R$,
	\begin{align*}
		u''(z) = \, & \frac{1}{C(t, \gamma)^2} - \kappa g''(z) \ge \frac{1}{C(t, \gamma)^2} - \frac{\kappa}{\gamma(1) - h_2} = \frac{1}{\int_{t}^{1} \gamma(s)^2 \d s} - \frac{\kappa}{\gamma(1) - h_2} \\
		= \, & \frac{\kappa}{\int_{t}^{1} \gamma'(s) \d s} - \frac{\kappa}{\gamma(1) - h_2} = \frac{\kappa}{\gamma(1) - \gamma(t)} - \frac{\kappa}{\gamma(1) - h_2} \\
		= \, & \frac{\kappa}{\gamma(1) - \gamma(t)} \cdot \frac{\gamma(t) - h_2}{\gamma(1) - h_2} \ge \frac{1}{(1 - t) \gamma(1)^2} \cdot \frac{\gamma(t) - h_2}{\gamma(1) - h_2} > 0,
	\end{align*}
	which implies that $u$ is strongly convex, \modif{since $\gamma(1), \gamma(t) > h_2$.} Therefore, applying the 
	Brascamp-Lieb inequality (\cref{thm:poincare_ineq}) again yields
	\begin{align*}
		\partial_x^2 \Phi(t, x) = \, & \E_{t, x} \left[ g''(Z) \right] + \kappa \cdot \mathsf{Var}_{t, x} \left( g'(Z) \right) \le \E_{t, x} \left[ g''(Z) \right] + \kappa \cdot \E_{t, x} \left[ \frac{g''(Z)^2}{u''(Z)} \right] \\
		= \, & \E_{t, x} \left[ \frac{g''(Z) \left( u''(Z) + \kappa g''(Z) \right)}{u''(Z)} \right] = \E_{t, x} \left[ \frac{g''(Z)}{C(t, \gamma)^2 u''(Z)} \right] \\
		= \, & \E_{t, x} \left[ \frac{g''(Z)}{1 - \kappa C(t, \gamma)^2 g''(Z)} \right] 
		\stackrel{(i)}{=}  \E_{t, x} \left[ \frac{g''(Z)}{1 - (\gamma(1) - \gamma(t)) g''(Z)} \right] \\
		\stackrel{(ii)}{\le} \, & \E_{t, x} \left[ \frac{1 / (\gamma(1) - h_2)}{1 - (\gamma(1) - \gamma(t)) / (\gamma(1) - h_2)} \right] = \frac{1}{\gamma(t) - h_2},
	\end{align*}
	where $(i)$ holds because $\kappa \gamma^2(t) = \gamma'(t)$  for $t\in [t_{m-1},1)$,
	whence $\kappa C(t,\gamma)^2 = \int_t^1\gamma'(s)\d s =\gamma(1)-\gamma(t)$,
	and $(ii)$ is due to the fact that $\sup_{z \in \R} \{ g''(z) \} \le 1 / (\gamma(1) - h_2)$. 
	
	Finally, we deal with the case $\kappa < 0$ in the proof of  Eq.~\eqref{eq:curv_bd_Phi_gen}. Note that without loss of generality we may assume that $\gamma(t) - h_2 \le (1 + 1 / \sqrt{2}) (\gamma(1) - h_2)$, otherwise it is always possible to find interpolating points $t = s_0 < s_1 < \cdots < s_{k - 1} < s_k = 1$ such that
	\begin{equation*}
		\gamma(s_{i - 1}) - h_2 \le (1 + 1 / \sqrt{2}) (\gamma(s_i) - h_2), \ \text{for all} \ i = 1, \cdots, k.
	\end{equation*}
	Given $\partial_x^2 \Phi(1, x) \le 1 / (\gamma(1) - h_2)$, if we can show that $\partial_x^2 \Phi(s_{k - 1}, x) \le 1 / (\gamma(s_{k - 1}) - h_2)$, then proceeding with the same argument we obtain that $\partial_x^2 \Phi(s_{k - 2}, x) \le 1 / (\gamma(s_{k - 2}) - h_2)$, and eventually $\partial_x^2 \Phi(t, x) \le 1 / (\gamma(t) - h_2)$. Hence, it suffices to prove that under the additional assumption $\gamma(t) - h_2 \le (1 + 1 / \sqrt{2}) (\gamma(1) - h_2)$, we have $\partial_x^2 \Phi(t, x) \le 1 / (\gamma(t) - h_2)$.
	
	For future convenience, let us denote
	\begin{equation*}
		b = \frac{1}{\gamma(t) - \gamma(1)}, \ d = \frac{1}{\gamma(1) - h_2},
	\end{equation*}
	then it follows that $0 \le g''(z) \le d$, and
	\begin{align*}
		u (z) = \vert \kappa \vert g(z) + \frac{(z - x)^2}{2 C(t, \gamma)^2} = \vert \kappa \vert g(z) + \frac{\vert \kappa \vert (z - x)^2}{2 (\gamma(t) - \gamma(1))} = \vert \kappa \vert \left( g(z) + \frac{b}{2} (z - x)^2 \right).
	\end{align*}
	To simplify the notation, we drop the subscripts from $p_{t, x}$, $\E_{t, x}$ and $\mathsf{Var}_{t, x}$ whenever no confusion arises. Since $p(z) \propto \exp (- u(z))$, by integration by parts we obtain that
	\begin{equation*}
		\E \left[ u'(Z) \right] = 0, \ \mathsf{Var} \left( u'(Z) \right) = \E \left[ u''(Z) \right],
	\end{equation*}
	which leads to
	\begin{align*}
		& \E \left[ g''(Z) \right] + b = \frac{1}{\vert \kappa \vert} \E \left[ u''(Z) \right] \\
		= \, & \frac{1}{\vert \kappa \vert} \mathsf{Var} \left( u'(Z) \right) = \vert \kappa \vert \cdot \mathsf{Var} \left( g'(Z) + b(Z -x) \right).
	\end{align*}
	Using Cauchy-Schwarz inequality and Theorem~\ref{thm:poincare_ineq}, we obtain the following estimate:
	\begin{align*}
		& \E \left[ g''(Z) \right] + b - \vert \kappa \vert \cdot \mathsf{Var} \left( g'(Z) \right) = \vert \kappa \vert \cdot \mathsf{Var} \left( b(Z -x) \right) + 2 \vert \kappa \vert  \cdot \mathsf{Cov} \left( g'(Z), b(Z - x) \right) \\
		= \, & \vert \kappa \vert \cdot \mathsf{Cov} \left( b(Z - x), b(Z - x) + 2 g'(Z) \right) \le \vert \kappa \vert \cdot \sqrt{\mathsf{Var} \left( b(Z - x) \right) \mathsf{Var} \left( b(Z - x) + 2 g'(Z) \right)} \\
		\le \, & \vert \kappa \vert \cdot \sqrt{\E \left[ \frac{b^2}{u''(Z)} \right] \E \left[ \frac{(b + 2 g''(Z))^2}{u''(Z)} \right]} = \sqrt{\E \left[ \frac{b^2}{g''(Z) + b} \right] \E \left[ \frac{(b + 2 g''(Z))^2}{g''(Z) + b} \right]},
	\end{align*}
	where the last inequality follows from the fact $u''(z) = \vert \kappa \vert (g''(z) + b)$.
	Denote $A = g''(Z)$, then $A$ is a bounded random variable such that $0 \le A \le d$. Our previous assumption $\gamma(t) - h_2 \le (1 + 1 / \sqrt{2}) (\gamma(1) - h_2)$ is equivalent to $b \ge \sqrt{2} d$. We claim that
	\begin{equation}\label{eq:expc_ineq}
		\sqrt{\E \left[ \frac{b^2}{A + b} \right] \E \left[ \frac{(b + 2 A)^2}{A + b} \right]} \le \frac{b(2 d + b)}{d + b},
	\end{equation}
	which further implies
	\begin{equation*}
		\partial_x^2 \Phi(t, x) \modif{=} \, \E \left[ g''(Z) \right] - \vert \kappa \vert \cdot \mathsf{Var} \left( g'(Z) \right) \le \frac{b(2 d + b)}{d + b} - b = \frac{b d}{d + b} = \frac{1}{\gamma(t) - h_2},
	\end{equation*}
	the desired curvature upper bound.
	
	Now it suffices to prove the claimed inequality~\eqref{eq:expc_ineq}. First, note that
	\begin{equation*}
		\E \left[ \frac{(b + 2 A)^2}{A + b} \right] = 4 \E \left[ A + b \right] + \E \left[ \frac{b^2}{A + b} \right] - 4 b. 
	\end{equation*}
	Hence, when $\E [b^2 / (A + b)]$ is fixed, $\E [A + b]$ is maximized if and only if $A$ only takes its extreme values, i.e., $A \in \{ 0, d \}$. In particular, there exists $p \in [0, 1]$ such that
	\begin{equation*}
		\P (A = d) = p, \quad \P (A = 0) = 1 - p.
	\end{equation*}
	With this choice of $A$, we define
	\begin{equation*}
		\psi (p) = \E \left[ \frac{b^2}{A + b} \right] \E \left[ \frac{(b + 2 A)^2}{A + b} \right] = \left( p \cdot \frac{b^2}{d + b} + (1 - p) \cdot b \right) \cdot \left( p \cdot \frac{(2d + b)^2}{d + b} + (1 - p) \cdot b \right).
	\end{equation*}
	Direct calculation reveals that
	\begin{equation*}
		\psi'(p) = \frac{2bd (2d + b)}{d + b} - \frac{2 b d^2 (4d + 3b)}{(b + d)^2} \cdot p,
	\end{equation*}
	which is decreasing in $p$, and (recall that $b \ge \sqrt{2} d$)
	\begin{equation*}
		\psi'(1) = \frac{2bd (b^2 - 2 d^2)}{(d + b)^2} \ge 0.
	\end{equation*}
	Therefore, $\psi'(p) \ge 0$ for all $p \in [0, 1]$, meaning that 
	\begin{equation*}
		\E \left[ \frac{b^2}{A + b} \right] \E \left[ \frac{(b + 2 A)^2}{A + b} \right] \le \psi(1) = \frac{b^2 (b + 2 d)^2}{(d + b)^2},
	\end{equation*}
	which finally leads to our claim~\eqref{eq:expc_ineq}.
	
	Now we have proved that for all $t \in [t_{m - 1}, 1]$, $\gamma(t) > h_2$ implies that $\sup_{x \in \R} \partial_x^2 \Phi(t, x) \le 1 / (\gamma(t) - h_2)$. Repeating this argument for smaller $t$ until $\gamma(t) \le h_2$ gives us the second part of Eq.~\eqref{eq:curv_bd_Phi_gen}. This completes the proof.
\end{proof}

\begin{prop}\label{prop:a_priori_est_gen}
	Let $h$ satisfy \cref{ass:h_regularity_gen}. Assume $\mu \in \mathsf{SF} [0, 1]$ is such that $(\mu, c) \in \FS$, i.e., $c + \int_{t}^{1} \mu(s) \d s > 0$ for all $t \in [0, 1]$, \modif{and $\sup_{z \in \R} h''(z) < 1/c$}. Then, $f_{\mu} (t, \cdot) \in C^{\infty} (\R)$ for any $t \in [0, 1]$, and $f_{\mu} \in C^{\infty, \infty}$ at all continuity points of $\mu$. Further, the following estimates hold:
	\begin{align}
		\label{eq:estimate_df}
		\norm{\partial_x f_{\mu} (t, \cdot)}_{L^{\infty} (\R)} \le \, & \norm{h'}_{L^{\infty} (\R)}, \\
		\label{eq:estimate_d2f}
		- \gamma(t) < \, \partial_x^2 f_{\mu} (t, x) \le \, & C(\mu, 2), \,\, \forall (t, x) \in [0, 1] \times \R, \\
		\label{eq:estimate_dkf}
		\norm{\partial_x^k f_{\mu} (t, \cdot)}_{L^{\infty} (\R)} \le \, & C(\mu, k), \ k = 3, 4.
	\end{align}
	Here, for all $2 \le k \le 4$, $C(\mu, k)$ has the following property: for any sequence $\{ \mu_n \}$ such that $(\mu_n, c) \stackrel{\FS}{\longrightarrow} (\mu, c)$, $C(\mu_n, k) \to C(\mu, k)$ as $n \to \infty$.
\end{prop}

\begin{proof}
	Throughout the proof we denote $h_2 = \sup_{z \in \R} h''(z)$. The estimate~\eqref{eq:estimate_df} follows directly from Cole-Hopf transform and Lemma~\ref{lem:terminal_regularity_gen}. As for \eqref{eq:estimate_d2f}, we already know $\partial_x^2 f_{\mu} (t, x) > - \gamma(t)$ from Proposition~\ref{prop:curv_bd_gen}, it suffices to prove the upper bound. Since $\gamma(1)= 1/c > h_2$ and $\mu = \gamma' / \gamma^2 \in L^1 [0, 1]$, there exists $\theta = \theta(\mu) \in [0, 1)$ such that
	\begin{equation*}
		\inf_{t \in [\theta, 1]} \gamma(t) - h_2 \ge \frac{1}{2} \left( \gamma(1) - h_2 \right).
	\end{equation*}
	According to Proposition~\ref{prop:curv_bd_gen}, on $[\theta, 1]$ we always have
	\begin{equation*}
		\partial_x^2 f_{\mu} (t, x) \le \, \frac{\gamma(t) h_2}{\gamma(t) - h_2} \le \frac{(\gamma(1) + h_2) h_2}{\gamma(1) - h_2}.
	\end{equation*}
	On $[0, \theta]$, we can apply Duhamel's principle to the Parisi PDE (defined on $[0, \theta]$, with $f_{\mu} (\theta, \cdot)$ being the new terminal condition) to obtain that
	\begin{equation}\label{eq:duhamel_gen}
		\begin{split}
			f_{\mu} (t, x) = \, & \int_{\R} K \left( \frac{\theta-t}{2}, x-y \right) f_{\mu} (\theta, y) \d y \\
			& + \frac{1}{2} \int_{t}^{\theta} \mu(s) \d s \int_{\R} K \left( \frac{s-t}{2}, x-y \right) (\partial_x f_{\mu} (s, y))^2 \d y,
		\end{split}
	\end{equation}
	where
	\begin{equation*}
		K (t, x) = \frac{1}{\sqrt{4 \pi t}} \exp \left( - \frac{x^2}{4 t} \right)
	\end{equation*}
	is the heat kernel. Note that Eq.~\eqref{eq:duhamel_gen} implies
	\begin{equation*}
		\begin{split}
			\partial_x^2 f_{\mu} (t, x) = \, & \int_{\R} K \left( \frac{\theta-t}{2}, x-y \right) \partial_x^2 f_{\mu} (\theta, y) \d y \\
			& + \int_{t}^{\theta} \mu(s) \d s \int_{\R} \partial_x K \left( \frac{s-t}{2}, x-y \right) \partial_x f_{\mu} (s, y) \partial_x^2 f_{\mu} (s, y) \d y,
		\end{split}
	\end{equation*}
	which leads to the estimate
	\begin{align*}
		\norm{\partial_x^2 f_{\mu} (t, \cdot)}_{L^{\infty} (\R)} \stackrel{(i)}{\le} \, & \norm{\partial_x^2 f_{\mu} (\theta, \cdot)}_{L^{\infty} (\R)} \\
		& + C \int_{t}^{\theta} \frac{\mu(s)}{\sqrt{s-t}} \norm{\partial_x f_{\mu} (s, \cdot)}_{L^{\infty} (\R)} \norm{\partial_x^2 f_{\mu} (s, \cdot)}_{L^{\infty} (\R)} \d s \\
		\stackrel{(ii)}{\le} \, & C + C \norm{\mu}_{L^{\infty} [0, \theta]} \int_{t}^{\theta} \frac{1}{\sqrt{s-t}} \norm{\partial_x^2 f_{\mu} (s, \cdot)}_{L^{\infty} (\R)} \d s,
	\end{align*}
	where $(i)$ follows from Young's inequality \smodif{and the fact that
		\begin{equation*}
			\left\| \partial_x K \left( \frac{s-t}{2}, \cdot \right) \right\|_{L^1 (\R)} \lesssim \, \frac{1}{(s-t)^{3/2}} \int_{\R} \vert x \vert \exp \left( - \frac{x^2}{2 (s-t)} \right) \d x \lesssim \frac{1}{\sqrt{s - t}},
		\end{equation*}
	}$(ii)$ follows from the estimate~\eqref{eq:estimate_df}. Note that the constant $C$ does not depend on $\mu$. Denote $g(t) = \norm{\partial_x^2 f_{\mu} (t, \cdot)}_{L^{\infty} (\R)}$, then we have
	\begin{align*}
		g(t) \le \, & C + C \norm{\mu}_{L^{\infty} [0, \theta]} \int_{t}^{\theta} \frac{1}{\sqrt{s-t}} g(s) \d s \\
		\stackrel{(i)}{\le} \, & C + C \norm{\mu}_{L^{\infty} [0, \theta]} \left( \int_{t}^{\theta} (s-t)^{-3/4} \d s \right)^{2/3} \left( \int_{t}^{\theta} g(s)^{3} \d s \right)^{1/3} \\
		\le \, & C + C \norm{\mu}_{L^{\infty} [0, \theta]} \left( \int_{t}^{\theta} g(s)^{3} \d s \right)^{1/3},
	\end{align*}
	where $(i)$ is due to H\"{o}lder's inequality. We thus obtain that
	\begin{align*}
		g(t)^3 \le \, C + C \norm{\mu}_{L^{\infty} [0, \theta]}^3 \int_{t}^{\theta} g(s)^{3} \d s.
	\end{align*}
	Using Gr\"{o}nwall's inequality, it follows that
	\begin{align*}
		g(t)^3 \le \, C \exp \left( C (\theta - t) \norm{\mu}_{L^{\infty} [0, \theta]}^3 \right) \le C \exp \left( C \norm{\mu}_{L^{\infty} [0, \theta]}^3 \right),
	\end{align*}
	which further implies $g(t) \le C \exp ( C \norm{\mu}_{L^{\infty} [0, \theta]}^3 )$.
	We can then define
	\begin{equation*}
		C(\mu, 2) = \max \left\{ C \exp \left( C \norm{\mu}_{L^{\infty} [0, \theta]}^3 \right), \ \frac{(\gamma(1) + h_2) h_2}{\gamma(1) - h_2} \right\}.
	\end{equation*}
	For any sequence $\{ \mu_n \}$ that converges to $\mu$ in the sense of Definition~\ref{def:conv_in_L_gen}, we know that the corresponding $\theta_n$ must converge to $\theta$. Therefore, $C(\mu_n, 2) \to C(\mu, 2)$ as $n \to \infty$. 
	
	In order to prove the estimate~\eqref{eq:estimate_dkf} on $\partial_x^k f_{\mu}$ for $k = 3, 4$, we can use a similar stochastic calculus argument as the proof of Proposition 6 in \cite{el2022algorithmic}. The resulting constants $C(\mu, k)$ depends continuously on $C(\mu, 2)$ and $\norm{\mu}_{L^1 [0, 1]}$, thus naturally satisfying the desired property. This completes the proof.
\end{proof}
\modif{\begin{rem}
		Note that the application of Duhamel's principle in the above proof critically relies on the assumption that $\norm{\mu}_{L^{\infty} [0, \theta]} < \infty$ for all $\theta \in [0, 1)$. Since we do not assume $\norm{\mu}_{L^{\infty} [0, 1]} < \infty$, we have to employ \cref{prop:curv_bd_gen} to derive an a priori upper bound for $\partial_x^2 f_{\mu} (t, \cdot)$ over $t \in [\theta, 1]$.
\end{rem}}
We are now in position to establish the following:
\begin{thm}[Solution to the Parisi PDE]\label{thm:existence_parisi_gtr}
	Let $h$ satisfy \cref{ass:h_regularity_gen}. For any $(\mu, c) \in \FS$ satisfying $\sup_{z \in \R} h''(z) < 1/c$, the Parisi PDE~\eqref{eq:parisi_mu_gen} admits a weak solution $f_{\mu}$ such that $f_{\mu} (t, \cdot) \in C^4 (\R)$ for all $t \in [0, 1]$, and
	\begin{align}
		\label{eq:estimate_df_gen}
		\norm{\partial_x f_{\mu} (t, \cdot)}_{L^{\infty} (\R)} \le \, & \norm{h'}_{L^{\infty} (\R)}, \\
		\label{eq:estimate_d2f_gen}
		- \gamma(t) < \, \partial_x^2 f_{\mu} (t, x) \le \, & C(\mu, 2), \, \forall (t, x) \in [0, 1] \times \R, \\
		\label{eq:estimate_dkf_gen}
		\norm{\partial_x^k f_{\mu} (t, \cdot)}_{L^{\infty} (\R)} \le \, & C(\mu, k), \ k = 3, 4.
	\end{align}
	Further, for any $\theta < 1$ and $0 \le k \le 2$, one has $\partial_t \partial_x^k f_{\mu} \in L^{\infty} ([0, \theta] \times \R)$.
\end{thm}

\begin{proof}
	We will establish the existence of a weak solution to the Parisi PDE~\eqref{eq:parisi_mu_gen} for general $\mu$ (not necessarily in $\mathsf{SF} [0, 1]$) via an approximation procedure. Let $(\mu, c) \in \FS$ be such that $\sup_{z \in \R} h''(z) < 1/c$. Then, there exists a sequence $\{ \mu_n \}_{n=1}^{\infty} \subset \mathsf{SF} [0, 1]$ such that $(\mu_n, c) \stackrel{\FS}{\longrightarrow} (\mu, c)$. Let $f_{\mu_n}$ be the solution to the Parisi PDE associated with $\mu_n$, we can follow the proof of Lemma 14 in \cite{jagannath2016dynamic} to show that
	\begin{equation*}
		\norm{f_{\mu_n} - f_{\mu_m}}_{L^{\infty} ([0, 1] \times \R)} \le \frac{\norm{h'}_{L^{\infty} (\R)}^2}{2} \norm{\mu_n - \mu_m}_{L^1 [0, 1]} \to 0 \ \text{as} \ n, m \to \infty.
	\end{equation*}
	Therefore, we know that as $n \to \infty$, $f_{\mu_n}$ converges pointwise to some function $f_{\mu}: [0, 1] \times \R \to \R$, and this convergence is uniform on $[0, 1] \times K$ for any compact set $K \subset \R$. Now since $\{ \partial_x^2 f_{\mu_n} \}_{n \ge 1}$ is uniformly bounded on compact sets (Proposition~\ref{prop:a_priori_est_gen}), we deduce from Lemma~\ref{lem:diff_convergence_gen} that $\partial_x f_{\mu}$ exists and $\partial_x f_{\mu_n} \to \partial_x f_{\mu}$ uniformly on any compact set. Exploiting the bounds~\eqref{eq:estimate_dkf} and repeating the same argument, we know that $\partial_x^k f_{\mu}$ exists and $\partial_x^k f_{\mu_n} \to \partial_x^k f_{\mu}$ uniformly on compact sets for $k = 2, 3$. Further, since $\partial_x^3 f_{\mu_n} (t, \cdot)$ is $C(\mu_n, 4)$-Lipschitz and $C(\mu_n, 4) \to C(\mu, 4)$, it follows that $\partial_x^3 f_{\mu} (t, \cdot)$ is $C(\mu, 4)$-Lipschitz. Therefore, $\partial_x^4 f_{\mu}$ exists and is upper bounded by $C(\mu, 4)$ almost everywhere. As a consequence, the estimates~\eqref{eq:estimate_df} to \eqref{eq:estimate_dkf} hold for $f_{\mu}$ up to $k=4$ as well. This proves Eqs.~\eqref{eq:estimate_df_gen}, \eqref{eq:estimate_d2f_gen}, and \eqref{eq:estimate_dkf_gen}.
	
	Finally, similar to the proof of Lemma 6.2 in \cite{el2021optimization}, we can show that $f_{\mu}$ is a weak solution to the Parisi PDE~\eqref{eq:parisi_mu_gen}. Further, as $\mu \in L^{\infty} [0, \theta]$ for any $\theta \in [0, 1)$, we know that $\partial_t \partial_x^k f_{\mu} \in L^{\infty} ([0, \theta] \times \R)$ for $k = 0, 1, 2$. This establishes the desired regularity of $f_{\mu}$.
\end{proof}
In the proof of \cref{thm:existence_parisi_gtr}, we make use of the following auxiliary lemma, which we state and prove below for the reader's convenience.
\begin{lem}\label{lem:diff_convergence_gen}
	Let $\{ \varphi_n \}_{n \ge 1}$ be a sequence of twice-differentiable real-valued functions satisfying:
	\begin{itemize}
		\item [$(a)$] For any compact set $K \subset \R$, $\sup_{x \in K} \vert \varphi_n (x) - \varphi (x) \vert \to 0$ as $n \to \infty$.
		\item [$(b)$] $\sup_{n \in \mathbb{N}} \sup_{x \in K} \vert \varphi_n''(x) \vert < +\infty$ for any compact $K$.
	\end{itemize}
	Then, $\varphi$ is differentiable and $\varphi_n' \to \varphi'$ uniformly on any compact set as $n \to \infty$.
\end{lem}
\begin{proof}
	Fix a compact set $K$ and denote $C_{2, K} = \sup_{n \in \mathbb{N}} \sup_{x \in K} \vert \varphi_n''(x) \vert$. We first show that $\{ \varphi_n' \}_{n \ge 1}$ is a Cauchy sequence in $L^{\infty} (K)$. To this end, note that for any $x, y \in K$ and $m, n \in \mathbb{N}$,
	\begin{align*}
		\left\vert \varphi_n'(x) - \varphi_m' (x) \right\vert \le \, & \left\vert \varphi_n'(x) - \frac{\varphi_n (y) - \varphi_n(x)}{y - x} \right\vert  + \left\vert \varphi_m'(x) - \frac{\varphi_m (y) - \varphi_m(x)}{y - x} \right\vert \\
		& + \left\vert \frac{\varphi_m (y) - \varphi_m(x)}{y - x} - \frac{\varphi_n (y) - \varphi_n(x)}{y - x} \right\vert \\
		\le \, & 2C_{2, K} \vert x - y \vert + \frac{2}{\vert x - y \vert} \sup_{x \in K} \vert \varphi_n (x) - \varphi_m (x) \vert,
	\end{align*}
	which implies that for any $\veps > 0$,
	\begin{align*}
		& \sup_{x \in K} \left\vert \varphi_n'(x) - \varphi_m' (x) \right\vert \le 2C_{2, K} \veps + \frac{2}{\veps} \sup_{x \in K} \vert \varphi_n (x) - \varphi_m (x) \vert \\
		\implies \, & \limsup_{m, n \to \infty} \sup_{x \in K} \left\vert \varphi_n'(x) - \varphi_m' (x) \right\vert \le 2 C_{2, K} \veps.
	\end{align*}
	Since $\veps$ is arbitrary, this proves $\norm{\varphi_n' - \varphi_m'}_{L^{\infty} (K)} \to 0$ as $m, n \to \infty$. As a consequence, $\varphi_n'$ uniformly converges to some $f$ in $C(K)$. It remains to show that $f = \varphi'$. For any $x, y \in K$, we have
	\begin{equation*}
		\varphi (x) - \varphi(y) = \lim_{n \to \infty} \left\{ \varphi_n (x) - \varphi_n (y) \right\} = \lim_{n \to \infty} \int_{x}^{y} \varphi_n'(z) \d z = \int_{x}^{y} f (z) \d z,
	\end{equation*}
	where the last inequality follows by dominated convergence. Since $f$ is continuous, we know that $\varphi' = f$. This completes the proof.
\end{proof}

We now prove \cref{thm:solve_Parisi_PDE} without imposing \cref{ass:h_regularity_gen}, via an approximation argument. Note that under the conditions of \cref{thm:solve_Parisi_PDE}, we still have $\norm{\partial_x f_{\mu} (1, \cdot)}_{L^{\infty} (\R)} \le \norm{h'}_{L^{\infty} (\R)}$ and $\partial_x^2 f_{\mu} (1, \cdot) \in C_b (\R)$, as established in the proof of \cref{lem:terminal_regularity_gen}.

For any $\veps > 0$, define $f_{\mu}^{\veps} (1, \cdot)$ to be the $\veps$-mollifier of $f_{\mu} (1, \cdot)$ via the heat kernel, namely
\begin{equation}\label{eq:eps_mollifier}
	f_{\mu}^{\veps} (1, x) = \, \E_{G \sim \normal(0, 1)} \left[ f_{\mu} \left( 1, x + \veps G \right) \right].
\end{equation}
Of course, $f_{\mu}^{\veps} (1, \cdot) \in C^{\infty} (\R)$, and we have the following $L^{\infty}$-norm bounds regarding its partial derivatives:
\begin{prop}\label{prop:terminal_approximation}
	For any $\veps > 0$, we have
	\begin{itemize}
		\item [$(a)$] $\norm{f_{\mu}^{\veps} (1, \cdot) - f_{\mu} (1, \cdot)}_{L^{\infty} (\R)} \le \veps^2 \norm{\partial_x^2 f_{\mu} (1, \cdot)}_{L^{\infty} (\R)}$.
		\item [$(b)$] For any compact set $K \subset \R$, we have
		\begin{align*}
			\norm{\partial_x f_{\mu}^{\veps} (1, \cdot) - \partial_x f_{\mu} (1, \cdot)}_{L^{\infty} (K)} \le \, & \veps \sup_{x \in K} \sup_{t \in [0, \veps]} \left\vert \E_{G} \left[ G \cdot \partial_x^2 f_{\mu} (1, x + t G) \right] \right\vert, \\
			\norm{\partial_x^2 f_{\mu}^{\veps} (1, \cdot) - \partial_x^2 f_{\mu} (1, \cdot)}_{L^{\infty} (K)} \to \, & 0 \ \text{as} \ \veps \to 0, \\
			\norm{\partial_x^3 f_{\mu}^{\veps} (1, \cdot)}_{L^{\infty} (K)} \le \, & \frac{1}{\veps} \sup_{x \in K} \sup_{t \in [0, \veps]} \left\vert \E_{G} \left[ G \cdot \partial_x^2 f_{\mu} (1, x + t G) \right] \right\vert.
		\end{align*}
		Further, we have
		\begin{equation*}
			\sup_{x \in K} \sup_{t \in [0, \veps]} \left\vert \E_{G} \left[ G \cdot \partial_x^2 f_{\mu} (1, x + t G) \right] \right\vert \to 0 \ \text{as} \ \veps \to 0.
		\end{equation*}
	\end{itemize}
\end{prop}

\begin{proof}
	We first prove $(a)$. By definition, we have
	\begin{align*}
		& \norm{f_{\mu}^{\veps} (1, \cdot) - f_{\mu} (1, \cdot)}_{L^{\infty} (\R)} = \, \sup_{x \in \R} \left\vert f_{\mu}^{\veps} (1, x) - f_{\mu} (1, x) \right\vert \\
		= \, & \sup_{x \in \R} \left\vert \E_{G} \left[ f_{\mu} \left( 1, x + \veps G \right) - f_{\mu} (1, x) \right] \right\vert \stackrel{(i)}{=} \sup_{x \in \R} \left\vert \E_{G} \left[ \partial_x f_{\mu} ( 1, x) \cdot \veps G + \frac{1}{2} \partial_x^2 f_{\mu} (1, x^*) \cdot \veps^2 G^2 \right] \right\vert \\
		\le \, & \frac{\veps^2}{2} \sup_{x \in \R} \left\vert \partial_x^2 f_{\mu} (1, x) \right\vert \cdot \E_{G} [G^2] \le \veps^2 \norm{\partial_x^2 f_{\mu} (1, \cdot)}_{L^{\infty} (\R)},
	\end{align*}
	\smodif{where in $(i)$ we use Taylor expansion and $x^*$ is between $x$ and $x + \veps G$.}
	Similarly, we can show that
	\begin{align*}
		& \norm{\partial_x f_{\mu}^{\veps} (1, \cdot) - \partial_x f_{\mu} (1, \cdot)}_{L^{\infty} (K)} = \, \sup_{x \in K} \left\vert \partial_x f_{\mu}^{\veps} (1, x) - \partial_x f_{\mu} (1, x) \right\vert \\
		= \, & \sup_{x \in K} \left\vert \E_{G} \left[ \partial_x f_{\mu} \left( 1, x + \veps G \right) - \partial_x f_{\mu} (1, x) \right] \right\vert \le \sup_{x \in K} \left\{ \veps \cdot \sup_{t \in [0, \veps]} \left\vert \E_{G} \left[ G \cdot \partial_x^2 f_{\mu} (1, x + t G) \right] \right\vert \right\} \\
		= \, & \veps \sup_{x \in K} \sup_{t \in [0, \veps]} \left\vert \E_{G} \left[ G \cdot \partial_x^2 f_{\mu} (1, x + t G) \right] \right\vert,
	\end{align*}
	where we have
	\begin{equation*}
		\sup_{x \in K} \sup_{t \in [0, \veps]} \left\vert \E_{G} \left[ G \cdot \partial_x^2 f_{\mu} (1, x + t G) \right] \right\vert \to 0
	\end{equation*}
	since $\partial_x^2 f_{\mu} (1, \cdot)$ is uniformly continuous on $K$. \smodif{To establish the other estimates in $(b)$, we note that
		\begin{align*}
			& \norm{\partial_x^2 f_{\mu}^{\veps} (1, \cdot) - \partial_x^2 f_{\mu} (1, \cdot)}_{L^{\infty} (K)} = \, \sup_{x \in K} \left\vert \E_{G} \left[ \partial_x^2 f_{\mu} \left( 1, x + \veps G \right) - \partial_x^2 f_{\mu} (1, x) \right] \right\vert \\
			\le \, & \E_{G} \left[ \sup_{x \in K} \left\vert \partial_x^2 f_{\mu} \left( 1, x + \veps G \right) - \partial_x^2 f_{\mu} (1, x) \right\vert \right] \to 0
		\end{align*}
		almost surely as $\veps \to 0$ by bounded convergence theorem, since $\partial_x^2 f_{\mu} (1, \cdot)$ is uniformly continuous and bounded on $K$. Further, denoting by $\phi$ the Gaussian PDF, we have
		\begin{align*}
			\partial_x^3 f_{\mu}^{\veps} (1, x) = \, & - \int_{\R} \partial_y^2 f_{\mu} (1, y) \frac{1}{\veps^2} \phi' \left( \frac{y - x}{\veps} \right) \d y = - \int_{\R} \partial_x^2 f_{\mu} (1, x + \veps z ) \frac{1}{\veps} \phi' ( z ) \d z \\
			= \, & \frac{1}{\veps} \int_{\R} \partial_x^2 f_{\mu} (1, x + \veps z ) z \phi ( z ) \d z = \frac{1}{\veps} \E_G \left[ G \cdot \partial_x^2 f_{\mu} (1, x + \veps G ) \right],
		\end{align*}
		which immediately leads to
		\begin{equation*}
			\norm{\partial_x^3 f_{\mu}^{\veps} (1, \cdot)}_{L^{\infty} (K)} \le \, \frac{1}{\veps} \sup_{x \in K} \sup_{t \in [0, \veps]} \left\vert \E_{G} \left[ G \cdot \partial_x^2 f_{\mu} (1, x + t G) \right] \right\vert.
		\end{equation*}
		This completes the proof.}
\end{proof}
\begin{rem}
	These estimates are still valid if we replace $f_{\mu}$ by $f_{\mu}^{\eta}$ for some $\eta \le \veps$, since $f_{\mu}^{\veps}$ can be viewed as a mollifier of $f_{\mu}^{\eta}$ as well.
\end{rem}
Now, let $f_{\mu}^{\veps} : [0, 1] \times \R \to \R$ be the solution to the Parisi PDE~\eqref{eq:parisi_mu_gen} with terminal value $f_{\mu}^{\veps} (1, \cdot)$. We will show that for any sequence $\veps_n \to 0$, $\{ f_{\mu}^{\veps_n} \}_{n \ge 1}$ is a Cauchy sequence in the following sense:
\begin{thm}\label{thm:C2_approx_gen}
	The following holds for the sequence $\{ f_{\mu}^{\veps_n} \}_{n \ge 1}$ as $\veps_n \to 0$:
	\begin{itemize}
		\item [$(a)$] $\lim_{m, n \to \infty} \norm{f_{\mu}^{\veps_m} - f_{\mu}^{\veps_n}}_{L^{\infty} ([0, 1] \times \R)} = 0$.
		\item [$(b)$] $\lim_{m, n \to \infty} \norm{ \partial_x f_{\mu}^{\veps_m} - \partial_x f_{\mu}^{\veps_n}}_{L^{\infty} ([0, 1] \times \R)} = 0$.
		\item [$(c)$] For any compact set $K \subset \R$, we have
		\begin{equation*}
			\lim_{m, n \to \infty} \norm{ \partial_x^2 f_{\mu}^{\veps_m} - \partial_x^2 f_{\mu}^{\veps_n}}_{L^{\infty} ([0, 1] \times K)} = 0.
		\end{equation*}
	\end{itemize}
\end{thm}
\begin{proof}
	Throughout the proof we denote $\veps = \veps_n$, $\eta = \veps_m$, and $w = f_{\mu}^{\veps_m} - f_{\mu}^{\veps_n}$.
	
	\vspace{0.5em} 
	\noindent \textbf{Proof of $(a)$.} Note that $w$ satisfies the following PDE:
	\begin{equation*}
		\partial_t w + \frac{1}{2} \mu(t) \left( \partial_x f_{\mu}^{\veps} + \partial_x f_{\mu}^{\eta} \right) \partial_x w + \frac{1}{2} \partial_x^2 w = 0.
	\end{equation*}
	Fix $(t, x) \in [0, 1] \times \R$, and let $\{ X_s \}_{s \in [t, 1]}$ be the solution to the SDE:
	\begin{equation*}
		X_t = x, \ \d X_s = \frac{1}{2} \mu(s) \left( \partial_x f_{\mu}^{\veps} (s, X_s) + \partial_x f_{\mu}^{\eta} (s, X_s) \right) \d s + \d B_s.
	\end{equation*}
	Note that the solution uniquely exists since $\partial_x f_{\mu}^{\veps}$ and $\partial_x f_{\mu}^{\eta}$ are Lipschitz. Using It\^{o}'s formula, we obtain that
	\begin{align*}
		\d w(s, X_s) = \partial_xw(s, X_s) \d B_s \implies w(t, x) = \E_{X_t = x} \left[ w(1, X_1) \right].
	\end{align*}
	The conclusion then follows from Proposition~\ref{prop:terminal_approximation} $(a)$.
	
	\vspace{0.5em} 
	\noindent \textbf{Proof of $(b)$.} Define $w_1 = \partial_x w$, then we know that $w_1$ satisfies
	\begin{equation*}
		\partial_t w_1 + \mu(t) \partial_x f_{\mu}^{\veps} \cdot \partial_x w_1 + \frac{1}{2} \partial_x^2 w_1 + \mu(t) \partial_x^2 f_{\mu}^{\eta} \cdot w_1 = 0.
	\end{equation*}
	Fix $(t, x) \in [0, 1] \times \R$, and let $\{ Y_s \}_{s \in [t, 1]}$ solve the SDE:
	\begin{equation*}
		Y_t = x, \ \d Y_s = \mu(s) \partial_x f_{\mu}^{\veps} (s, Y_s) \d s + \d B_s.
	\end{equation*}
	Similarly, we know that the solution exists uniquely. Using Feynman-Kac formula, it follows that
	\begin{equation*}
		w_1 (t, x) = \E_{Y_t = x} \left[ \exp \left( \int_{t}^{1} \mu(s) \partial_x^2 f_{\mu}^{\eta} (s, Y_s) \d s \right) w_1 (1, Y_1) \right].
	\end{equation*}
	Since $\partial_x^2 f_{\mu}^{\eta}$ is uniformly bounded, and from the proof of Proposition~\ref{prop:terminal_approximation} we know that
	\begin{equation*}
		\norm{w_1 (1, \cdot)}_{L^{\infty} (\R)} = \sup_{x \in \R} \left\vert \partial_x f_{\mu}^{\veps} (1, x) - \partial_x f_{\mu}^{\eta} (1, x) \right\vert \to 0 \ \text{as} \ \veps, \eta \to 0,
	\end{equation*}
	we deduce that $\norm{w_1}_{L^{\infty} ([0, 1] \times \R)} \to 0$ as $\veps, \eta \to 0$. This proves part $(b)$.
	
	\vspace{0.5em} 
	\noindent \textbf{Proof of $(c)$.} Define $w_2 = \partial_x^2 w$. Then, we know that $w_2$ satisfies the follwing PDE:
	\begin{equation*}
		\partial_t w_2 + \mu(t) \partial_x f_{\mu}^{\eta} \cdot \partial_x w_2 + \frac{1}{2} \partial_x^2 w_2 + \mu(t) \left( \partial_x^2 f_{\mu}^{\veps} + \partial_x^2 f_{\mu}^{\eta} \right) \cdot w_2 + \mu(t) w_1 \cdot \partial_x^3 f_{\mu}^{\veps} = 0.
	\end{equation*}
	Fix $(t, x) \in [0, 1] \times \R$, and let $\{Z_s \}_{s \in [t, 1]}$ be the unique solution to the SDE:
	\begin{equation*}
		Z_t = x, \ \d Z_s = \mu(s) \partial_x f_{\mu}^{\eta} (s, Z_s) \d s + \d B_s.
	\end{equation*}
	Using again Feynman-Kac formula, we obtain that
	\begin{align*}
		& w_2 (t, x) \\
		= \, & \E_{Z_t = x} \left[ \int_{t}^{1} \exp \left( \int_{t}^{\tau} \mu(s) \left( \partial_x^2 f_{\mu}^{\veps} (s, Z_s) + \partial_x^2 f_{\mu}^{\eta} (s, Z_s) \right) \d s \right) \mu(\tau) w_1 (\tau, Z_{\tau}) \partial_x^3 f_{\mu}^{\veps} (\tau, Z_\tau) \d \tau \right] \\
		& + \E_{Z_t = x} \left[ \exp \left( \int_{t}^{1} \mu(s) \left( \partial_x^2 f_{\mu}^{\veps} (s, Z_s) + \partial_x^2 f_{\mu}^{\eta} (s, Z_s) \right) \d s \right) w_2 (1, Z_1) \right].
	\end{align*}
	The second term converges to $0$ uniformly on $[0, 1] \times K$, since $\{ \operatorname{Law} (Z_1 \vert Z_t = x): (t, x) \in [0, 1] \times K \}$ is a tight family of probability distributions (see, e.g., the proof of Proposition~\ref{prop:approximation_process_gen}), and we recall from Proposition~\ref{prop:terminal_approximation}~$(b)$ that $w_2 (1, x) \to 0$ uniformly on $K$. To prove that the first term converges to $0$ uniformly on $[0, 1] \times K$, we can use Feynman-Kac formula to estimate $\partial_x^3 f_{\mu}^{\veps}$ and combine this with the estimate of $w_1$ in part~$(b)$. Finally, one can show that their product uniformly converges to $0$ on $[0, 1] \times K$ using the estimates in Proposition~\ref{prop:terminal_approximation}~$(b)$.
\end{proof}

\cref{thm:C2_approx_gen} $(a)$ immediately implies that as $\veps \to 0$, $f_{\mu}^{\veps}$ converges to some $f_{\mu}$ uniformly on $[0, 1] \times \R$ (hence of course on any compact set). Further, \cref{thm:C2_approx_gen} $(b)$ tells us that $\partial_x f_{\mu}^{\veps}$ uniformly converges to some $g_{\mu}$. Applying dominated convergence theorem, we know that $\partial_x f_{\mu}$ exists and equals $g_{\mu}$, namely $\partial_x f_{\mu}^{\veps}$ uniformly converges to $\partial_x f_{\mu}$. Repeating the same argument and using \cref{thm:C2_approx_gen} 
$(c)$, we know that $\partial_x^2 f_{\mu}$ exists and $\partial_x^2 f_{\mu}^{\veps}$ converges to $\partial_x^2 f_{\mu}$ uniformly on any compact set as $\veps \to 0$. \modif{Therefore, $\partial_x^2 f_{\mu}$ is continuous in $x$ for any $t \in [0, 1]$, and \cref{eq:estimate_df_gen,eq:estimate_d2f_gen} hold for $f_{\mu}$.}

According to the Parisi PDEs observed by $f_{\mu}^{\veps}$, we know that $\partial_t f_{\mu}^{\veps}$ converges to some $h_{\mu}$ uniformly on compact sets as $\veps \to 0$. Using again the dominated convergence theorem, we know that $h_{\mu} = \partial_t f_{\mu}$, i.e., $\partial_t f_{\mu}^{\veps}$ converges to $\partial_t f_{\mu}$ uniformly on any compact set as $\veps \to 0$. Similar to the proof of \cite[Lemma 6.2]{el2021optimization}, we obtain that $f_{\mu}$ is a weak solution of Eq.~\eqref{eq:parisi_mu_gen}. \modif{Since we already obtain the regularity estimates for $f_{\mu}$, this completes the proof of \cref{thm:solve_Parisi_PDE} for general $h \in C^2 (\R)$ and $(\mu, c) \in \FS$ satisfying $\sup_{z \in \R} h''(z) < 1/c$.}

\subsection{The case $\sup_{z \in \R} h'' (z) \ge 1/c$}    

To construct solutions to the Parisi PDE~\eqref{eq:parisi_mu_gen} under this setting, we first reduce the problem to the case $\sup_{z \in \R} h'' (z) \le 1/c$. Let $\operatorname{conc} (g(x))$ denote the upper concave envelope of a function $g$.
For $c = 1 / \gamma(1)$, we define 
\begin{equation}\label{eq:defn_h_c}
	h_c (x) = \operatorname{conc} \left( h(x) - \frac{x^2}{2 c} \right) + \frac{x^2}{2 c} = \operatorname{conc} \left( h(x) - \frac{\gamma(1) x^2}{2} \right) + \frac{\gamma(1) x^2}{2}.
\end{equation}
Then, we know that $h_c \in C^2 (\R)$ is also Lipschitz continuous, and bounded from above. Further, we have $\gamma(1) \ge \sup_{z \in \R} h_c''(z)$. By direct calculation, we get
\begin{equation}\label{eq:equiv_f_mu_h_c}
	f_\mu (1, x) = \, \sup_{u \in \R} \left\{ h_c (x+u) - \frac{u^2}{2 c} \right\},
\end{equation}
which means that one can use $h_c$ instead of $h$ when defining the Parisi PDE and the Parisi functional.
Under this simplification, we proceed to constructing solutions to the Parisi PDE~\eqref{eq:parisi_mu_gen} for general $h$.
\begin{thm}\label{prop:GeneralPDE_Constr}
	Assume $h\in C^2(\R)$ is Lipschitz continuous and bounded from above,
	and $(\mu,c)\in\FS$. Then, a weak solution to the Parisi PDE below (equivalent to \cref{eq:parisi_mu_gen}) exists and is unique:
	\begin{equation}\label{eq:parisi_mu_gen_c}
		\begin{split}
			\partial_t f_\mu (t, x) +& \frac{1}{2} \mu(t) \left( \partial_x f_\mu (t, x) \right)^2 + \frac{1}{2} \partial_x^2 f_\mu (t, x) = \, 0, \\
			f_\mu (1, x) = \, & \sup_{u \in \R} \left\{ h_c(x+u) - \frac{u^2}{2 c} \right\}.
		\end{split}
	\end{equation}
	Further, for a sequence $c_n < c$, let $f_{\mu}^n(t,x)$ denote the solution to the same PDE with terminal condition 
	$$
	f_\mu^n (1, x) = \, \sup_{u \in \R} \left\{ h_c(x+u) - \frac{u^2}{2 c_n} \right\}.
	$$
	Then, we have 
	$f_{\mu}^n \to f_{\mu}$ uniformly as $c_n \uparrow c$.
	Finally,
	\begin{align}
		\norm{\partial_x f_{\mu} (t, \cdot)}_{L^{\infty} (\R)} \le \, & \norm{h'}_{L^{\infty} (\R)}, \ \forall t \in [0, 1], \label{eq:GeneralPDE_Est1}\\
		\partial_x^2 f_{\mu} (t, x) > \, & - \gamma(t), \ \forall (t, x) \in [0, 1] \times \R.\label{eq:GeneralPDE_Est2}
	\end{align}
\end{thm}
\begin{proof}
	Let  $c_n \in \R_+$ be a sequence such that $c_n < c$ for each $n$ and $c_n \to c$. As in the theorem statement, denote the (weak) solution to Eq.~\eqref{eq:parisi_mu_gen_c} with $c$ replaced by $c_n$ as $f_{\mu}^n$. Then, from the conclusions of \cref{thm:solve_Parisi_PDE} (applied to $(\mu, c_n)$ and $h_c$) we know that
	\begin{align*}
		\norm{\partial_x f_{\mu}^n (t, \cdot)}_{L^{\infty} (\R)} \le \, & \norm{h_c'}_{L^{\infty} (\R)} \le \norm{h'}_{L^{\infty} (\R)}, \ \forall t \in [0, 1], \\
		- \gamma_n (t) < \, \partial_x^2 f_{\mu}^n (t, x) \le \, & C_n (\mu, 2), \ \forall (t, x) \in [0, 1] \times \R.
	\end{align*}
	By direct calculation, we know that $f_{\mu}^n (1, \cdot)$ converges uniformly to $f_{\mu} (1, \cdot)$. According to the maximum principle (or similar to the proof of \cref{thm:C2_approx_gen} $(a)$), it follows that $f_{\mu}^n (t, \cdot)$ converges uniformly to some $f_{\mu} (t, \cdot)$ for all $t \in [0, 1]$. Further, since $\partial_x^2 f_{\mu}^n (t, x) > -\gamma_n (t)$ and the sequence $\{ \partial_x f_{\mu}^n (t, \cdot) \}_{n \ge 1}$ is uniformly bounded in $L^{\infty}$, we know that $\partial_x f_{\mu}^n (t, \cdot)$ is of bounded variation on any finite interval.
	Using an argument similar to that in the proof of \cite[Lemma 6.2]{el2021optimization}, we deduce that $\partial_x f_{\mu}$ exists and $\partial_x f_{\mu}^n \to \partial_x f_{\mu}$ almost everywhere. As a consequence, $\partial_x f_{\mu} (t, \cdot)$ is of bounded variation on any finite interval as well, which implies that $\partial_x^2 f_{\mu}$ exists almost everywhere. Via a similar argument to the proof of \cite[Lemma 6.2]{el2021optimization}, we know that $f_{\mu}$ weakly solves the Parisi PDE~\eqref{eq:parisi_mu_gen_c}, thus establishing existence. Uniqueness follows from the uniqueness theorem for weak solutions of the heat equation with at most linear growth (since $\partial_x f_{\mu}$ is bounded we know that $f_{\mu}$ has at most linear growth).
	
	Further, applying Duhamel's principle implies that $\partial_x^2 f_{\mu}$ is continuous in $x$. Hence, it follows that $\forall t \in [0, 1]$, $\partial_x f_{\mu} (t, \cdot) \in C^1 (\R)$, and consequently $\partial_x f_{\mu}^n (t, \cdot)$ uniformly converges to $\partial_x f_{\mu} (t, \cdot)$ on any compact set (Dini's theorem). This in turns implies the estimates
	\eqref{eq:GeneralPDE_Est1}, \eqref{eq:GeneralPDE_Est2}, completing the proof of \cref{prop:GeneralPDE_Constr}.
\end{proof}
\begin{rem}
	Note that  in the case of general $h$, we do not have an upper bound on $\partial_x^2 f_{\mu} (t, x)$ because $C_n (\mu, 2) \to \infty$ as $n \to \infty$. 
\end{rem}

\modif{Combining the two settings discussed in this section concludes the proof of \cref{thm:solve_Parisi_PDE}.}

\section{Verification argument}\label{sec:verification_argument}

This section is devoted to the proof of \cref{prop:two_stage_veri_arg}, i.e., the duality between $V_{\gamma}$ and $f_{\mu}$. This is achieved by first establishing a connection between the Hamilton-Jacobi-Bellman (HJB) equation and the Parisi PDE, then constructing a control process and proving its optimality via the so-called ``verification argument''. To begin with, we recall the definition of $V_{\gamma}$ from \cref{eq:redefine_value_func}:
\begin{equation}\label{val_fct_gen}
	V_{\gamma} (t, z) = \sup_{\phi \in D[t, 1]} \E \left[ h \left( z + \int_{t}^{1} (1 + \phi_s) \d B_s \right) - \frac{1}{2} \int_{t}^{1} \gamma(s) \left( \phi_s^2 - \frac{1}{\alpha} \right) \d s \right],
\end{equation}
and define the HJB equation:
\begin{equation}\label{eq:HJB_1d_gen}
	\begin{split}
		\partial_t V_{\gamma} (t, z) + \frac{1}{2} \frac{\gamma(t) \partial_z^2 V_{\gamma} (t, z)}{\gamma(t) - \partial_z^2 V_{\gamma} (t, z)} + \frac{\gamma(t)}{2 \alpha} & = 0, \\
		V_{\gamma} (1, z) & = h(z).
	\end{split}
\end{equation}
\modif{Our proof follows a similar strategy as that used in the proof of \cref{thm:solve_Parisi_PDE}. Namely, we first establish the verification argument under $\sup_{z \in \R} h''(z) < 1/c$ and \cref{ass:h_regularity_gen}, and then extend our results to general $h \in C^2 (\R)$ and $(\mu, c) \in \FS$ via an approximation argument.}

\subsection{The case $\sup_{z \in \R} h''(z) < 1/c$} \modif{We first proceed with the verification argument for $h$ satisfying \cref{ass:h_regularity_gen}.}

\begin{prop}\label{prop:veri_arg_SF_gen}
	\modif{Let $h$ satisfy \cref{ass:h_regularity_gen} and $\sup_{z \in \R} h''(z) < 1/c$}. \modif{For $(\mu,c) \in \FS$ and the associated Lagrange multiplier $\gamma \in \FSG$,} denote $f_{\mu}$ as the solution to the Parisi PDE~\eqref{eq:parisi_mu_gen}. Then, we have
	\begin{equation}\label{eq:V_gamma_and_f_mu_gen}
		\begin{split}
			V_{\gamma} (t, z) =\, & \inf_{x \in \R} \left\{ f_{\mu} (t, x) + \frac{\gamma(t)}{2} (x - z)^2 \right\} + \frac{1}{2 \alpha} \int_{t}^{1} \gamma(s) \d s, \\
			f_{\mu} (t, x) =\, & \sup_{z \in \R} \left\{ V_{\gamma} (t, z) - \frac{\gamma(t)}{2} (z - x)^2 \right\} - \frac{1}{2 \alpha} \int_{t}^{1} \gamma(s) \d s
		\end{split}
	\end{equation}
	for all $t \in [0, 1]$ and $x, z \in \R$.
	Further, $V_{\gamma}$ solves Eq.~\eqref{eq:HJB_1d_gen}, and the supremum in the definition of $V_{\gamma} (t, z)$ is achieved at $(\phi_s^z)_{s \in [t, 1]}$ satisfying
	\begin{equation}\label{eq:def_ctrl_proc_gen}
		\phi_s^z = \frac{1}{\gamma(s)} \partial_x^2 f_{\mu} (s, X_s^z),
	\end{equation}
	where $\{ X_s^z \}_{s \in [t, 1]}$ \modif{is the unique solution to the SDE}
	\begin{equation}\label{eq:def_par_sde_gen}
		\frac{1}{\gamma(t)} \partial_x f_{\mu} (t, X_t^z) + X_t^z = z, \ \d X_s^z = \mu(s) \partial_x f_{\mu} (s, X_s^z) \d s + \d B_s, \ s \in [t, 1].
	\end{equation}
	\modif{(Existence and uniqueness of the solution will be established in the proof.)}
\end{prop}

\begin{proof}
	In virtue of Eqs.~\eqref{eq:Parisi_Phi_gen} and \eqref{eq:f_Phi_gen} (which can be defined for general $(\mu, c) \in \FS$ and $\gamma \in \FSG$), we will prove the following statements instead:
	\begin{itemize}
		\item [$(a)$] Define for $(t, z) \in [0, 1] \times \R$:
		\begin{equation*}
			V(t, z) = \frac{\gamma(t)}{2} z^2 - \Phi^{*} (t, z) - \frac{1}{2} \left( 1 - \frac{1}{\alpha} \right) \int_{t}^{1} \gamma(s) \d s,
		\end{equation*}
		\modif{where $\Phi^{*} (t, z)$ means the convex conjugate of $\Phi$ with respect to the second variable.}
		Then, $V$ solves the HJB equation~\eqref{eq:HJB_1d_gen}.
		\item [$(b)$] The verification argument implies $V= V_{\gamma}$
		with $V_{\gamma}$ given by Eq.~\eqref{val_fct_gen} (which in particular implies uniqueness of solution to Eq.~\eqref{eq:HJB_1d_gen}), 
		and characterizes the optimal control process. 
	\end{itemize}
	\vspace{0.5em} \noindent \textbf{Proof of $(a)$.}
	Since $\Phi(t, \cdot)$ is strictly convex, we have $\partial_z^2 V(t, z) < \gamma(t)$. By direct calculation,
	\begin{align*}
		\partial_t V (t, z) =\, & \frac{\gamma'(t)}{2} z^2 - \partial_t \Phi^{*} (t, z) + \frac{1}{2} \left( 1 - \frac{1}{\alpha} \right) \gamma(t) \\
		=\, & \frac{\gamma'(t)}{2} z^2 + \partial_t \Phi (t, x_t^{*} (z)) + \frac{1}{2} \left( 1 - \frac{1}{\alpha} \right) \gamma(t),
	\end{align*}
	and
	\begin{align*}
		\partial_z^2 V (t, z) = \gamma(t) - \partial_z^2 \Phi^{*} (t, z) = \gamma(t) - \frac{1}{\partial_x^2 \Phi (t, x_t^{*} (z))},
	\end{align*}
	where $x_t^{*} (z)$ is the unique solution to the equation $z = \partial_x \Phi (t, x)$. We thus obtain that
	\begin{align*}
		& \partial_t V (t, z) + \frac{1}{2} \frac{\gamma(t) \partial_z^2 V (t, z)}{\gamma(t) - \partial_z^2 V (t, z)} + \frac{\gamma(t)}{2 \alpha} \\
		=\, & \frac{\gamma'(t)}{2} z^2 + \partial_t \Phi (t, x_t^{*} (z)) + \frac{1}{2} \left( 1 - \frac{1}{\alpha} \right) \gamma(t) + \frac{\gamma(t)^2}{2} \partial_x^2 \Phi (t, x_t^{*} (z)) - \frac{\gamma(t)}{2} + \frac{\gamma(t)}{2 \alpha} \\
		=\, & \partial_t \Phi (t, x_t^{*} (z)) + \frac{\gamma(t)^2}{2} \partial_x^2 \Phi (t, x_t^{*} (z)) + \frac{\gamma'(t)}{2} \left( \partial_x \Phi (t, x_t^{*} (z)) \right)^2 = 0,
	\end{align*}
	where the last line follows from the Parisi PDE observed by $\Phi$. The terminal condition $V (1, z) = h(z)$ is quite straightforward to verify:
	\begin{equation*}
		V (1, z) = \frac{\gamma(1)}{2} z^2 - \Phi^{*} (1, z) = \frac{\gamma(1)}{2} z^2 - \left( \frac{\gamma(1)}{2} z^2 - h(z) \right)^{**} = h(z).
	\end{equation*}
	This proves that $V$ solves the HJB equation. 
	
	\vspace{0.5em} \noindent \textbf{Proof of $(b)$.}  Fix any $(t, z) \in [0, 1] \times \R$, we will prove that $V(t, z) = V_{\gamma} (t, z)$,  
	where the latter is defined by Eq.~\eqref{val_fct_gen}. To this end, we need to define the candidate process $(\phi_s^{\gamma})_{s \in [t, 1]}$ as follows (note that this process depends on $(t, z)$):
	\begin{enumerate}
		\item Let $(X_s^{\gamma})_{s \in [t, 1]}$ be the solution to the SDE:
		\begin{equation}\label{eq:Parisi_SDE}
			\d X_s^{\gamma} = \gamma(s) \d B_s + \gamma'(s) \partial_x \Phi \left( s, X_s^{\gamma} \right) \d s
		\end{equation}
		with initial condition $\partial_x \Phi (t, X_t^{\gamma}) = z$. The existence and uniqueness of $( X_s^{\gamma} )$ is guaranteed by strict convexity of $\Phi(t, x)$ and Lipschitzness of $\partial_x \Phi (t, x)$ with respect to $x$. \modif{Using \cref{eq:f_Phi_gen}, one can verify that $(X_s^{\gamma}) = (X_s^z)$ defined in \cref{eq:def_par_sde_gen}.}
		
		\item We then define for $s \in [t, 1]$:
		\begin{equation}\label{eq:candidate_proc}
			\phi_s^{\gamma} = \gamma(s) \cdot \partial_x^2 \Phi (s, X_s^{\gamma}) - 1.
		\end{equation}
		From the curvature bound on $\Phi$ (cf. \cref{thm:existence_parisi_gtr}), we know that $\phi_s^{\gamma}$ is almost surely bounded, uniformly for all $s \in [t, 1]$. \modif{Further, $(\phi_s^{\gamma}) = (\phi_s^z)$ in \cref{eq:def_ctrl_proc_gen}.}
	\end{enumerate}
	First, we show that $V_{\gamma}(t, z) \le V (t, z)$. Let $\phi \in D[t, 1]$ be an arbitrary control process, and define for $\theta \in [t, 1]$ the continuous martingale:
	\begin{equation*}
		M_{\theta}^{\phi} = z + \int_{t}^{\theta} \left( 1 + \phi_s \right) \d B_s.
	\end{equation*}
	Then, using It\^{o}'s formula, we obtain that
	\begin{align*}
		& \E \left[ V \left( \theta, M_{\theta}^{\phi} \right) \right] - V (t, z) \\
		=\, & \E \left[ \int_{t}^{\theta} \partial_s V \left( s, M_{s}^{\phi} \right) \d s + \partial_x V \left( s, M_{s}^{\phi} \right) \d M_s^{\phi} + \frac{1}{2} \partial_x^2 V \left( s, M_{s}^{\phi} \right) \d \<M^{\phi}\>_s  \right] \\
		=\, & \E \left[ \int_{t}^{\theta} \left( \partial_s V \left( s, M_{s}^{\phi} \right) + \frac{1}{2} \partial_x^2 V \left( s, M_{s}^{\phi} \right) \left( 1 + \phi_s \right)^2 \right) \d s \right] \\
		\stackrel{(i)}{\le}\, & \E \left[ \int_{t}^{\theta} \frac{\gamma(s)}{2} \left( \phi_s^2 - \frac{1}{\alpha} \right) \d s \right],
	\end{align*}
	where $(i)$ follows from the HJB equation, \modif{and the fact $\partial_z^2 V(t, z) < \gamma(t)$}. Sending $\theta \to 1^{-}$ and using the terminal condition $V (1, z) = h(z)$, we further deduce that
	\begin{align*}
		& \E \left[ h \left( z + \int_{t}^{1} \left( 1 + \phi_s \right) \d B_s \right) \right] - V (t, z) \le \E \left[ \int_{t}^{1} \frac{\gamma(s)}{2} \left( \phi_s^2 - \frac{1}{\alpha} \right) \d s \right] \\
		\implies \, & \E \left[ h \left( z + \int_{t}^{1} \left( 1 + \phi_s \right) \d B_s \right) - \frac{1}{2} \int_{t}^{1} \gamma(s) \left( \phi_s^2 - \frac{1}{\alpha} \right) \d s \right] \le V (t, z).
	\end{align*}
	Taking supremum over all $\phi \in D[t, 1]$ gives $V_{\gamma} (t, z) \le V (t, z)$.
	
	To show the reverse bound, it suffices to find an optimal control which achieves equality in $(i)$, namely
	\begin{equation}\label{eq:condition_optimal_control}
		\phi_s = \argmax_{\phi \in \R} \left\{ \partial_z^2 V \left( s, M_s^{\phi} \right) \left( 1 + \phi \right)^2 - \gamma(s) \phi^2 \right\} = \frac{\partial_z^2 V ( s, M_s^{\phi} )}{\gamma(s) - \partial_z^2 V ( s, M_s^{\phi} )}.
	\end{equation}
	(By definition of $V$ we know that $\partial_z^2 V ( s, M_s^{\phi} ) < \gamma(s)$, so the above $\argmax$ exists.)
	Next, we verify that the candidate process defined as per Eq.~\eqref{eq:candidate_proc} satisfies the above condition. We claim that
	\begin{equation*}
		M_{\theta}^{\phi^{\gamma}} = z + \int_{t}^{\theta} \left( 1 + \phi_s^{\gamma} \right) \d B_s = \partial_x \Phi \left( \theta, X_{\theta}^{\gamma} \right), \ \forall \theta \in [t, 1].
	\end{equation*}
	Note that our claim holds trivially for $\theta = t$ from the definition of $X_t^{\gamma}$. For $\theta > t$, applying It\^{o}'s formula yields
	\begin{align*}
		\d \partial_x \Phi \left( \theta, X_{\theta}^{\gamma} \right) =\, & \partial_{tx} \Phi \left( \theta, X_{\theta}^{\gamma} \right) \d \theta + \partial_{x}^2 \Phi \left( \theta, X_{\theta}^{\gamma} \right) \d X_{\theta}^{\gamma} + \frac{1}{2} \partial_{x}^3 \Phi \left( \theta, X_{\theta}^{\gamma} \right) \d \<X^{\gamma}\>_{\theta} \\
		=\, & \partial_{tx} \Phi \left( \theta, X_{\theta}^{\gamma} \right) \d \theta + \gamma(\theta) \partial_{x}^2 \Phi \left( \theta, X_{\theta}^{\gamma} \right) \d B_{\theta} + \gamma'(\theta) \partial_{x} \Phi \left( \theta, X_{\theta}^{\gamma} \right) \partial_{x}^2 \Phi \left( \theta, X_{\theta}^{\gamma} \right) \d \theta \\
		& + \frac{1}{2} \gamma(\theta)^2 \partial_{x}^3 \Phi \left( \theta, X_{\theta}^{\gamma} \right) \d \theta \\
		=\, & \partial_x \left( \partial_t \Phi \left( \theta, X_{\theta}^{\gamma} \right) + \frac{1}{2} \gamma'(\theta) \left( \partial_{x} \Phi \left( \theta, X_{\theta}^{\gamma} \right) \right)^2 + \frac{1}{2} \gamma(\theta)^2 \partial_{x}^2 \Phi \left( \theta, X_{\theta}^{\gamma} \right) \right) \d \theta \\
		& + \gamma(\theta) \partial_{x}^2 \Phi \left( \theta, X_{\theta}^{\gamma} \right) \d B_{\theta} \\
		\stackrel{(i)}{=} \, & \gamma(\theta) \partial_{x}^2 \Phi \left( \theta, X_{\theta}^{\gamma} \right) \d B_{\theta} = \left( 1 + \phi_{\theta}^{\gamma} \right) \d B_{\theta} = \d M_{\theta}^{\phi^{\gamma}},
	\end{align*}
	where $(i)$ follows from Eq.~\eqref{eq:Parisi_Phi_gen}. This proves the claim. It then follows that for any $s \in [t, 1]$:
	\begin{align*}
		\phi_{s}^{\gamma} = \gamma(s) \cdot \partial_x^2 \Phi (s, X_s^{\gamma}) - 1 = \frac{\gamma(s)}{\gamma(s) - \partial_z^2 V (s, M_s^{\phi^{\gamma}})} - 1 = \frac{\partial_z^2 V (s, M_s^{\phi^{\gamma}})}{\gamma(s) - \partial_z^2 V (s, M_s^{\phi^{\gamma}})}.
	\end{align*}
	This justifies Eq.~\eqref{eq:condition_optimal_control} and proves that $(\phi_s^{\gamma})_{s \in [t, 1]}$ is indeed an optimal control process, which implies that $V(t, z) = V_{\gamma} (t, z)$. Further, using \cref{eq:f_Phi_gen}, one can easily verify that $\phi^{\gamma} = \phi^z$, thus establishing the optimality of $\phi^z$. This completes the proof of Proposition~\ref{prop:veri_arg_SF_gen}.
\end{proof}

We now extend the conclusions of \cref{prop:veri_arg_SF_gen} to general $h \in C^2 (\R)$, thus completing the verification argument for the setting $\sup_{z \in \R} h''(z) < 1/c$. \modif{Note that, the regularity estimates in \cref{thm:solve_Parisi_PDE} guarantees that the SDE~\eqref{eq:def_par_sde_gen} has a unique solution, and the corresponding optimal control process is well-defined.} Similar to the proof of \cref{prop:veri_arg_SF_gen}, it suffices to show that for all $t, z$:
\begin{equation}\label{eq:V_and_f_relation_1}
	V_{\gamma} (t, z) =\, \inf_{x \in \R} \left\{ f_{\mu} (t, x) + \frac{\gamma(t)}{2} (x - z)^2 \right\} + \frac{1}{2 \alpha} \int_{t}^{1} \gamma(s) \d s,
\end{equation}
and that the supremum in the definition of $V_{\gamma} (t, z)$ is achieved at $(\phi_s^z)_{s \in [t, 1]}$ defined in Eqs.~\eqref{eq:def_ctrl_proc_gen} and \eqref{eq:def_par_sde_gen}.

Without loss of generality, we assume $(t, z) = (0, 0)$, as the proof for general $(t, z) \in [0, 1] \times \R$ is identical. \modif{For notational simplicity, we use $X$ and $\phi$ to denote $(X_s^z)_{s \in [t, 1]}$ and $(\phi_s^z)_{s \in [t, 1]}$ defined in \cref{eq:def_par_sde_gen,eq:def_ctrl_proc_gen}.} For $\veps > 0$, let $f_{\mu}^{\veps}(1, \,\cdot\,)$  be the $\veps$-mollification of $f_{\mu}(1,\,\cdot\,)$ as per \cref{eq:eps_mollifier}, and define
\begin{align*}
	h^{\veps}(z) = \, & \inf_{x \in \R} \left\{ f_{\mu}^{\veps}(1, x) + \frac{\gamma (1)}{2} (x - z)^2 \right\}, \\
	V_{\gamma}^{\veps} (0, 0) = \, & \sup_{\phi \in D[0, 1]} \E \left[ h^{\veps} \left( \int_{0}^{1} (1 + \phi_t) \d B_t \right) - \frac{1}{2} \int_{0}^{1} \gamma(t) \phi_t^2 \d t \right] + \frac{1}{2 \alpha} \int_{0}^{1} \gamma(t) \d t.
\end{align*}
Since $h_{\veps}$ satisfies Assumption~\ref{ass:h_regularity_gen} by construction, we apply \cref{prop:veri_arg_SF_gen} to obtain that
\begin{align*}
	V_{\gamma}^{\veps} (0, 0) = \, & \inf_{x \in \R} \left\{ f_{\mu}^{\veps} (0, x) + \frac{x^2}{2 (c + \int_{0}^{1} \mu(t) \d t)} \right\} + \frac{1}{2 \alpha} \int_{0}^{1} \frac{\d t}{c + \int_{t}^{1} \mu(s) \d s} \\
	= \, & \E \left[ h^{\veps} \left( \int_{0}^{1} (1 + \phi_t^{\veps}) \d B_t \right) - \frac{1}{2} \int_{0}^{1} \gamma(t) (\phi_t^{\veps})^2 \d t \right] + \frac{1}{2 \alpha} \int_{0}^{1} \gamma(t) \d t,
\end{align*}
where $\phi_t^{\veps} = \partial_x^2 f_{\mu}^{\veps} (t, X_t^{\veps}) / \gamma(t)$, and $\{ X_t^{\veps} \}_{t \in [0, 1]}$ solves the SDE
\begin{equation}\label{eq:parisi_SDE_gen_veps}
	\d X_t^{\veps} = \mu(t) \partial_x f_{\mu}^{\veps} (t, X_t^{\veps}) \d t + \d B_t, \ \partial_x f_{\mu}^{\veps} (0, X_0^{\veps}) + \gamma(0) X_0^{\veps} = 0.
\end{equation}
By \cref{prop:terminal_approximation}, we know that $f_{\mu}^{\veps} (1, \cdot) \to f_{\mu} (1, \cdot)$ uniformly, leading to $h^{\veps} \to h$ uniformly. \modif{Therefore, $V_{\gamma}^{\veps} (0, 0) \to V_{\gamma} (0, 0)$ as $\veps \to 0$.} Applying \cref{thm:C2_approx_gen} yields that $f_{\mu}^{\veps} (0, \cdot) \to f_{\mu} (0, \cdot)$ uniformly, hence
\begin{align*}
	& V_{\gamma} (0, 0) = \lim_{\veps \to 0^+} V_{\gamma}^{\veps} (0, 0) \\
	= \, & \lim_{\veps \to 0^+} \inf_{x \in \R} \left\{ f_{\mu}^{\veps} (0, x) + \frac{x^2}{2 (c + \int_{0}^{1} \mu(t) \d t)} \right\} + \frac{1}{2 \alpha} \int_{0}^{1} \frac{\d t}{c + \int_{t}^{1} \mu(s) \d s} \\
	= \, & \inf_{x \in \R} \left\{ f_{\mu} (0, x) + \frac{x^2}{2 (c + \int_{0}^{1} \mu(t) \d t)} \right\} + \frac{1}{2 \alpha} \int_{0}^{1} \frac{\d t}{c + \int_{t}^{1} \mu(s) \d s}.
\end{align*}
To prove that $\phi$ achieves the supremum in the definition of $V_{\gamma} (0, 0)$, it remains to show that
\begin{align*}
	& \lim_{\veps \to 0^+} \E \left[ h^{\veps} \left( \int_{0}^{1} (1 + \phi_t^{\veps}) \d B_t \right) - \frac{1}{2} \int_{0}^{1} \gamma(t) (\phi_t^{\veps})^2 \d t \right] \\
	= \, & \E \left[ h \left( \int_{0}^{1} (1 + \phi_t) \d B_t \right) - \frac{1}{2} \int_{0}^{1} \gamma(t) \phi_t^2 \d t \right].
\end{align*}
Using Proposition~\ref{prop:approximation_process_gen}, we know that $X^{\veps}$ converges in law to $X$. Theorem~\ref{thm:C2_approx_gen} $(c)$ then implies that $\phi^{\veps}$ converges in law to $\phi$. Further, since $\phi^{\veps}$ is uniformly bounded and $h^{\veps} \to h$ uniformly, the above claim immediately follows from bounded convergence theorem. This completes the proof of \cref{prop:two_stage_veri_arg} for general $h \in C^2 (\R)$.

\subsection{The case $\sup_{z \in \R} h''(z) \ge 1/c$}

As in \cref{sec:solve_Parisi_PDE}, we first reduce the problem to the setting $\sup_{z \in \R} h''(z) \le 1/c$. Recall from \cref{eq:equiv_f_mu_h_c} that the function $h_c$ (defined in \cref{eq:defn_h_c}) is equivalent to $h$ when defining the Parisi PDE and Parisi functional. In what follows, we show that $h_c$ and $h$ also yield the same value function $V_{\gamma}$.

\begin{prop}[Equivalence of $h_c$ and $h$]\label{prop:Equivalence_h_hc}
	Recall $V_{\gamma} (t, z)$ from Eq.~\eqref{val_fct_gen}. \modif{Under the same settings as \cref{prop:GeneralPDE_Constr},} we have
	\begin{equation*}
		V_{\gamma} (t, z) = \sup_{\phi \in D[t, 1]} \E \left[ h_c \left( z + \int_{t}^{1} (1 + \phi_s) \d B_s \right) - \frac{1}{2} \int_{t}^{1} \gamma(s) \phi_s^2 \d s \right] + \frac{1}{2 \alpha} \int_{t}^{1} \gamma(s) \d s,
	\end{equation*}
	which means that one can use $h_c$ instead of $h$ when defining $V_{\gamma}$.
\end{prop}
\begin{proof} 
	Again for simplicity, we assume $(t, z) = (0, 0)$, and use shorthand $V(\gamma)$ for $V_{\gamma} (0, 0)$. Under this simplification, we prove that $V(\gamma) = V_c (\gamma)$, where
	\begin{equation*}
		V_c (\gamma) = \sup_{\phi \in D[0, 1]} \E \left[ h_c \left( \int_{0}^{1} (1 + \phi_t) \d B_t \right) - \frac{1}{2} \int_{0}^{1} \gamma(t) \phi_t^2 \d t \right] + \frac{1}{2 \alpha} \int_{0}^{1} \gamma(t) \d t.
	\end{equation*}
	Without loss of generality, we can assume $\gamma (t) = \gamma(1)$ on $[\theta, 1]$ for some $\theta < 1$. Otherwise, one can find a sequence of $\gamma_n$, each satisfying $\gamma_n (t) = \gamma_n (1)$ on $[\theta_n, 1]$ for some $\theta_n < 1$, and $\gamma_n \stackrel{L^{\infty}}{\to} \gamma$. If we can show that $V_c (\gamma_n) = V (\gamma_n)$ for each $n$, then applying \cref{prop:VAL_continuity_gen} yields
	\begin{equation*}
		V_c (\gamma) = \lim_{n \to \infty} V_c (\gamma_n) = \lim_{n \to \infty} V (\gamma_n) = V (\gamma).
	\end{equation*}
	Based on this consideration, we will assume that $\gamma (t) = \gamma(1)$ on $[\theta, 1]$ for some $\theta < 1$. Since $h_c \ge h$, we always have $V_c (\gamma) \ge V (\gamma)$. To prove the reverse bound, note that for any $\veps > 0$, there exists $\phi \in D[0, 1]$ such that
	\begin{equation*}
		\E \left[ h_c \left( \int_{0}^{1} (1 + \phi_t) \d B_t \right) - \frac{1}{2} \int_{0}^{1} \gamma(t) \phi_t^2 \d t \right] + \frac{1}{2 \alpha} \int_{0}^{1} \gamma(t) \d t \ge V_c (\gamma) - \veps.
	\end{equation*}
	Further, by continuity, there exists $\theta_{\veps} \ge \theta$, such that
	\begin{equation*}
		\E \left[ h_c \left( \int_{0}^{\theta_{\veps}} (1 + \phi_t) \d B_t \right) - \frac{1}{2} \int_{0}^{\theta_{\veps}} \gamma(t) \phi_t^2 \d t \right] + \frac{1}{2 \alpha} \int_{0}^{1} \gamma(t) \d t \ge V_c (\gamma) - 2 \veps.
	\end{equation*}
	According to \cref{lem:var_rep_conc}, we have
	\begin{equation*}
		h_c (x) = \sup_{U \in L^2 (\Omega), \ \E [U] = 0} \E \left[ h (x + U) - \frac{\gamma(1)}{2} U^2 \right].
	\end{equation*}
	Using martingale representation theorem, we get that
	\begin{align*}
		& \E \left[ h_c \left( \int_{0}^{\theta_{\veps}} (1 + \phi_t) \d B_t \right) - \frac{1}{2} \int_{0}^{\theta_{\veps}} \gamma(t) \phi_t^2 \d t \right] \\
		\stackrel{(i)}{=} \, & \sup_{\psi \in D[\theta_{\veps}, 1]} \E \left[ h \left( \int_{0}^{\theta_{\veps}} (1 + \phi_t) \d B_t + \int_{\theta_{\veps}}^{1} \psi_t \d B_t \right) - \frac{1}{2} \int_{0}^{\theta_{\veps}} \gamma(t) \phi_t^2 \d t - \frac{1}{2} \int_{\theta_{\veps}}^{1} \gamma(t) \psi_t^2 \d t \right] \\
		\stackrel{(ii)}{\le} \, & \sup_{\psi \in D[\theta_{\veps}, 1]} \E \left[ h \left( \int_{0}^{\theta_{\veps}} (1 + \phi_t) \d B_t + \int_{\theta_{\veps}}^{1} ( 1 + \psi_t) \d B_t \right) - \frac{1}{2} \int_{0}^{\theta_{\veps}} \gamma(t) \phi_t^2 \d t - \frac{1}{2} \int_{\theta_{\veps}}^{1} \gamma(t) \psi_t^2 \d t \right] + \veps \\
		\le \, & V (\gamma) - \frac{1}{2 \alpha} \int_{0}^{1} \gamma(t) \d t + \veps
	\end{align*}
	for $\theta_{\veps}$ sufficiently close to $1$, where $(i)$ is due to $\gamma(t) = \gamma (1)$ on $[\theta_{\veps}, 1]$, $(ii)$ is because of the Lipschitzness of $h$. We thus deduce that
	\begin{align*}
		V_c (\gamma) - 2 \veps \le V (\gamma) + \veps.
	\end{align*}
	Sending $\veps \to 0$ yields that $V_c (\gamma) \le V (\gamma)$. This concludes the proof.
\end{proof}
From now on, we use $V_{\gamma}^{c} (\cdot, \cdot)$ to denote the value function when $h$ is replaced with $h_c$, which satisfies $\gamma(1) \ge \sup_{z \in \R} h_c''(z)$. The above proposition implies that $V_{\gamma}^{c} = V_{\gamma}$.

We are now ready to prove \cref{prop:two_stage_veri_arg}. First, note that \cref{eq:two_stage_V_and_f} follows immediately
from \cref{prop:Equivalence_h_hc} and \cref{prop:GeneralPDE_Constr}.
To see this, take a sequence $c_n<c$, with $c_n\uparrow c$,
and let $V^n_{\gamma}$, $f^n_{\mu}$ be defined as in the theorem statement with
$h$ replaced by $h_c$, and $(\mu, c)$ replaced by $(\mu,c_n)$.
Then we have $f^n_{\mu}\to f_{\mu}$ uniformly by \cref{prop:GeneralPDE_Constr}, and 
$V^n_{\gamma}\to V^{\gamma}$ uniformly by \cref{prop:Equivalence_h_hc} 
and \cref{prop:VAL_continuity_gen}.

Next, we will construct the corresponding optimal control process $\{ \phi_s^z \}_{s \in [t, 1]}$ and show that it achieves
\begin{equation*}
	V_{\gamma}^c (t, z) = \sup_{\phi \in D[t, 1]} \E \left[ h_c \left( z + \int_{t}^{1} (1 + \phi_s) \d B_s \right) - \frac{1}{2} \int_{t}^{1} \gamma(s) \phi_s^2 \d s \right] + \frac{1}{2 \alpha} \int_{t}^{1} \gamma(s) \d s.
\end{equation*}
One natural idea would be to use the same definition as before: 
$\phi_s^z = \partial_x^2 f_{\mu} (s, X_s^z) / \gamma(s)$,
where $(X_s^z)_{s \in [t, 1]}$ uniquely solves the SDE
\begin{equation*}
	\frac{1}{\gamma(t)} \partial_x f_{\mu} (t, X_t^z) + X_t^z = z, \ \d X_s^z = \mu(s) \partial_x f_{\mu} (s, X_s^z) \d s + \d B_s, \ s \in [t, 1].
\end{equation*}
\modif{(Existence and uniqueness follow from \cite[Proposition 1.10]{cherny2005singular}, since $\partial_x f_{\mu}$ is bounded.)}
However, due to the lack of existence and a priori estimates for third or higher-order partial derivatives of $f_{\mu}$ with respect to $x$ (they can not be established using Duhamel's principle or stochastic calculus techniques as before if we only assume $\mu \in L^1 [0, 1]$), we are not able to prove some key properties of $\{ \phi_s^z \}$ that are crucial to the verification argument. To circumvent this difficulty, we will construct $\{ \phi_s^z \}$ via martingale representation theorem instead, and show that such defined $\{ \phi_s^z \}$ has desired properties. Namely, defining
\begin{equation}
	M_s^z = \frac{1}{\gamma(s)} \partial_x f_{\mu} (s, X_s^z) + X_s^z,\label{eq:MDef}
\end{equation}
we then have the following:
\begin{lem}\label{lem:phi_via_martingale}
	$\{ M_s^z \}_{s \in [t, 1]}$ is a square integrable martingale with respect to the standard filtration.
\end{lem}
\begin{proof}
	Without loss of generality, we assume $(t, z) = (0, 0)$, and drop the superscript ``$z$" from now on. Using the definition of $\{ X_s \}$, and the fact that $\partial_xf_{\mu}(t,\,\cdot\,)$ is bounded (by \cref{prop:GeneralPDE_Constr}), it is easy to see that $\{ M_s \}$ is square integrable with $M_0 = 0$. It then remains to show that $\E [M_s-M_u \vert \cF_u] = 0$ for any $u < s$. Since $\{ X_s \}$ is a Markov process and $M_s$ only depends on $X_s$, it suffices to prove that $\E [M_s-M_u \vert X_u = x] = 0$ for any $x \in \R$. For simplicity, we will show that $\E [M_s \vert X_0] = \E [M_s] = 0$, as the proof for general $u$ and $x$ is similar. To this end, let $\{ X_s^n \}_{s \in [0, 1]}$ be the solution to the SDE:
	\begin{align*}
		&\d X_s^n = \, \mu(s) \partial_x f_\mu^n (s, X_s^n) \d s + \d B_s, \\
		&\partial_x f_{\mu}^n (0, X_0^n) + \gamma(0) X_0^n = 0\, ,
	\end{align*}
	where $f^n_{\mu}$ is defined as in \cref{prop:GeneralPDE_Constr}. Then, we know that
	\begin{equation*}
		M_s^n = \frac{1}{\gamma_n (s)} \partial_x f_{\mu}^n (s, X_s^n) + X_s^n
	\end{equation*}
	is a martingale. Hence, $\E [M_s^n] = 0$. Further, Proposition~\ref{prop:approximation_process_gen} implies that $M_s^n$ converges to $M_s$ in distribution as $n \to \infty$. Now since $\{ M_s^n \}_{n=1}^{\infty}$ is a family of uniformly integrable random variables (easily seen from its definition), we know that
	\begin{equation*}
		\E [M_s] = \lim_{n \to \infty} \E [M_s^n] = 0.
	\end{equation*}
	This completes the proof.
\end{proof}
According to martingale representation theorem, there exists $\{ \phi_s^z \} \in D [t, 1]$ such that
\begin{equation}\label{eq:def_phi_martingale}
	M_s^z = M_t^z + \int_{t}^{s} (1 + \phi_u^z) \d B_u, \ \forall s \in [t, 1].
\end{equation}
The proposition below shows that $\{ \phi_s^z \}_{s \in [t, 1]}$ indeed achieves $V_{\gamma}^c (t, z)$, \modif{thus completing the verification argument for the setting $\sup_{z \in \R} h''(z) \ge 1/c$.}
\begin{prop}\label{prop:less_achievability}
	For $\{ \phi_s^z \}_{s \in [t, 1]}$ defined as per Eq.~\eqref{eq:def_phi_martingale}, we have
	\begin{equation}\label{eq:veri_arg_gen_equal}
		V_{\gamma}^c (t, z) = \E \left[ h_c \left( z + \int_{t}^{1} (1 + \phi_s^z) \d B_s \right) - \frac{1}{2} \int_{t}^{1} \gamma(s) (\phi_s^z)^2 \d s \right] + \frac{1}{2 \alpha} \int_{t}^{1} \gamma(s) \d s,
	\end{equation}
	i.e., $\{ \phi_s^z \}_{s \in [t, 1]}$ is indeed optimal.
\end{prop}

\begin{proof}
	Similar as before, we assume $(t, z) = (0, 0)$, drop the superscript ``$z$", and use $V_c (\gamma)$ as a shorthand for $V_{\gamma}^c (0, 0)$.
	Let $\phi^n \in D[0, 1]$ be the optimal control process associated with $h_c$ and $(\mu, c_n)$, where $c_n \to c$ from below. Then, we know that
	\begin{equation*}
		V_c (\gamma_n) = \E \left[ h_{c} \left( \int_{0}^{1} (1 + \phi_t^n) \d B_t \right) - \frac{1}{2} \int_{0}^{1} \gamma_n (t) (\phi_t^n)^2 \d t \right] + \frac{1}{2 \alpha} \int_{0}^{1} \gamma_n (t) \d t.
	\end{equation*}
	Since $V_c (\gamma_n) \to V_c (\gamma)$, $\gamma_n (t) \to \gamma(t)$ in $L^{\infty} [0, 1]$, it suffices to show that
	\begin{equation*}
		\begin{split}
			& \E \left[ h_{c} \left( \int_{0}^{1} (1 + \phi_t^n) \d B_t \right) - \frac{1}{2} \int_{0}^{1} \gamma_n (t) (\phi_t^n)^2 \d t \right] \\ 
			\to \, & \E \left[ h_c \left( \int_{0}^{1} (1 + \phi_t) \d B_t \right) - \frac{1}{2} \int_{0}^{1} \gamma(t) \phi_t^2 \d t \right]
		\end{split}
	\end{equation*}
	as $n \to \infty$. To this end, we will prove
	\begin{equation}\label{eq:limit_hc}
		\E \left[ h_{c} \left( \int_{0}^{1} (1 + \phi_t^n) \d B_t \right) \right] \to \E \left[ h_c \left( \int_{0}^{1} (1 + \phi_t) \d B_t \right) \right]
	\end{equation}
	and
	\begin{equation}\label{eq:limit_quad}
		\frac{1}{2} \int_{0}^{1} \gamma_n (t) \E \left[ (\phi_t^n)^2 \right] \d t \to \frac{1}{2} \int_{0}^{1} \gamma (t) \E \left[ \phi_t^2 \right] \d t,
	\end{equation}
	respectively. First, note that
	\begin{equation*}
		\int_{0}^{1} (1 + \phi_t) \d B_t = M_1, \ \int_{0}^{1} (1 + \phi_t^n) \d B_t = M_1^n,
	\end{equation*}
	where $\{ M_t^n \}_{t \in [0, 1]}$ is defined in the proof of Lemma~\ref{lem:phi_via_martingale}. Similarly as in that proof, we know that $\{ M_1^n \}$ is a sequence of uniformly integrable random variables that converges to $M_1$ in law. Since $h_c$ is Lipschitz, we know that $\E [h_c (M_1^n)] \to \E [h_c (M_1)]$. This proves Eq.~\eqref{eq:limit_hc}. To show Eq.~\eqref{eq:limit_quad}, let us define
	\begin{equation*}
		A(t) = \int_{0}^{t} \E [\phi_s^2] \d s, \quad N_t = M_t - B_t, \ \forall t \in [0, 1].
	\end{equation*}
	Then, by definition, we know that $A(t) = \E [N_t^2]$. Further,
	\begin{align*}
		\frac{1}{2} \int_{0}^{1} \gamma (t) \E \left[ \phi_t^2 \right] \d t = \, \frac{1}{2} \int_{0}^{1} \gamma (t) \d A(t)  = \frac{1}{2} \left( \gamma(1) A(1) - \int_{0}^{1} \gamma'(t) A(t) \d t \right).
	\end{align*}
	Similarly,
	\begin{equation*}
		\frac{1}{2} \int_{0}^{1} \gamma_n (t) \E \left[ (\phi_t^n)^2 \right] \d t = \frac{1}{2} \left( \gamma_n (1) A_n (1) - \int_{0}^{1} \gamma_n'(t) A_n(t) \d t \right).
	\end{equation*}
	Note that for all $n \in \mathbb{N}$ and $t \in [0, 1]$,
	\begin{align*}
		N_t^n = \, & M_t^n - B_t = \frac{1}{\gamma_n (t)} \partial_x f_{\mu}^n (t, X_t^n) + X_t^n - B_t \\
		= \, & \frac{1}{\gamma_n (t)} \partial_x f_{\mu}^n (t, X_t^n) + \int_{0}^{t} \mu(s) \partial_x f_{\mu}^{n} (s, X_s^n) \d s
	\end{align*}
	is uniformly bounded. Applying Proposition~\ref{prop:approximation_process_gen} and bounded convergence theorem, we deduce that $A_n (t) = \E [(N_t^n)^2] \to \E [N_t^2] = A(t)$ as $n \to \infty$. Further, $\gamma_n' (t) = \gamma_n (t)^2 \mu(t) \to \gamma(t)^2 \mu(t) = \gamma'(t)$, thus leading to (use dominated convergence theorem)
	\begin{align*}
		& \frac{1}{2} \int_{0}^{1} \gamma_n (t) \E \left[ (\phi_t^n)^2 \right] \d t = \, \frac{1}{2} \left( \gamma_n (1) A_n (1) - \int_{0}^{1} \gamma_n'(t) A_n(t) \d t \right) \\
		\to \, & \frac{1}{2} \left( \gamma (1) A (1) - \int_{0}^{1} \gamma'(t) A (t) \d t \right) = \frac{1}{2} \int_{0}^{1} \gamma (t) \E \left[ \phi_t^2 \right] \d t.
	\end{align*}
	This proves Eq.~\eqref{eq:limit_quad} and concludes Eq.~\eqref{eq:veri_arg_gen_equal}.
\end{proof}

\section{Analysis of the variational problem}\label{sec:variational_problem}

\subsection{Proof of \cref{prop:first_order_var_Parisi}}\label{sec:greater_than_case}

As in \cref{sec:solve_Parisi_PDE,sec:verification_argument}, we first prove \cref{prop:first_order_var_Parisi} under the additional \cref{ass:h_regularity_gen}, and then remove this assumption via an approximation argument. Recall that the Parisi functional is defined as
\begin{equation}\label{parisi_fct_gen}
	\mathsf{F} (\mu, c) = f_{\mu} (0, 0)
	+ \frac{1}{2 \alpha} \int_{0}^{1} \frac{\d t}{c + \int_{t}^{1} \mu(s) \d s}, \quad (\mu, c) \in \FS,
\end{equation}
where $f_{\mu}$ solves the Parisi PDE~\eqref{eq:parisi_mu_gen}.

We begin by stating a useful lemma.
\begin{lem}\label{lem:1st_2nd_moment_gen}
	\modif{Under the settings of \cref{thm:existence_parisi_gtr}}, we have for any $s, t \in [0, 1]$:
	\begin{equation*}
		\E \left[ \left( \partial_x f_{\mu} (t, X_t) \right)^2 \right] - \E \left[ \left( \partial_x f_{\mu} (s, X_s) \right)^2 \right] = \int_{s}^{t} \E \left[ \left( \partial_x^2 f_{\mu} (u, X_u) \right)^2 \right] \d u.
	\end{equation*}
\end{lem}
\begin{proof}
	This follows from a straightforward application of It\^{o}'s formula.
\end{proof}

\vspace{0.5em} \noindent \textbf{Proof of $(i)$.} The claim $X_q = B_q \sim \normal(0, q)$ follows directly from definition. To prove \cref{eq:var_rep_F_opt}, we note that the proof of \cref{prop:two_stage_veri_arg} in \cref{sec:verification_argument} implies that $\phi_t$ defined in \eqref{eq:two_stage_opt_control} achieves the value $V_{\gamma}$, whence
\begin{align*}
	& \E \left[ V_{\gamma} \left( q, X_q + \frac{1}{\alpha} F(X_q) \right) \right] \\
	= \, &
	\E \left[ h \left( X_q + \frac{1}{\alpha} F(X_q) + \int_{q}^{1} \left( 1 + \phi_t \right) \d B_t \right) - \frac{1}{2} \int_{q}^{1} \gamma(t) \left( \phi_t^2 - \frac{1}{\alpha} \right) \d t \right].
\end{align*}
Hence, the right hand side of \cref{eq:var_rep_F_opt} equals
\begin{align*}
	\E \left[ V_{\gamma} \left( q, X_q + \frac{1}{\alpha} F(X_q) \right) - \frac{\gamma(q)}{2 \alpha} \left( \frac{1}{\alpha} F(X_q)^2 - q \right) \right].
\end{align*}
By our choice of $F$ and \cref{prop:two_stage_veri_arg}, it follows that
\begin{equation*}
	f_{\mu} (\modif{q}, x) + \frac{1}{2 \alpha} \int_{q}^{1} \gamma(s) \d s = V_{\gamma} \left( q, x + \frac{1}{\alpha} F(x) \right) - \frac{\gamma(q)}{2 \alpha^2} F(x)^2, \ \forall x \in \R,
\end{equation*}
thus leading to
\begin{align*}
	& \E \left[ V_{\gamma} \left( q, X_q + \frac{1}{\alpha} F(X_q) \right) - \frac{\gamma(q)}{2 \alpha} \left( \frac{1}{\alpha} F(X_q)^2 - q \right) \right] \\
	= \, & \E \left[ f_{\mu} (\modif{q}, X_q) \right] + \frac{1}{2 \alpha} \left( q \gamma(q) + \int_{q}^{1} \gamma(s) \d s \right) \\
	\stackrel{(i)}{=} \, & f_{\mu} (0, 0) + \frac{1}{2 \alpha} \int_{0}^{1} \gamma(s) \d s = \mathsf{F} (\mu, c),
\end{align*}
where $(i)$ follows from the fact that $\mu \equiv 0$ on $[0, q]$, so the Parisi PDE degenerates to a heat equation.
This completes the proof of part $(i)$.

\vspace{0.5em} \noindent \textbf{Proof of $(ii)$.} This is a direct consequence of \cref{lem:1st_2nd_moment_gen} and \cref{thm:existence_parisi_gtr}.

\vspace{0.5em} \noindent \textbf{Proof of $(iii)$.} The proof is similar to Proposition 6.8 in \cite{el2021optimization}. Recall that
\begin{equation}
	\mathsf{F} (\mu, c) = f_{\mu} (0, 0) + \frac{1}{2 \alpha} \int_{0}^{1} \frac{\d t}{c + \int_{t}^{1} \mu(s) \d s} =:  f_{\mu} (0, 0) +
	\mathsf{S}(\mu, c).\label{eq:EntropyDerivative_0}
\end{equation}
Then, it is easy to see that the first-order variation of the entropy term $\mathsf{S}(\mu, c)$ equals
\begin{equation}
	\frac{\d}{\d s}  \mathsf{S}(\mu+s\delta, c)  \Bigg\vert_{s=0} =- \frac{1}{2 \alpha} \int_{0}^{1} \delta(t) \int_{0}^{t} \gamma(s)^2 \d s \d t,\label{eq:EntropyDerivative}
\end{equation}
so we only need to show that
\begin{equation}\label{eq:var_first_gen}
	\begin{split}
		\frac{\d}{\d s} f_{\mu + s \delta} (0, 0) \Bigg\vert_{s=0}
		= \, \frac{1}{2} \int_{0}^{1} \delta(t) \E \left[ \left( \partial_x f_\mu (t, X_t) \right)^2 \right] \d t.
	\end{split}
\end{equation}
To this end, we rewrite $f_{\mu + s \delta}$ as $f_s$. Similar to the proof of Lemma 14 in \cite{jagannath2016dynamic}, we obtain that
\begin{equation}\label{eq:int_first_gen}
	f_s (0, x) - f_0 (0, x) = \frac{s}{2} \int_{0}^{1} \delta(t) \E_{X_0^s = x} \left[ \left( \partial_x f_{0} (t, X_t^s) \right)^2 \right] \d t,
\end{equation}
where $\{ X_t^s \}_{t \in [0, 1]}$ is the unique solution to the SDE:
\begin{equation*}
	\d X_t^s = \mu(t) \frac{\partial_x f_s + \partial_x f_0}{2} (t, X_t^s) \d t + \d B_t, \ X_0^s = x.
\end{equation*}
Now since $\partial_x f_s \to \partial_x f$ as $s \to 0$ (via a similar argument as in the proof of \cite[Lemma 14]{jagannath2016dynamic}), further they are continuous and uniformly bounded and $\mu \in L^1 [0, 1]$, we deduce from Proposition~\ref{prop:approximation_process_gen} that $\operatorname{Law} (X^s) \stackrel{w}{\Rightarrow} \operatorname{Law} (X)$ as $s \to 0$. By bounded convergence theorem,
\begin{equation*}
	\int_{0}^{1} \delta(t) \E_{X_0^s = x} \left[ \left( \partial_x f_{0} (t, X_t^s) \right)^2 \right] \d t = \int_{0}^{1} \delta(t) \E_{X_0 = x} \left[ \left( \partial_x f_{0} (t, X_t) \right)^2 \right] \d t + o_s (1), \ s \to 0,
\end{equation*}
which further implies that
\begin{equation*}
	\frac{\d}{\d s} f_s (0, x)  \Bigg\vert_{s=0}= \frac{1}{2} \int_{0}^{1} \delta(t) \E_{X_0 = x} \left[ \left( \partial_x f_{0} (t, X_t) \right)^2 \right] \d t \, .
\end{equation*}
The desired result \eqref{eq:var_first_gen} follows by taking $x = 0$.

\vspace{0.5em}

We next extend our proof to general $h \in C^2 (\R)$ (not necessarily satisfying \cref{ass:h_regularity_gen}), thus completing the proof of \cref{prop:first_order_var_Parisi}.

\vspace{0.5em} \noindent \textbf{Proof of \cref{prop:first_order_var_Parisi} for $h \in C^2$.} The proof of part $(i)$ follows from a similar approximation argument as in the proof of \cref{prop:two_stage_veri_arg} for $h \in C^2$. For any $\veps > 0$, denote by $\sF^{\veps} (\mu, c)$ the Parisi functional associated with $f_{\mu}^{\veps}$ and $h^{\veps}$. We can then define $F^{\veps}$ and $\phi^{\veps}$ accordingly. Further, we know that \cref{eq:var_rep_F_opt} holds for $\sF^{\veps} (\mu, c)$, $h^{\veps}$, $F^{\veps}$ and $\phi^{\veps}$. From the proof of \cref{prop:two_stage_veri_arg} for $h \in C^2$, we know that $f_{\mu}^{\veps} \to f_{\mu}$ uniformly, $h^{\veps} \to h$ uniformly, $\phi^{\veps} \to \phi$ in law, and $\partial_x f_{\mu}^{\veps} \to \partial_x f_{\mu}$ uniformly, which further implies that $\sF^{\veps} (\mu, c) \to \sF (\mu, c)$ and $F^{\veps} \to F$ uniformly. Similar to the proof of \cref{prop:two_stage_veri_arg}, sending $\veps \to 0$ and applying dominated convergence theorem yields \cref{eq:var_rep_F_opt}, completing the proof of part $(i)$.

To prove part $(ii)$, we need to extend Lemma~\ref{lem:1st_2nd_moment_gen}, which again follows from Proposition~\ref{prop:approximation_process_gen},  Theorem~\ref{thm:C2_approx_gen}, and the same approximation argument. Indeed, as $X^{\veps} \to X$ in law, $\partial_x f_{\mu}^{\veps}, \partial_x^2 f_{\mu}^{\veps} \to \partial_x f_{\mu}, \partial_x^2 f_{\mu}$ uniformly over compact sets, and the limiting functions are bounded, we deduce from bounded convergence theorem that $\forall t \in [0, 1]$:
\begin{align*}
	& \lim_{\veps \to 0} \E \left[ \left( \partial_x f_{\mu}^{\veps} (t, X_t^{\veps}) \right)^2 \right] = \E \left[ \left( \partial_x f_{\mu} (t, X_t) \right)^2 \right], \\ 
	& \lim_{\veps \to 0} \E \left[ \left( \partial_x^2 f_{\mu}^{\veps} (t, X_t^{\veps}) \right)^2 \right] = \E \left[ \left( \partial_x^2 f_{\mu} (t, X_t) \right)^2 \right].
\end{align*}
Further, the above convergence is uniform in $t$. Part $(ii)$ then follows from the conclusion of Lemma~\ref{lem:1st_2nd_moment_gen} for $f_{\mu}^{\veps}$ and taking the limit $\veps \to 0$.

It now remains to show part $(iii)$. 
To this end, note that Eq.~\eqref{eq:int_first_gen} still holds since its proof does not involve third or higher-order partial derivatives with respect to $x$. It then suffices to show that $\partial_x f_s \to \partial_x f$ as $s \to 0$. Re-examining the proof of Theorem~\ref{thm:C2_approx_gen} $(b)$, we know that $\partial_x f_s^{\veps} \to \partial_x f_s$ uniformly for $s$ in a neighborhood of $0$ as $\veps \to 0$, since the error bound depends continuously on $\mu$. Further, for any fixed $\veps > 0$, we have $\lim_{s \to 0} \partial_x f_{s}^{\veps} = \partial_x f^{\veps}$, which follows in a similarly way as the proof of \cref{thm:existence_parisi_gtr}. We thus conclude that $\lim_{s \to 0} \partial_x f_s = \partial_x f$, and the same approximation argument as in the proof of part $(iii)$ for the $C^4$ case follows. This concludes the proof of \cref{prop:first_order_var_Parisi} $(iii)$ for $h \in C^2$.

\subsection{Proof of strong duality for $1 / c_* \le \sup_{z \in \R} h''(z)$}\label{sec:less_than_case}

In this section, we prove part $(c)$ of \cref{thm:two_stage_strong_duality} under the setting $\gamma_* (1) = 1 / c_* \le \sup_{z \in \R} h''(z)$. To this end, we need to compute the first-order variation of $\mathsf{F}$ with respect to both $\mu$ and $c$.

Fix any $(\mu, c) \in \FS (q)$. Let $\{ X_t \}_{t \in [0, 1]}$ solve the Parisi SDE~\eqref{eq:two_stage_SDE}, and define $\{ \phi_t \}_{t \in [0, 1]}$ as per \cref{eq:def_phi_martingale}, namely, via martingale representation theorem. We also define
\begin{equation}\label{eq:def_F_less_than}
	F(x) = \frac{\alpha}{\gamma(q)} \partial_x f_{\mu} (q, x).
\end{equation}
Recall from \cref{sec:verification_argument} that $V_{\gamma}^{c} (\cdot, \cdot)$ denotes the value function when $h$ is replaced by $h_c$ (defined in \eqref{eq:defn_h_c}). In the rest of this section, we also define $\mathsf{F}_c (\mu, c)$ as the Parisi functional associated with $h_c$. Then, \cref{eq:equiv_f_mu_h_c} implies that $\mathsf{F}_c (\mu, c) = \mathsf{F} (\mu, c)$.
Similar to the proof of \cref{prop:first_order_var_Parisi} $(i)$, we can use Propositions \ref{prop:two_stage_veri_arg} and 
\ref{prop:less_achievability} to show that
\begin{equation}\label{eq:Fc_DualExpression}
	\begin{split}
		\mathsf{F}_c (\mu, c) 
		= \, \E \bigg[ \, & h_c \left( X_q + \frac{1}{\alpha} F(X_q) + \int_{q}^{1} \left( 1 + \phi_t \right) \d B_t \right) - \frac{1}{2} \int_{q}^{1} \gamma(t) \left( \phi_t^2 - \frac{1}{\alpha} \right) \d t \\
		& - \frac{\gamma(q)}{2 \alpha} \left( \frac{1}{\alpha} F(X_q)^2 - q \right) \bigg].
	\end{split}
\end{equation}

Further, we establish the following result, which is analogous to \cref{prop:first_order_var_Parisi} $(ii)$:
\begin{prop}\label{prop:1st_2nd_moment_gen_equal}
	For any $0 \le s < t \le 1$, we have
	\begin{equation*}
		\E \left[ \left( \partial_x f_{\mu} (t, X_t) \right)^2 \right] - \E \left[ \left( \partial_x f_{\mu} (s, X_s) \right)^2 \right] = \int_{s}^{t} \gamma(u)^2 \E \left[ \phi_u^2 \right] \d u.
	\end{equation*}
\end{prop}
\begin{proof}
	Without loss of generality we assume $s=0$. Recall the definition of $A_n (t)$ and $A(t)$ in the proof of \cref{prop:less_achievability}. Using integration by parts, we obtain that $\forall t \in [0, 1]$:
	\begin{align*}
		\int_{0}^{t} \gamma(s)^2 \E \left[ \phi_s^2 \right] \d s = \, & \int_{0}^{t} \gamma(s)^2 \d A(s)  = \gamma(t)^2 A(t) - 2 \int_{0}^{t} \gamma(s) \gamma'(s) A(s) \d s, \\
		\int_{0}^{t} \gamma_n (s)^2 \E \left[ (\phi_s^n)^2 \right] \d s = \, & \int_{0}^{t} \gamma_n (s)^2 \d A_n(s)  = \gamma_n(t)^2 A_n(t) - 2 \int_{0}^{t} \gamma_n (s) \gamma_n'(s) A_n(s) \d s.
	\end{align*}
	Since $\gamma_n \to \gamma$, $\gamma_n' \to \gamma'$, $A_n \to A$, by dominated convergence theorem we know that
	\begin{equation*}
		\lim_{n \to \infty} \int_{0}^{t} \gamma_n (s)^2 \E \left[ (\phi_s^n)^2 \right] \d s = \int_{0}^{t} \gamma(s)^2 \E \left[ \phi_s^2 \right] \d s.
	\end{equation*}
	According to Lemma~\ref{lem:1st_2nd_moment_gen}, we have
	\begin{equation*}
		\E \left[ \left( \partial_x f_{\mu}^n (t, X_t^n) \right)^2 \right] - \partial_x f_{\mu}^n (0, X_0^n)^2 = \int_{0}^{t} \gamma_n (s)^2 \E \left[ (\phi_s^n)^2 \right] \d s.
	\end{equation*}
	Further, since $\partial_x f_{\mu}^n$ uniformly converges to $\partial_x f_{\mu}$ on any compact set, we know that $\partial_x f_{\mu}^n (t, X_t^n) \to \partial_x f_{\mu} (t, X_t)$ in distribution. By bounded convergence theorem,
	\begin{equation*}
		\lim_{n \to \infty} \E \left[ \left( \partial_x f_{\mu}^n (t, X_t^n) \right)^2 \right] = \E \left[ \left( \partial_x f_{\mu} (t, X_t) \right)^2 \right].
	\end{equation*}
	This completes the proof.
\end{proof}
We are now ready to compute the first-order variation of $\sF$ with respect to $\mu$, extending \cref{prop:first_order_var_Parisi} $(iii)$ to the case $\gamma(1) \le \sup_{z \in \R} h''(z)$.
\begin{prop}\label{prop:first_var_less_than}
	Assume that $\delta: [0, 1] \to \R$ is in $L^1 [0, 1]$, $ \delta \vert_{[0, t]} \in L^{\infty} [0, t]$ for any $t \in [0, 1)$, and that $\delta \equiv 0$ on $[0, q]$. Then, $(\mu + s \delta, c) \in \mathscr{L} (q)$ for sufficiently small $s \in \R$, and
	\begin{equation*}
		\frac{\d}{\d s} \mathsf{F} \left( \mu + s \delta, c \right) \bigg\vert_{s=0} = \,  \frac{1}{2} \int_{q}^{1} \delta(t) \left( \E \left[ \left( \partial_x f_\mu (t, X_t) \right)^2 \right] - \frac{1}{\alpha} \int_{0}^{t} \gamma(s)^2 \d s \right) \d t.
	\end{equation*}
\end{prop}
\begin{proof}
	As in the proof of \cref{prop:first_order_var_Parisi} $(iii)$  
	(cf. Eq.~\eqref{eq:int_first_gen}),
	we still have
	\begin{equation*}
		f_{\mu + s \delta} (0, x) - f_{\mu} (0, x) = \frac{s}{2} \int_{0}^{1} \delta(t) \E_{X_0^s = x} \left[ \left( \partial_x f_{\mu} (t, X_t^s) \right)^2 \right] \d t,
	\end{equation*}
	where $\{ X_t^s \}_{t \in [0, 1]}$ uniquely solves the SDE:
	\begin{equation*}
		\d X_t^s = \mu(t) \frac{\partial_x f_{\mu + s \delta} + \partial_x f_{\mu}}{2} (t, X_t^s) \d t + \d B_t, \ X_0^s = x,
	\end{equation*}
	since the proof of this identity does not involve third of higher-order partial derivatives of $f_{\mu}$ with respect to $x$.
	Then, we know that $\partial_x f_{\mu + s \delta} \to \partial_x f_{\mu}$ as $s \to 0$ almost everywhere, which follows  from a similar argument as that in the proof of \cref{prop:GeneralPDE_Constr}, and the following facts: $(a)$ $f_{\mu + s \delta} \to f_{\mu}$ uniformly as $s \to 0$, which can be established using Feynman-Kac formula, see also in the proof of \cref{thm:existence_parisi_gtr} and \cite[Lemma 14]{jagannath2016dynamic}, 
	$(b)$ $\partial_x f_{\mu + s \delta}$ and $\partial_x f_{\mu}$ are continuous, uniformly bounded, and of bounded variation on any finite interval, which follows from \cref{prop:GeneralPDE_Constr}.
	
	As a consequence, we deduce similarly that
	\begin{equation*}
		\frac{\d}{\d s} f_{\mu + s \delta} (0, x) \bigg\vert_{s = 0} = \frac{1}{2} \int_{0}^{1} \delta(t) \E_{X_0 = x} \left[ \left( \partial_x f_{\mu} (t, X_t) \right)^2 \right] \d t,
	\end{equation*}
	which concludes the calculation, as the first-order variation of the entropy term ${\mathsf{S}}(\mu, c)$ is still the same,
	cf. Eqs.~\eqref{eq:EntropyDerivative_0} and \eqref{eq:EntropyDerivative}.
\end{proof}

The lemma below computes the first derivative of $\sF(\mu, c)$ with respect to $c$ for fixed $\mu$:
\begin{lem}[First derivative with repsect to $c$]\label{lem:diff_P_c_equal}
	We have (note that $\frac{\d}{\d c} \mathsf{F}_c (\mu, c)$ represents the derivative of $\mathsf{F}_c (\mu, c)$ taken with respect to the ``$c$'' in both its subscript and second argument):
	\begin{equation*}
		\frac{\d}{\d c} \mathsf{F} (\mu, c) = \frac{\d}{\d c} \mathsf{F}_c (\mu, c) = \, \E \left[ g_c ( M_1 ) \right] + \frac{1}{2} \E \left[ \partial_x f_{\mu} (1, X_1)^2 \right] - \frac{1}{2 \alpha} \int_{0}^{1} \gamma(t)^2 \d t ,
	\end{equation*}
	where we recall that $M_t$ is the martingale defined by Eq.~\eqref{eq:MDef}, and define $g_c (x) := (\partial / \partial c) h_c (x)$ (existence is guaranteed by monotonicity and convexity with respect to $1/c$).
\end{lem}
\begin{proof}
	We denote by $f_{\mu}^c$ the solution to Parisi PDE~\eqref{eq:parisi_mu_gen_c} to emphasize its dependence on $c$. Recall that
	\begin{equation*}
		\mathsf{F}_c (\mu, c) = f_{\mu}^c (0, 0) + \frac{1}{2 \alpha} \int_{0}^{1} \frac{\d t}{c + \int_{t}^{1} \mu(s) \d s}.
	\end{equation*}
	By dominated convergence theorem, we know that
	\begin{equation*}
		\frac{\d}{\d c} \left( \frac{1}{2 \alpha} \int_{0}^{1} \frac{\d t}{c + \int_{t}^{1} \mu(s) \d s} \right) = \frac{1}{2 \alpha} \int_{0}^{1} - \frac{\d t}{(c + \int_{t}^{1} \mu(s) \d s)^2} = - \frac{1}{2 \alpha} \int_{0}^{1} \gamma(t)^2 \d t.
	\end{equation*}
	It then suffices to compute $\d f_{\mu}^c (0, x) / \d c$ for each $x \in \R$, then we can just take $x = 0$. Using stochastic calculus, we know that for $c, c' > 0$,
	\begin{equation*}
		f_{\mu}^{c'} (0, x) - f_{\mu}^c (0, x) = \E_{X_0' = x} \left[ \left(f_{\mu}^{c'} - f_{\mu}^c \right) (1, X_1') \right],
	\end{equation*}
	where $\{ X_t' \}_{t \in [0, 1]}$ solves the SDE
	\begin{equation*}
		\d X_t' = \frac{1}{2} \mu(t) \left( \partial_x f_{\mu}^{c} + \partial_x f_{\mu}^{c'} \right) (t, X_t') \d t + \d B_t, \ X_0' = x.
	\end{equation*}
	As $c' \to c$, we know that $f_{\mu}^{c'}$ converges uniformly to $f_{\mu}^{c}$, which follows from a similar argument as in the proof of \cref{prop:GeneralPDE_Constr}. 
	Further since $\partial_x f_{\mu}^{c'}$ and $\partial_x f_{\mu}^{c}$ are continuous and have bounded total variation on any finite interval, we deduce (similarly as before) that $\partial_x f_{\mu}^{c'} \to \partial_x f_{\mu}^{c}$. Applying Proposition~\ref{prop:approximation_process_gen} implies that $\{ X_t' \}$ converges in law to $\{ X_t \}$, the solution to the Parisi SDE. In particular, $X_1'$ converges in law to $X_1$. According to Lemma~\ref{lem:diff_boundary_c_equal}, we know that
	\begin{equation*}
		\partial_c f_{\mu}^c (1, x) = g_c \left( x + c \partial_x f_{\mu}^c (1, x) \right) + \frac{\left( \partial_x f_{\mu}^c (1, x) \right)^2}{2},
	\end{equation*}
	which is bounded and continuous. Further, one can show that the above convergence is uniform on any compact set in $\R$. As a consequence, applying continuous mapping theorem and bounded convergence theorem yields that
	\begin{align*}
		\frac{\d f_{\mu}^c (0, x)}{\d c} = \, & \lim_{c' \to c} \frac{f_{\mu}^{c'} (0, x) - f_{\mu}^c (0, x)}{c' - c} \\
		= \, & \lim_{c' \to c} \E_{X_0' = x} \left[ \left(\frac{f_{\mu}^{c'} - f_{\mu}^c}{c' - c} \right) (1, X_1') \right] \\
		= \, & \E_{X_0 = x} \left[ g_c \left( X_1 + c \partial_x f_{\mu}^c (1, X_1) \right) \right] + \frac{1}{2} \E_{X_0 = x} \left[ \partial_x f_{\mu}^c (1, X_1)^2 \right].
	\end{align*}
	Choosing $x = 0$, and noting that
	\begin{equation*}
		M_1 = \frac{1}{\gamma(1)} \partial_x f_{\mu}^c (1, X_1) + X_1 = X_1 + c \partial_x f_{\mu}^c (1, X_1)
	\end{equation*}
	(since $c = 1 / \gamma(1)$) completes the proof.
\end{proof}

We are now in position to complete the proof of \cref{thm:two_stage_strong_duality} $(c)$. Assume that 
$$
\inf_{(\mu, c) \in \FS(q)} \mathsf{F} (\mu, c) = \inf_{(\mu, c) \in \FS(q)} \mathsf{F}_c (\mu, c)
$$ 
is achieved at some $(\mu_*, c_*)$. For notational simplicity, we recast $(\mu_*, c_*)$ as $(\mu, c)$, and denote the associated $(F^*, \phi^*)$ by $(F, \phi)$. Exploiting the expressions for the first-order variation of $\mathsf{F}_c$ (\cref{prop:first_var_less_than} and \cref{lem:diff_P_c_equal}), we get
\begin{align*}
	& \E \left[ \left( \partial_x f_{\mu} (t, X_t) \right)^2 \right] - \frac{1}{\alpha} \int_{0}^{t} \gamma(s)^2 \d s = \, 0, \ \forall t \in [q, 1], \\
	& \E \left[ g_c ( M_1 ) \right] + \frac{1}{2} \E \left[ \partial_x f_{\mu} (1, X_1)^2 \right] - \frac{1}{2 \alpha} \int_{0}^{1} \gamma(t)^2 \d t = \, 0,
\end{align*}
which further implies $\E [ g_c ( M_1 ) ] = 0$. Applying Proposition~\ref{prop:1st_2nd_moment_gen_equal}, we get,
using the fact that $\gamma(t)>0$ for all $t \in [0, 1]$, 
\begin{equation*}
	\E [\phi_t^2] = \frac{1}{\alpha}, \ \text{a.e.} \ t \in [q, 1].
\end{equation*}
Then, there exists a modification of $\{ \phi_t \}$ (still denoted as $\{ \phi_t \}$) such that $\E [\phi_t^2] = 1 / \alpha$ for all $t \in [q, 1]$. Therefore, $\{ \phi_t \}_{t \in [q, 1]}$ is feasible.
As a consequence, using Eq.~\eqref{eq:Fc_DualExpression}, we deduce that
\begin{align*}
	\mathsf{F} (\mu, c) = \mathsf{F}_c (\mu, c) = \, & \E \left[ h_c \left( X_q + \frac{1}{\alpha} F(X_q) + \int_{q}^{1} \left( 1 + \phi_t \right) \d B_t \right) - \frac{\gamma(q)}{2 \alpha} \left( \frac{1}{\alpha} F(X_q)^2 - q \right) \right] \\
	= \, & \E \left[ h_c \left( M_1 \right) - \frac{\gamma(q)}{2 \alpha} \left( \frac{1}{\alpha} F(X_q)^2 - q \right) \right].
\end{align*}
Next, we claim $\E [h_c (M_1)] = \E [h (M_1)]$. To see this, note that since $g_c$ is non-negative, $\E \left[ g_c ( M_1 ) \right] = 0$ implies $ g_c ( M_1 ) = 0$ almost surely. Hence, $h_c (M_1) = h (M_1)$ almost surely (use Proposition~\ref{prop:hc_and_gc}), and consequently $\E [h_c (M_1)] = \E [h (M_1)]$. We thus conclude that
\begin{align*}
	\mathsf{F} (\mu, c) 
	= \, & \E \left[ h \left( M_1 \right) - \frac{\gamma(q)}{2 \alpha} \left( \frac{1}{\alpha} F(X_q)^2 - q \right) \right] \\
	= \, & \E \left[ h \left( X_q + \frac{1}{\alpha} F(X_q) + \int_{q}^{1} \left( 1 + \phi_t \right) \d B_t \right) - \frac{\gamma(q)}{2 \alpha} \left( \frac{1}{\alpha} F(X_q)^2 - q \right) \right] \, .
\end{align*}
It now remains to show that $F$ is feasible. By definition, we have
\begin{equation*}
	\E \left[ F(v)^2 \right] = \frac{\alpha^2}{\gamma (q)^2} \E \left[ \left( \partial_x f_{\mu} (q, X_q) \right)^2 \right] = \frac{\alpha}{\gamma (q)^2} \int_{0}^{q} \gamma (s)^2 \d s = \alpha q,
\end{equation*}
since $\gamma$ is constant on $[0, q]$. This immediately implies that
\begin{equation*}
	\mathsf{F} (\mu, c) = \, \E \left[ h \left( M_1 \right)  \right] = \E \left[ h \left( X_q + \frac{1}{\alpha} F(X_q) + \int_{q}^{1} \left( 1 + \phi_t \right) \d B_t \right) \right].
\end{equation*}
Namely, $(\mu, c)$ achieves the optimal value.

However, here we cannot directly conclude $\E [F'(v)^2] \le \alpha$ from the feasibility of $\phi$, since $\phi$ is not defined in terms of $F'$. To circumvent this issue, we use will the same standard approximation argument as before, i.e., approximating $(\mu, c)$ by a sequence $\{(\mu, c_n)\}_{n \ge 1}$ with $c_n \to c^-$. Denoting the corresponding solution to the Parisi PDE by $f_{\mu}^n$, we define (note that $\phi^n$ is defined via $\partial_x^2 f_{\mu}^n$ since $c_n < c$)
\begin{equation*}
	F_n (x) = \frac{\alpha}{\gamma_n (q)} \partial_x f_{\mu}^n (q, x), \ \phi_t^n = \frac{1}{\gamma_n (t)} \partial_x^2 f_{\mu}^n (t, X_t^n), \ \forall t \in [0, 1],
\end{equation*}
where $X^n$ is the solution to the corresponding Parisi SDE. Then, we know that $\E [F_n' (v)^2] = \alpha^2 \E [(\phi_q^n)^2]$, and $F_n \to F$ uniformly on any compact set. Further, \cref{prop:1st_2nd_moment_gen_equal} and It\^{o}'s isometry together imply that $\E [(\phi_t^n)^2] \to \E [\phi_t^2]$ as $n \to \infty$. To show that $\E [F' (v)^2] \le \alpha$, it suffices to establish that for any test function $\psi \in C_{c}^{\infty} (\R)$:
\begin{equation}\label{eq:weak_feasible_claim}
	\left\vert \E [F'(v) \psi(v)] \right\vert \le \sqrt{\alpha} \E [\psi(v)^2]^{1/2}.
\end{equation}
Using integration by parts, we know that $\E [F_n'(v) \psi(v)]$ converges to $\E [F'(v) \psi(v)]$ as $n \to \infty$. Further, Cauchy-Schwarz inequality implies that
\begin{equation*}
	\left\vert \E [F_n'(v) \psi(v)] \right\vert \le \, \E [F_n'(v)^2]^{1/2} \E [\psi(v)^2]^{1/2} = \alpha \E [(\phi_q^n)^2]^{1/2} \E [\psi(v)^2]^{1/2}.
\end{equation*}
Taking the limit $n \to \infty$, we obtain that
\begin{equation*}
	\left\vert \E [F'(v) \psi(v)] \right\vert \le \, \alpha \E [(\phi_q)^2]^{1/2} \E [\psi(v)^2]^{1/2}.
\end{equation*}
The feasibility of $F'$ then follows from the feasibility of $\phi$ and \cref{eq:weak_feasible_claim}. This completes the proof of \cref{thm:two_stage_strong_duality} $(c)$.

	\section*{Acknowledgments}
The authors would like to thank Brice Huang, Mark Sellke and Nike Sun for discussing their work \cite{huang2024preparation} prior to publication. This research was supported by the NSF through award DMS-2031883, the Simons Foundation through Award 814639 for the Collaboration on the Theoretical Foundations of Deep Learning, the NSF grant CCF-2006489 and the ONR grant N00014-18-1-2729.

\appendix

\section{Proof of Theorem \ref{thm:dual_char_prob}}
\label{sec:ProofDuality}
The ``only if'' part is obvious. To prove the ``if'' part, we first recall some classical results from analysis.
Denote by $rba(\R^m)$ the space of regular, bounded and finitely additive measures on $\R^m$. Then, \cite[Thm. IV.6]{dunford1988linear} implies that, $rba(\R^m)$ equipped with the weak-* topology is the topological dual space of $C_b (\R^m)$. By \cite[Thm. IV.20]{reed1980methods}, we deduce that
\begin{equation*}
	(rba(\R^m), w_*)^* = \, C_b (\R^m).
\end{equation*}
Further, similar to the proof of \cite[Prop. 3.4]{brezis2010functional}, one can show that the weak-* topology is locally convex, thus implying the local convexity of $rba(\R^m)$.

Now, assume by contradiction that $\mu \notin E$. Since $E$ is closed and convex in $rba(\R^m)$, the singleton $\{ \mu \}$ is compact and convex, applying the Hahn-Banach theorem \cite[Thm. 1.7]{brezis2010functional} implies that exists an $h \in C_b (\R^m)$ that strictly separates $E$ and $\{ \mu \}$. Without loss of generality we can choose $h$ such that
\begin{equation*}
	\int_{\R^m} \!h \, \d \mu > \sup_{\nu \in E} \left\{ \int_{\R^m}\! h\, \d \nu \right\},
\end{equation*}
a contradiction. Therefore, $\mu \in E$ as desired. This completes the proof.

%
%
\section{The replica calculation}\label{append:replica}

In this section we carry out calculations using the non-rigorous replica
method from statistical physics, to support Conjecture 
\ref{conj:Parisi_formula_mdim}.
We will focus on the case $m = 1$, since the case of general $m > 1$ is almost identical, but  less transparent. We
refer to \cite{montanari2024friendly} for a friendly introduction to these techniques.  The derivation presented here is quite straightforward (from a physics perspective) and generalizes the replica calculation in \cite{gyorgyi2000beyond}.

Recall the definition of the Hamiltonian
\begin{equation*}
	H_{n, d} (\ww) = \frac{1}{n} \sum_{i=1}^{n} h \left( \langle \xx_i, \ww \rangle \right)\, ,\,\,\, \ww \in \S^{d-1}\, .
\end{equation*}
Define for $\beta > 0$,
\begin{equation*}
	Z_{\beta} = \int_{\S^{d-1}} e^{n\beta H_{n,d}(\ww)}\,
	\nu_0 (\d \ww),
\end{equation*}
where $\nu_0$ is the uniform measure on $\S^{d-1}$. Assume that
\begin{equation*}
	\lim_{n, d \to \infty, \ n/d \to \alpha} \max_{\ww \in \S^{d-1}} H_{n, d} (\ww)
\end{equation*}
exists almost surely, and concentrates around its expectation (this is true if $h$ is Lipschitz), then we would like to compute
\begin{align*}
	\lim_{n \to \infty} \E \left[ \max_{\ww \in \S^{d-1}} H_{n, d} (\ww) \right] = \lim_{n \to \infty} \E \left[ \lim_{\beta \to \infty} \frac{1}{n \beta} \log Z_{\beta} \right].
\end{align*}
Here, the limit $n \to \infty$ should be understood as $n, d \to \infty$ simultaneously with $n/d \to \alpha$. 
Within the replica method,
we interchange expectation and limit arbitrarily. It then suffices to compute the quantity
\begin{equation*}
	\lim_{\beta \to \infty} \lim_{n \to \infty} \frac{1}{n \beta} \E \left[ \log Z_{\beta} \right] = \lim_{\beta \to \infty} \lim_{n \to \infty} \lim_{k \to 0^+} \frac{1}{n \beta k} \log \E \left[ Z_{\beta}^k \right],
\end{equation*}
where we use the identity
\begin{equation*}
	\E \left[ \log Z \right] = \lim_{k \to 0^+} \frac{1}{k} \log \E \left[ Z^k \right].
\end{equation*}

While the above interchange of limits is not justified in the present derivation, it is not the most problematic step in the replica calculation.
Indeed, the critical step is to first consider $k$ as an integer, and then
extrapolate to non-integer values of $k$. For $k \in \mathbb{N}$, we have
\begin{align*}
	\E \left[ Z_{\beta}^k \right] = \, & \E \left[ \int_{(\S^{d-1})^k} \exp \left( \beta \cdot \sum_{j=1}^{k} \sum_{i=1}^{n} h \left( \langle \ww_j, \xx_i \rangle \right) \right) \cdot \prod_{j=1}^{k} \nu_0(\d \ww_j) \right] \\
	= \, & \int_{(\S^{d-1})^k} \E \left[ \exp \left( \beta \cdot \sum_{j=1}^{k} \sum_{i=1}^{n} h \left( \langle \ww_j, \xx_i \rangle \right) \right) \right] \cdot \prod_{j=1}^{k} \nu_0(\d \ww_j) \\
	= \, & \int_{(\S^{d-1})^k} \E \left[ \exp \left( \beta \cdot \sum_{j=1}^{k} h \left( \langle \ww_j, \xx \rangle \right) \right) \right]^n \cdot \prod_{j=1}^{k} \nu_0 (\d \ww_j).
\end{align*}
Denoting by $Q$ the overlap matrix of the $\ww_j$'s, namely $Q_{ij} = \langle \ww_i, \ww_j \rangle$ for $1 \le i, j \le k$, then we have for $\xx \sim \normal (\bzero, \id_d)$,
\begin{equation*}
	\E \left[ \exp \left( \beta \cdot \sum_{j=1}^{k} h \left( \langle \ww_j, \xx \rangle \right) \right) \right] = \E_{G \sim \normal (0, Q)} \left[ \exp \left( \beta \cdot \sum_{j=1}^{k} h(G_j) \right) \right].
\end{equation*}
For future convenience, we denote the above quantity as $f_{\beta, h} (Q)$, i.e.,
\begin{equation*}
	f_{\beta, h} (Q) = \E_{G \sim \normal (0, Q)} \left[ \exp \left( \beta \cdot \sum_{j=1}^{k} h(G_j) \right) \right],
\end{equation*}
it then follows that
\begin{align*}
	\E \left[ Z_{\beta}^k \right] = \, & \int_{(\S^{d-1})^k} f_{\beta, h} (Q)^n \cdot \prod_{j=1}^{k} \nu_0 (\d \ww_j) \\
	= \, & \int_{\S_+^k (1)} f_{\beta, h} (Q)^n \exp \left( d I_d (Q) \right) \nu (\d Q),
\end{align*}
where $\S_+^k (1)$ denotes the space of all $k \times k$ positive semidefinite matrices with all ones on the diagonal, and \modif{$\nu$ represents the uniform probability measure on this space, i.e.,
	\begin{equation*}
		\nu(\d Q) = \frac{1}{\int_{\S_+^k (1)} \prod_{1 \le i < j \le k} \d Q_{ij}} \cdot \bone \{Q \in \S_+^k (1) \} \cdot \prod_{1 \le i < j \le k} \d Q_{ij}.
\end{equation*}}
Moreover, we have for fixed $k$,
\begin{equation*}
	\lim_{d \to \infty} I_d (Q) = \frac{1}{2} \log \det Q,
\end{equation*}
thus leading to
\begin{align*}
	\lim_{n \to \infty} \frac{1}{n} \log \E \left[ Z_{\beta}^k \right] = \, & \max_{Q \in \S_+^k (1)} \left\{ \log f_{\beta, h} (Q) + \frac{1}{2 \alpha} \log \det Q \right\} \\
	= \, & \max_{Q \in \S_+^k (1)} \left\{ \log \E_{G \sim \normal (0, Q)} \left[ \exp \left( \beta \cdot \sum_{j=1}^{k} h(G_j) \right) \right] + \frac{1}{2 \alpha} \log \det Q \right\} \\
	:= \, & S_{\beta, h} (\alpha, k).
\end{align*}
Assume again that we can interchange the limits arbitrarily, then we get that
\begin{align*}
	\lim_{\beta \to \infty} \lim_{n \to \infty} \frac{1}{n \beta} \E \left[ \log Z_{\beta} \right] = \, & \lim_{\beta \to \infty} \lim_{n \to \infty} \lim_{k \to 0^+} \frac{1}{n \beta k} \log \E \left[ Z_{\beta}^k \right] \\
	= \, & \lim_{\beta \to \infty} \frac{1}{\beta} \lim_{k \to 0^+} \frac{1}{k} \lim_{n \to \infty} \frac{1}{n} \log \E \left[ Z_{\beta}^k \right] \\
	= \, & \lim_{\beta \to \infty} \frac{1}{\beta} \lim_{k \to 0^+} \frac{1}{k} S_{\beta, h} (\alpha, k).
\end{align*}
To compute this limit, we resort to the full RSB (full replica symmetry breaking) ansatz described in Section 3 of \cite{gyorgyi2000beyond}. Following their calculation, the limiting free energy can be expressed as the extreme value of a variational problem. To be specific, we have
\begin{align*}
	&\frac{1}{\beta} \lim_{k \to 0^+} \frac{1}{k} S_{\beta, h} (\alpha, k) =
	\inf_{y \in \mathscr{U} [0, 1]} \sA(y,\beta)\, ,\\
	&\sA(y,\beta) := f_y (0, 0) + \frac{1}{2 \alpha \beta} \int_{0}^{1} \left( \frac{1}{D_y(t)} - \frac{1}{1 - t} \right) \d t ,
\end{align*}
where $\mathscr{U} [0, 1]$ is the space of all non-descreasing function $y: [0, 1] \to [0, 1]$,
\begin{equation*}
	D_y(t) = \int_{t}^{1} y(s) \d s,
\end{equation*}
and $f_y (t, x)$ satisfies the PDE:
\begin{equation}\label{eq:Parisi_beta}
	\begin{split}
		\partial_t f_y (t, x) + \frac{1}{2} \beta y(t) \left( \partial_x f_y (t, x) \right)^2 + \frac{1}{2} \partial_x^2 f_y (t, x) = \, & 0, \\
		f_y (1, x) = \, & h(x).
	\end{split}
\end{equation}
The lemma below gives the zero-temperature limit ($\beta \to \infty$) of 
the variational functional $\sA(y,\beta)$, along specific 
sequences of $y_{\beta}$.
\begin{lem}\label{lem:Parisi_final}
	Let $c > 0$, and $\mu(t): [0, 1) \to \R_{\ge 0}$ be a non-decreasing function with $\int_{0}^1\mu(t)\, \d t<\infty$. Further, assume that $y(t)$ has the following form:
	\begin{equation*}
		y_{\beta}(t) = \frac{\mu(t)}{\beta} \bone_{t < 1 - \frac{c}{\beta}} + \bone_{t \ge 1 - \frac{c}{\beta}}.
	\end{equation*}
	Then, we have
	\begin{align*}
		& \lim_{\beta \to \infty} 
		\sA(y_{\beta},\beta) = \sF_1(\mu,c)\, ,\\
		& \sF_1(\mu,c) := f_\mu (0, 0) + \frac{1}{2 \alpha} \int_{0}^{1} \frac{\d t}{c + \int_{t}^{1} \mu(s) \d s},
	\end{align*}
	where $f_\mu$ solves the terminal-value problem:
	\begin{equation}\label{eq:Parisi_limit}
		\begin{split}
			&\partial_t f_\mu (t, x) + \frac{1}{2} \mu(t) \left( \partial_x f_\mu (t, x) \right)^2 + \frac{1}{2} \partial_x^2 f_\mu (t, x) = \,  0, \\
			&f_\mu (1, x) = \,  \sup_{u \in \R} \left\{ h(x+u) - \frac{u^2}{2 c} \right\}.
		\end{split}
	\end{equation}
\end{lem}

\begin{proof}
	Defining $t_{\beta} = 1 - c / \beta$, Eq.~\eqref{eq:Parisi_beta} reduces to
	\begin{align*}
		\partial_t f_y (t, x) + \frac{1}{2} \mu(t) \left( \partial_x f_y (t, x) \right)^2 + \frac{1}{2} \partial_x^2 f_y (t, x) = \, & 0, \ t \in [0, t_{\beta}), \\
		\partial_t f_y (t, x) + \frac{1}{2} \beta \left( \partial_x f_y (t, x) \right)^2 + \frac{1}{2} \partial_x^2 f_y (t, x) = \, & 0, \ t \in [t_{\beta}, 1), \\
		f_y (1, x) = \, & h(x).
	\end{align*}
	Using Cole-Hopf transform, we know that
	\begin{align*}
		f_y (t_{\beta}, x) = \, & \frac{1}{\beta} \log \E_{G \sim \normal(0, 1)} \left[ \exp \left( \beta \cdot h \left( x + \sqrt{\frac{c}{\beta}} G \right) \right) \right] \\
		=\, & \frac{1}{\beta} \log \left( \frac{1}{\sqrt{2 \pi}} \int_{\R} \exp \left( \beta \cdot h \left( x + \sqrt{\frac{c}{\beta}} z \right) - \frac{z^2}{2} \right) \d z \right) \\
		=\, & \frac{1}{\beta} \log \left( \frac{\sqrt{\beta}}{\sqrt{2 \pi c}} \int_{\R} \exp \left( \beta \cdot h \left( x + u \right) - \frac{\beta u^2}{2 c} \right) \d u \right),
	\end{align*}
	which converges to $\sup_{u \in \R} \left\{ h(x+u) - u^2 / (2 c) \right\}$ as $\beta \to \infty$, uniformly over
	compact sets. Moreover, since $t_{\beta} \to 1$, we deduce that $f_y$ converges to $f_\mu$ where $f_\mu$ solves Eq.~\eqref{eq:Parisi_limit}. As a consequence, $f_y (0, 0) \to f_{\mu} (0, 0)$. To compute the limit of the second term, we note that $D(t) = 1 - t$ if $t \ge t_{\beta}$. Therefore,
	\begin{align*}
		& \frac{1}{2 \alpha \beta} \int_{0}^{1} \left( \frac{1}{D(t)} - \frac{1}{1 - t} \right) \d t = \, \frac{1}{2 \alpha \beta} \int_{0}^{t_{\beta}} \left( \frac{1}{D(t)} - \frac{1}{1 - t} \right) \d t \\
		= \, & \frac{1}{2 \alpha} \int_{0}^{t_{\beta}} \frac{\d t}{c + \int_{t}^{t_{\beta}} \mu(s) \d s} + \frac{1}{2 \alpha \beta} \log \left( 1 - t_{\beta} \right) \\
		= \, & \frac{1}{2 \alpha} \int_{0}^{t_{\beta}} \frac{\d t}{c + \int_{t}^{t_{\beta}} \mu(s) \d s} + \frac{1}{2 \alpha \beta} \log \left( \frac{c}{\beta} \right) \\
		\to \, & \frac{1}{2 \alpha} \int_{0}^{1} \frac{\d t}{c + \int_{t}^{1} \mu(s) \d s} \ \text{as} \ \beta \to \infty.
	\end{align*}
	This completes the proof.
\end{proof}
Defining
\begin{equation*}
	\mathscr{U} = \left\{ \mu: [0, 1) \to \R_{\ge 0}: \ \mu \ \text{non-decreasing}, \ \int_{0}^{1} \mu(t) \d t < \infty \right\},
\end{equation*}
we note that the function $\mathsf{F}_1 : \mathscr{U}\times \R_{>0}\to \R$ defined in the last lemma 
coincides with the one of Conjecture \ref{conj:Parisi_formula_mdim} and Remark \ref{rmk:ReplicaM1}. This establishes the replica prediction.

\section{Technical lemmas}

The proposition below presents some analytical properties of $V_{\gamma} (t, z)$ (defined in Eq.~\eqref{val_fct_gen}) as a function of $\gamma$ with fixed $t, z$.
\begin{prop}[Properties of $V_{\gamma}$]\label{prop:VAL_continuity_gen}
	Fix $t$ and $z$, then $\gamma \mapsto V_{\gamma}(t,z)$ is convex and lower semicontinuous  with respect to the $L^{\infty}$-norm. Further, let $\gamma_0 \in L_{\infty}^{+} [0, 1]$ be such that $\inf_{t \in [0, 1]} \gamma_0 (t) > 0$, then $V_{\gamma}$ is continuous at $\gamma_0$ with respect to the $L^{\infty}$-norm.
\end{prop}

\begin{proof}
	The convexity and lower semicontinuity follow directly from the definition of $V_{\gamma}$, since it is the pointwise supremum of linear functionals that are continuous in $L^{\infty}$-norm.
	
	To prove the second part, let $\gamma_n \stackrel{L^{\infty}}{\to} \gamma_0$, then we know that $\inf_{t \in [0, 1]} \gamma_n (t) \ge \inf_{t \in [0, 1]} \gamma_0 (t)/2 > 0$ for sufficiently large $n$. As a consequence, there exists a constant $C_0 > 0$ such that
	\begin{equation*}
		V_{\gamma} (t, z) = \sup_{\int_{t}^{1} \E[\phi_s^2] \d s \le C_0} \E \left[ h \left( z + \int_{t}^{1} (1 + \phi_s) \d B_s \right) - \frac{1}{2} \int_{t}^{1} \gamma(s) \phi_s^2 \d s \right] + \frac{1}{2 \alpha} \int_{t}^{1} \gamma(s) \d s
	\end{equation*}
	for $\gamma = \gamma_n$ or $\gamma_0$. We thus obtain that
	\begin{equation*}
		\left\vert V_{\gamma_n} (t, z) - V_{\gamma_0} (t, z) \right\vert \le \frac{1}{2} \left( C_0 + \frac{1}{\alpha} \right) \norm{\gamma_n - \gamma_0}_{L^{\infty} [0, 1]}.
	\end{equation*}
	This completes the proof. In fact, we proved an even stronger statement: $V_{\gamma}$ is locally Lipschitz at the interior of $L_{\infty}^{+} [0, 1]$.
\end{proof}

The following lemma gives a variational representation for the concave envelope of a function.
\begin{lem}\label{lem:var_rep_conc}
	Let $h \in C (\R)$ be upper bounded, denote by $\operatorname{conc} h$ the concave envelope of $h$, then
	\begin{equation}\label{eq:variational_conc_env_gen}
		\operatorname{conc} h(z) = \sup_{U \in L^2(\Omega), \ \E[U] = 0} \E \left[ h(z + U) \right].
	\end{equation} 
	As a consequence, we have
	\begin{equation*}
		\operatorname{conc} \left( h (z) - \frac{t}{2}  z^2 \right) + \frac{t}{2} z^2 = \sup_{U \in L^2(\Omega), \ \E[U] = 0} \E \left[ h (z+U) - \frac{t}{2} U^2 \right].
	\end{equation*} 
\end{lem}
\begin{proof}
	Since $h$ is upper bounded, we know that the right hand side of Eq.~\eqref{eq:variational_conc_env_gen} is well-defined and upper bounded. Let us denote
	\begin{equation*}
		g(z) = \sup_{U \in L^2(\Omega), \ \E[U] = 0} \E \left[ h(z + U) \right] = \sup_{U \in L^2(\Omega), \ \E[U] = z} \E \left[ h(U) \right],
	\end{equation*}
	then obviously we have $g(z) \ge h(z)$. Next we show that $g$ is concave. Fix $z_1, z_2 \in \R$ and $\alpha \in [0, 1]$, $\forall \veps > 0$ there exists $U_1, U_2 \in L^2 (\Omega)$ such that $\E[U_1] = z_1$, $\E [U_2] = z_2$, and
	\begin{equation*}
		g(z_1) \le \E [h(U_1)] + \veps, \ g(z_2) \le \E [h(U_2)] + \veps.
	\end{equation*}
	Now we define a new random variable $U \in L^2 (\Omega)$ by requiring
	\begin{equation*}
		\P (U = U_1) = \alpha, \ \P (U = U_2) = 1 - \alpha,
	\end{equation*}
	it follows that $\E [U] = \alpha z_1 + (1 - \alpha) z_2$, thus leading to
	\begin{align*}
		g \left( \alpha z_1 + (1 - \alpha) z_2 \right) \ge \, & \E [h(U)] = \alpha \E [h(U_1)] + (1 - \alpha) \E [h(U_2)] \\
		\ge \, & \alpha g(z_1) + (1 - \alpha) g(z_2) - \veps.
	\end{align*}
	Since $\veps > 0$ can be arbitrary, we finally deduce that $g \left( \alpha z_1 + (1 - \alpha) z_2 \right) \ge \alpha g(z_1) + (1 - \alpha) g(z_2)$. Therefore, $g$ is concave. It finally remains to show that $g$ is the smallest concave function that dominates $h$. To this end, assume $f \ge h$ is concave, then for any $U \in L^2(\Omega)$ with $\E[U] = z$, we deduce from Jensen's inequality:
	\begin{equation*}
		\E[h(U)] \le \E[f(U)] \le f \left( \E[U] \right) = f(z).
	\end{equation*}
	Taking supremum over all such random variable $U$ yields that $g(z) \le f(z)$. We have thus established that $g = \operatorname{conc} h$. The ``as a consequence'' part follows by direct calculation.
\end{proof}

We collect below a few useful properties of the functions $h_c$ and $f_{\mu}^c (1, \cdot)$ defined in \cref{sec:solve_Parisi_PDE}:
\begin{equation*}
	h_c (z) = \operatorname{conc} \left( h(z) - \frac{z^2}{2 c} \right) + \frac{z^2}{2 c}, \quad f_{\mu}^c (1, x) = \, \sup_{u \in \R} \left\{ h_c(x+u) - \frac{u^2}{2 c} \right\}.
\end{equation*}

\begin{lem}\label{lem:diff_boundary_c_equal}
	For any $x \in \R$ and $c > 0$, define $g_c (x) = (\partial/ \partial c) h_c (x)$ (existence is guaranteed by monotonicity and convexity with respect to $1/c$). Then, we have
	\begin{equation*}
		\frac{\d}{\d c} f_{\mu}^c (1, x) = g_c \left( x + c \partial_x f_{\mu}^c (1, x) \right) + \frac{\left( \partial_x f_{\mu}^c (1, x) \right)^2}{2}.
	\end{equation*}
\end{lem}
\begin{proof}
	By definition of $f_{\mu}^c (1, x)$ and the envelope theorem, we obtain
	\begin{align*}
		\frac{\d}{\d c} f_{\mu}^c (1, x) = \, \frac{\d}{\d c} \left( \sup_{u \in \R} \left\{ h_c(x+u) - \frac{u^2}{2 c} \right\} \right) = g_c (x + u(c, x)) + \frac{u(c, x)^2}{2 c^2},
	\end{align*}
	where
	\begin{equation*}
		u(c, x) \in \argmax_{u \in \R} \left\{ h_c (x+u) - \frac{u^2}{2 c} \right\}.
	\end{equation*}
	The above optimization problem is concave, and its first-order condition reads
	\begin{equation*}
		h_c' (x + u(c, x)) = \frac{u(c, x)}{c}.
	\end{equation*}
	Further, by duality, we know that one can take $u (c, x) = c \partial_x f_{\mu}^c (1, x)$. Now since $h_c'$ is uniformly bounded, it follows that as long as $c$ is bounded away from $0$, $\frac{\d}{\d c} f_{\mu}^c (1, x)$ is bounded.
	This completes the proof.
\end{proof}

\begin{prop}\label{prop:hc_and_gc}
	Fix $x \in \R$, then $h_c (x) = h(x)$ if and only if $g_c (x) = 0$.
\end{prop}
\begin{proof}
	By definition, we know that for any $c_1 \le c_2$, $h_{c_1} (x) \le h_{c_2} (x)$. Further, $h(x) = h_0 (x) = \lim_{c \to 0^+} h_c (x)$. We first prove the ``only if'' part. Assume $h_c (x) = h(x)$, then by monotonicity, $h_{c'} (x) = h_c(x)$ for any $c' \le c$, which implies $g_c (x) = 0$ since $h_c (x)$ is differentiable in $c$. To show the ``if'' part, define for $t > 0$:
	\begin{equation*}
		\varphi (t) = h_{1/t} (x) = \operatorname{conc} \left( h(x) - \frac{t}{2} x^2 \right) + \frac{t}{2} x^2.
	\end{equation*}
	Then, we know that $\varphi (1/c) = \varphi (+ \infty)$. Further, Lemma~\ref{lem:var_rep_conc} implies that
	\begin{equation}\label{eq:varphi_t}
		\varphi (t) = \sup_{U \in L^2(\Omega), \ \E[U] = 0} \E \left[ h (x+U) - \frac{t}{2} U^2 \right]
	\end{equation}
	is convex and continuous in $t$. Now since $\varphi' (1/c) = g_c (x) = 0$, we know that $\varphi' (t) \ge 0$ for all $t \ge 1/c$, which implies that $\varphi$ is increasing on $[1/c, +\infty]$. However, we know that $\varphi$ is non-increasing, by 
	\eqref{eq:varphi_t}. Therefore, $\varphi$ must be constant on $[1/c, +\infty]$, which implies $\varphi(1/c) = \varphi (+\infty)$, i.e., $h_c (x) = h(x)$. This concludes the proof.
\end{proof}

The proposition below is crucial to a number of approximation arguments in our proof.
\begin{prop}\label{prop:approximation_process_gen}
	Let $\{ g_n \}_{n=1}^{\infty}$ and $g$ be measurable functions defined on $[0, 1] \times \R$, satisfying
	\begin{itemize}
		\item [$(a)$] $\forall t \in [0, 1]$, $g_n (t,\, \cdot\,)$ and $g(t,\,\cdot\,):\R\to\R$ are continuous functions.
		\item [$(b)$] There exists a function $m \in L^1 [0, 1]$, such that $\vert g_n (t, x) \vert \le m(t)$ and $\vert g(t, x) \vert \le m(t)$ for all $(t, x) \in [0, 1] \times \R$.
		\item [(c)] $g_n (t, x) \to g(t, x)$ for all $(t, x) \in [0, 1] \times \R$.
	\end{itemize}
	Assume $\{ x_n \}_{n=1}^{\infty}$ is a sequence of real numbers converging to $x \in \R$, consider the SDEs (strong existence and pathwise uniqueness guaranteed by Proposition 1.10 of \cite{cherny2005singular})
	\begin{align*}
		\d X_t^n = \, & g_n (t, X_t^n) \d t + \d B_t, \,\, X_0^n = x_n, \\
		\d X_t = \, & g (t, X_t) \d t + \d B_t, \,\,\, X_0 = x.
	\end{align*}
	Then, we have $\operatorname{Law} ((X_t^n )_{t \in [0, 1]})$ weakly converges to $\operatorname{Law} (( X_t )_{t \in [0, 1]})$ as $n \to \infty$, where $( X_t^n )_{t \in [0, 1]}$ and $( X_t ))_{t \in [0, 1]}$ are viewed as random elements in $C[0, 1]$.
\end{prop}

\begin{proof}
	We first show that $\{ X^n \}_{n=1}^{\infty}$ is a tight sequence of $C[0, 1]$-valued random variables. First, it is obvious that $X_0^n = x_n$ is a tight sequence of random variables. Next, note that $\forall s \le t \in [0, 1]$,
	\begin{align*}
		\left\vert X_t^n - X_s^n \right\vert = \, & \left\vert \int_{s}^{t} g_{n} (u, X_u^n) \d u + B_t - B_s \right\vert \le \int_{s}^{t} m(u) \d u + \left\vert B_t - B_s \right\vert,
	\end{align*}
	which implies that for any $\epsilon>0$ and $\eta>0$, there is some $\delta>0$ such that for all large enough $n$ (depending on $\epsilon$ and $\eta$),
	\begin{equation*}
		\mathbb{P}\left(\omega_{X^n}(\delta)>\eta\right) \leq \epsilon,
	\end{equation*}
	where
	\begin{equation*}
		\omega_f(\delta):=\sup \{|f(s)-f(t)|: 0 \leq s, t \leq 1,|s-t| \leq \delta\}
	\end{equation*}
	defines the modulus of continuity for any $f \in C[0, 1]$. This proves that $\{ X^n \}_{n=1}^{\infty}$ is tight (cf. \cite{mitoma1983tightness}). As a consequence, any subsequence of $\{ X^{n_k} \}_{k=1}^{\infty}$ has a further subsequence that converges in distribution. It thus suffices to show that any such weak limit must be equal to $\operatorname{Law} (\{ X_t \}_{t \in [0, 1]})$. For simplicity, we still denote this subsequence by $\{ X^n \}_{n=1}^{\infty}$. According to Skorokhod's representation theorem, we may assume without loss of generality that each $X^n$ satisfies
	\begin{equation*}
		\d X_t^n = \, g_n (t, X_t^n) \d t + \d B_t^n, \ X_0^n = x_n,
	\end{equation*}
	where $B^n$ can possibly be different standard Brownian motions, and $X^n$ converges to some $Y \in C[0, 1]$ almost surely. Exploiting the SDE observed by $X^n$, we get that
	\begin{equation*}
		X_t^n = x_n + \int_{0}^{t} g_n (s, X_s^n) \d s + B_t^n, \ \forall t \in [0, 1].
	\end{equation*}
	By our assumption, on the event $X^n \to Y$, we have $g_n (t, X_t^n) \to g(t, Y_t)$ for all $t \in [0, 1]$ as $n \to \infty$, since $g_n$ converges to $g$ and $g$ is continuous. Using dominated convergence theorem ($g_n$ and $g$ are dominated by $m$), it follows that
	\begin{equation*}
		x_n + \int_{0}^{t} g_n (s, X_s^n) \d s \to x + \int_{0}^{t} g (s, Y_s) \d s, \ \forall t \in [0, 1].
	\end{equation*}
	As a consequence, $\{ B^n \}_{n=1}^{\infty}$ converges to some $B \in C[0, 1]$ almost surely. Of course, $B$ is a standard Brownian motion, which finally leads to the SDE
	\begin{equation*}
		Y_t = x + \int_{0}^{t} g (s, Y_s) \d s + B_t, \ \forall t \in [0, 1].
	\end{equation*}
	By uniqueness, we must have $\operatorname{Law} (Y) = \operatorname{Law} (X)$. This concludes the proof.
\end{proof}

	\newpage
	
	\bibliographystyle{alpha}
	\bibliography{small_nn}

\end{document}